\UseRawInputEncoding

\documentclass[11pt]{amsart}
    \usepackage{amsmath,amssymb,amsthm,amsfonts,amsrefs}
    \usepackage{geometry}
    \usepackage{graphicx}
    \usepackage[all]{xy}
    \usepackage[colorlinks, linkcolor=blue, anchorcolor=blue,
            citecolor=blue]{hyperref}

    \usepackage{fancyhdr}
    \usepackage{mathrsfs}
    \newtheorem{definition}{Definition}[section]
    \newtheorem{theorem}{Theorem}[section]
    \newtheorem{proposition}[theorem]{Proposition}
    \newtheorem{lemma}[theorem]{Lemma}
    \newtheorem{corollary}[theorem]{Corollary}
    \newtheorem{example}{Example}[section]
    \newtheorem{remark}[example]{Remark}

\begin{document}

\title{Lagrangian Cobordism Functor in Microlocal Sheaf Theory I}
\date{}
\author{Wenyuan Li}
\address{Department of Mathematics, Northwestern University.}
\email{wenyuanli2023@u.northwestern.edu}
\maketitle

\begin{abstract}
    Let $\Lambda_\pm$ be Legendrian submanifolds in the cosphere bundle $T^{*,\infty}M$. Given a Lagrangian cobordism $L$ of Legendrians from $\Lambda_-$ to $\Lambda_+$, we construct a functor $\Phi_L^*: Sh^c_{\Lambda_+}(M) \rightarrow Sh^c_{\Lambda_-}(M) \otimes_{C_{-*}(\Omega_*\Lambda_-)} C_{-*}(\Omega_*L)$ between sheaf categories of compact objects with singular support on $\Lambda_\pm$ and its right adjoint on sheaf categories of proper objects, using Nadler-Shende's work. This gives a sheaf theory description analogous to the Lagrangian cobordism map on Legendrian contact homologies and the right adjoint on their unital augmentation categories. We also deduce some long exact sequences and new obstructions to Lagrangian cobordisms between high dimensional Legendrian submanifolds.
\end{abstract}

\section{Introduction}

\subsection{Motivation and Background}
    A contact manifold is a $(2n+1)$-dimensional manifold $Y$ together with a maximally nonintegrable hyperplane distribution $\xi \subset TY$, and a Legendrian submanifold is an $n$-dimensional submanifold $\Lambda \subset Y$ such that $\xi|_\Lambda \subset T\Lambda$. Assume that $\xi \subset TY$ is defined by the kernel of a 1-form $\alpha \in \Omega^1(Y)$ called the contact form (this is equivalent to saying that the contact structure is coorientable).
    %Note that by fixing a coorientation of $\xi$, $\alpha$ is only determined up to a multiple of a positive function $f \in C^\infty(Y)$.
    Given a contact form $\alpha$, the Reeb vector field $R_\alpha$ is the vector field such that
    $$\alpha(R_\alpha) = 1, \,\,\, \iota(R_\alpha)d\alpha = 0.$$

    Contact manifolds (resp.~Legendrian submanifolds) naturally arise as boundaries of exact symplectic manifolds $(X, d\lambda)$ (resp.~exact Lagrangian submanifolds $L \subset X$ where $\lambda|_L = df_L$) from the point of view of symplectic field theory \cite{SFT}. In particular, in the symplectization $(Y \times \mathbb{R}_r, d(e^r\alpha))$ of the contact manifold $(Y, \ker\alpha)$, following \cite[Section 2.8]{SFT}, Chantraine \cite{Chantraine} and Ekholm \cite{Ekcobordism}, for instance, considered the category of Lagrangian cobordisms.

\begin{definition}
    The category of Lagrangian cobordisms $\mathrm{Cob}(Y)$, has objects being Legendrian submanifolds $\Lambda \subset Y$ and morphisms $Hom(\Lambda_-, \Lambda_+)$ being exact Lagrangian submanifolds $L \subset (Y \times \mathbb{R}_r, d(e^r\alpha))$ with $e^r\alpha|_L = df_L$ such that
    $$L \cap (Y \times (-\infty, -r)) = \Lambda_- \times (-\infty, -r), \,\, L \cap (Y \times (r, +\infty)) = \Lambda_+ \times (r, +\infty).$$
    for some $r > 0$, and the primitive $f_L$ is a constant on $\Lambda_- \times (-\infty, -r)$ and $\Lambda_+ \times (r, +\infty)$. We call such an $L$ a Lagrangian cobordism from $\Lambda_-$ to $\Lambda_+$.
\end{definition}
\begin{remark}
    Compositions in $\mathrm{Cob}(Y)$ are defined by concatenating Lagrangian cobordisms. We will denote the concatenation of $L_0 \in Hom(\Lambda_0, \Lambda_1)$ and $L_1 \in Hom(\Lambda_1, \Lambda_2)$ by $L_0 \cup L_1$.
\end{remark}

    Under certain conditions on $(Y, \ker\alpha)$ (for example, when $Y$ has no closed Reeb orbits or when it has an exact symplectic filling) previous works in this field considered a dg algebra called Legendrian contact homology/Chekanov-Eliashberg dg algebra $\mathcal{A}(\Lambda)$ associated to a Legendrian submanifold $\Lambda$ generated by Reeb trajectories starting and ending on $\Lambda$ \cite{Chekdga,EESPtimeR}. We consider the version $\mathcal{A}_{C_{-*}(\Omega_*\Lambda)}(\Lambda)$ that is a dg algebra over the dg algebra $C_{-*}(\Omega_*\Lambda)$ where $\Omega_*\Lambda$ is the based loop space of $\Lambda$ \cite{EkholmLekili}. Following \cite{Ekcobordism,EHK}, a Lagrangian cobordism $L$ from $\Lambda_-$ to $\Lambda_+$ is expected to induce a homomorphism
    $$\Phi_L^*: \, \mathcal{A}_{C_{-*}(\Omega_*\Lambda_+)}(\Lambda_+) \rightarrow \mathcal{A}_{C_{-*}(\Omega_*\Lambda_-)}(\Lambda_-) \otimes_{C_{-*}(\Omega_*\Lambda_-)} C_{-*}(\Omega_*L).$$
    The representations of $\mathcal{A}_{C_{-*}(\Omega_*\Lambda)}(\Lambda)$ over $\Bbbk$ are called augmentations. Given an augmentation $\epsilon_-: \mathcal{A}_{C_{-*}(\Omega_*\Lambda_-)}(\Lambda_-) \rightarrow \Bbbk$, its restriction
    $$\epsilon_-|_{C_{-*}(\Omega_*\Lambda_-)}: C_{-*}(\Omega_*\Lambda_-) \rightarrow \Bbbk$$
    defines a rank 1 local system $\delta_{\Lambda_-} \in Hom(C_0(\Omega_*\Lambda_-); \Bbbk) \cong H^1(\Lambda_-; \Bbbk^\times)$. For any rank 1 local system $\delta_L \in \mathrm{Hom}(C_0(\Omega_*L); \Bbbk) \cong H^1(L; \Bbbk^\times)$ that restricts to $\delta_{\Lambda_-}$ on $\Lambda_-$, we are able to define the induced augmentation on $\Lambda_+$
    $$\epsilon_+ = \Phi_L(\epsilon_-, \delta_L): \mathcal{A}_{C_{-*}(\Omega_*\Lambda_+)}(\Lambda_+) \xrightarrow{\Phi_L^*} \mathcal{A}_{C_{-*}(\Omega_*\Lambda_-)}(\Lambda_-) \otimes_{C_{-*}(\Omega_*\Lambda_-)}C_{-*}(\Omega_*L) \xrightarrow{(\epsilon_-, \delta_L)} \Bbbk$$
    (see \cite{PanTorus} for the case of Legendrian knots).

    For augmentations of $\mathcal{A}(\Lambda)$ and respectively $\mathcal{A}_{C_{-*}(\Omega_*\Lambda)}(\Lambda)$, Bourgeois-Chantraine \cite{BCAug} defined a non-unital $\mathcal{A}_\infty$-category $\mathcal{A}ug_-(\Lambda)$, while Ng-Rutherford-Sivek-Shende-Zaslow \cite{AugSheaf} defined a (strictly) unital $\mathcal{A}_\infty$-category $\mathcal{A}ug_+(\Lambda)$ for Legendrian knots in $\mathbb{R}^3_\text{std}$\footnote{The $\pm$ signs come from the fact that $\mathcal{A}ug_-(\Lambda)$ can be defined using small negative Reeb pushoffs of $\Lambda$, while $\mathcal{A}ug_+(\Lambda)$ is defined using positive pushoffs of $\Lambda$. Following \cite[Section 1.2]{EkholmLekili}, $\mathcal{A}ug_+(\Lambda)$ should be understood as augmentations of $\mathcal{A}(\Lambda)$ while $\mathcal{A}ug_+(\Lambda)$ should be understood as augmentations of $\mathcal{A}_{C_{-*}(\Omega_*\Lambda)}(\Lambda)$.}. A Lagrangian cobordism $L$ from $\Lambda_-$ to $\Lambda_+$ is expected to induce a functor between the corresponding augmentation categories
    $$\Phi_L: \, \mathcal{A}ug_+(\Lambda_-) \times_{Loc^1(\Lambda_-)} Loc^1(L) \rightarrow \mathcal{A}ug_+(\Lambda_+),$$
    where $Loc^1(-)$ stands for rank 1 local systems.

    In comparison, in recent years microlocal sheaf theory has also shown to be a powerful tool in symplectic and contact geometry \cite{NadZas,Nad,Tamarkin1,Shendeconormal,STZ,STWZ,CasalsZas}. The category of proper sheaves with singular support on $\Lambda$ is understood to be certain infinitesimal Fukaya category of Lagrangians asymptotic to $\Lambda$ considered by Nadler-Zaslow \cite{NadZas}, and in $\mathbb{R}^3_\text{std}$ Ng-Rutherford-Sivek-Shende-Zaslow proved that the unital augmentation category $\mathcal{A}ug_+(\Lambda)$ of the Chekanov-Eliashberg dg algebra is a sheaf category consisting of microlocal rank 1 (i.e.~simple) objects \cite{AugSheaf}
    $$Sh^\text{s}_\Lambda(\mathbb{R}^2)_0 \simeq \mathcal{A}ug_+(\Lambda)$$
    (in higher dimensions, some results have also been obtained \cite{CasalsMurphydga,AugSheafknot,AugSheafsurface}).

    At the same time, for Weinstein manifolds $X$ with skeleton $\mathfrak{c}_X$ and a Legendrian submanifold in the contact boundary $\Lambda \subset \partial_\infty X$, Ganatra-Pardon-Shende \cite{GPS3} showed the equivalence between the microlocal sheaf category on the Lagrangian skeleton and the partially wrapped Fukaya category (defined in \cite{Sylvan} and \cite{GPS1})
    $$\mu Sh^c_{\mathfrak{c}_X \cup \Lambda \times \mathbb{R}}(\mathfrak{c}_X \cup \Lambda \times \mathbb{R}) \simeq \mathrm{Perf}\,\mathcal{W}(X, \Lambda)^\text{op}.$$
    According to a conjecture by Sylvan \cite[Section 6.4]{GPS3} and Ekholm-Lekili \cite{EkholmLekili}, and works by Ekholm-Lekili \cite{EkholmLekili}, Ekholm \cite{EkSurgery} and Asplund-Ekholm \cite{AsplundEkholm}, when $X$ is a subcritical Weinstein manifold (a Weinstein $2n$-manifold with no index-$n$ critical points), then it is also expected that
    $$\mathrm{Perf}\,\mathcal{W}(X, \Lambda) \simeq \mathrm{Perf}\,\mathcal{A}_{C_{-*}(\Omega_*\Lambda)}(\Lambda)$$
    where $\mathcal{A}_{C_{-*}(\Omega_*\Lambda)}(\Lambda)$ is equipped with $C_{-*}(\Omega_*\Lambda)$-coefficients. Therefore one may expect to construct a Lagrangian cobordism functor between microlocal sheaf categories.

\subsection{Main Results}
    In this paper we construct a Lagrangian cobordism functor between microlocal sheaf categories of compact objects, and its right adjoint functor between microlocal sheaf categories of proper objects, using the result of Nadler-Shende \cite{NadShen}. Our construction is independent of Floer theory and symplectic field theory.

\begin{definition}
    Let $(X, d\lambda)$ be an exact symplectic manifold with ideal contact boundary $\partial_\infty X$. Let the Liouville vector field $Z_\lambda$ be defined by $\iota(Z_\lambda)d\lambda = \lambda$, which we assume to be transverse to the ideal contact boundary. $X$ is a (finite type) Weinstein manifold if there is a proper Morse function $f$ on $X$ such that $Z_\lambda$ is a gradient-like vector field. Write $X_c = f^{-1}((-\infty,c])$. Then the skeleton of $X$ is
    $$\mathfrak{c}_X = \bigcup_{c\in \mathbb{R}}\bigcap_{z > 0}\varphi_{Z_\lambda}^{-z}(X_c).$$
\end{definition}

\begin{remark}
    Throughout the paper, we assume that all Weinstein manifolds $X$, Lagrangian cobordisms $L$ and Legendrian submanifolds $\Lambda_\pm$ are equipped with Maslov data compatible with respect to the inclusions \cite[Section 10]{NadShen}. When $\Bbbk$ is a ring, it requires the first Chern class of the Weinstein manifold $2c_1(X) = 0$, the Maslov class of the Lagrangian $\mu(L) = 0$ and that of the Legendrians $\mu(\Lambda_\pm) = 0$. When $\mathrm{char}\,\Bbbk \neq 2$, we need to assume in addition that $L$ and $\Lambda_\pm$ are relatively spin.
\end{remark}

    Here is our main theorem. Recall that when we say a Lagrangian cobordism $L$ from $\Lambda_-$ to $\Lambda_+$, $\Lambda_+$ is always the Legendrian at the convex boundary (when $r \in \mathbb{R}$ is sufficiently large) and $\Lambda_-$ is at the concave boundary (when $r \in \mathbb{R}$ is sufficiently small).

\begin{theorem}\label{main}
    Let $X$ be a Weinstein manifold with subanalytic skeleton $\mathfrak{c}_X$, $\Lambda_-, \Lambda_+ \subset \partial_\infty X$ be Legendrian submanifolds, and $L \subset \partial_\infty X \times \mathbb{R}$ an exact Lagrangian cobordism from $\Lambda_-$ to $\Lambda_+$. There is a cobordism functor between the microlocal sheaf categories of compact objects
    $$\Phi_L^*: \, \mu Sh^c_{\mathfrak{c}_X \cup \Lambda_+ \times \mathbb{R}}(\mathfrak{c}_X \cup \Lambda_+ \times \mathbb{R}) \longrightarrow \mu Sh^c_{\mathfrak{c}_X \cup \Lambda_- \times \mathbb{R}}(\mathfrak{c}_X \cup \Lambda_- \times \mathbb{R}) \otimes_{Loc^c(\Lambda_-)} Loc^c(L),$$
    and a fully faithful adjoint functor between microlocal sheaf categories of proper objects
    $$\Phi_L: \, \mu Sh^b_{\mathfrak{c}_X \cup \Lambda_- \times \mathbb{R}}(\mathfrak{c}_X \cup \Lambda_- \times \mathbb{R}) \times_{Loc^b(\Lambda_-)} Loc^b(L) \hookrightarrow \mu Sh^b_{\mathfrak{c}_X \cup \Lambda_+ \times \mathbb{R}}(\mathfrak{c}_X \cup \Lambda_+ \times \mathbb{R}),$$
    such that concatenations of cobordisms give rise to compositions of cobordism functors.

    In particular, when $X = T^*M$, there is a cobordism functor between compact sheaves
    $$\Phi_L^*: \, Sh^c_{\Lambda_+}(M) \longrightarrow Sh^c_{\Lambda_-}(M) \otimes_{Loc^c(\Lambda_-)} Loc^c(L),$$
    and a fully faithful adjoint functor between proper sheaves
    $$\Phi_L: \, Sh^b_{\Lambda_-}(M) \times_{Loc^b(\Lambda_-)} Loc^b(L) \hookrightarrow Sh^b_{\Lambda_+}(M).$$
\end{theorem}

\begin{remark}
    The tensor product of categories
    $$\mu Sh^c_{\mathfrak{c}_X \cup \Lambda_- \times \mathbb{R}}(\mathfrak{c}_X \cup \Lambda_- \times \mathbb{R}) \otimes_{Loc^c(\Lambda_-)} Loc^c(L)$$
    is defined as the homotopy push-out of the following diagram
    $$\mu Sh^c_{\mathfrak{c}_X \cup \Lambda_- \times \mathbb{R}}(\mathfrak{c}_X \cup \Lambda_- \times \mathbb{R}) \longleftarrow Loc^c(\Lambda_-) \longrightarrow Loc^c(L)$$
    where the arrows are corestriction functors \cite[Section 3.6]{NadWrapped} (see Section \ref{variouscat}) since $Loc^c(\Lambda) \simeq \mu Sh^c_\Lambda(\Lambda)$ \cite{Gui} (see Section \ref{microsection}). In particular, when $X = T^*M$ the corestriction functor
    $$Loc^c(\Lambda_-) \longrightarrow Sh^c_{\Lambda_-}(M)$$
    is the left adjoint to the microlocalization functor (see Section \ref{microsection}).
\end{remark}
\begin{remark}
    The category of compact local systems $Loc^c(\Lambda)$ is derived Morita equivalent to the chains on based loop space $C_{-*}(\Omega_*\Lambda)$, i.e.~$Loc^c(\Lambda) \simeq \mathrm{Perf}\,C_{-*}(\Omega_*\Lambda)$.
\end{remark}
\begin{remark}
    In the setting of partially wrapped Fukaya categories, the first functor is
    $$\Phi_L^*: \mathcal{W}(X, \Lambda_+) \longrightarrow \mathcal{W}(X, \Lambda_-) \otimes_{Loc^c(\Lambda_-)} Loc^c(L).$$
\end{remark}
\begin{remark}
    Our result also works in the singular setting, including immersed exact Lagrangian cobordisms with vanishing action self intersection points (which lifts to immersed Legendrians with no Reeb chords), and even subanalytic Lagrangian cobordisms between subanalytic Legendrians satisfying the condition above (see Remark \ref{sing-cob}).
\end{remark}

    While the techniques in Nadler-Shende \cite{NadShen} will ensure the first part about existence and full faithfulness of the functor, some techniques beyond that will be necessary when we prove the second part that concatenations of Lagrangian cobordisms define compositions of the functors. These parts together with invariance under compactly supported Hamiltonian isotopies will be included in the Section \ref{mainthm} and \ref{conca-inv}.

    When $L$ is a Lagrangian concordance from $\Lambda_-$ to $\Lambda_+$, i.e.~$L$ is diffeomorphic to $\Lambda_- \times \mathbb{R}$, we have in particular the following fully faithful embedding.

\begin{corollary}
    Let $X$ be a Weinstein manifold with subanalytic skeleton $\mathfrak{c}_X$, $\Lambda_-, \Lambda_+ \subset \partial_\infty X$ be Legendrian submanifolds. Let $L \subset \partial_\infty X \times \mathbb{R}$ be a Lagrangian concordance from $\Lambda_-$ to $\Lambda_+$. Then there is a fully faithful functor between the categories
    $$\Phi_L: \, \mu Sh^b_{\mathfrak{c}_X \cup \Lambda_- \times \mathbb{R}}(\mathfrak{c}_X \cup \Lambda_- \times \mathbb{R}) \hookrightarrow \mu Sh^b_{\mathfrak{c}_X \cup \Lambda_+ \times \mathbb{R}}(\mathfrak{c}_X \cup \Lambda_+ \times \mathbb{R}).$$
    In particular, when $X = T^*M$, there is a fully faithful functor between proper sheaves
    $$\Phi_L: \, Sh^b_{\Lambda_-}(M) \hookrightarrow Sh^b_{\Lambda_+}(M).$$
\end{corollary}

    For Lagrangian cobordisms $L_0, L_1$ from $\Lambda_-$ to $\Lambda_+$, Chantraine-Dimitroglou Rizell-Ghiggini-Golovko \cite{Cthulhu} constructed an acyclic Cthulhu complex $\mathrm{Cth}(\Lambda_{\pm}, L_0, L_1)$ consisting of linearized contact homologies of $\Lambda_\pm$ and the Floer chain complex of $L_0, L_1$, and hence produced a number of exact sequences. Similar to Chantraine-Dimitroglou Rizell-Ghiggini-Golovko \cite{Cthulhu}, we are able to get a series of exact triangles from a Lagrangian cobordism, most of which are simple corollaries of the full faithfulness of our functor $\Phi_L$.

\begin{corollary}[Mayer-Vietoris exact triangle]\label{MVsequence}
    Let $X$ be a Weinstein manifold with subanalytic skeleton $\mathfrak{c}_X$, and $\Lambda_-, \Lambda_+ \subset \partial_\infty X$ be Legendrian submanifolds. Let $L \subset \partial_\infty X \times \mathbb{R}$ be an exact Lagrangian cobordism from $\Lambda_-$ to $\Lambda_+$. Suppose there are sheaves $\mathscr{F}_-, \mathscr{G}_- \in \mu Sh^b_{\mathfrak{c}_X \cup \Lambda_- \times \mathbb{R}}(\mathfrak{c}_X \cup \Lambda_- \times \mathbb{R})$ which restrict to constant local systems along $\Lambda_-$, and their microstalks at $\Lambda_-$ are $F, G$. Denoting by
    $$\mathscr{F}_+ = \Phi_L(\mathscr{F}_-, F_L), \,\,\, \mathscr{G}_+ = \Phi_L(\mathscr{G}_-, G_L),$$
    the images of $\mathscr{F}^-, \mathscr{G}^-$ glued with constant local systems on $L$ with stalks $F$ and $G$, then there is an exact triangle
    $$\Gamma(\mu hom(\mathscr{F}_+, \mathscr{G}_+)) \rightarrow \Gamma(\mu hom(\mathscr{F}_-, \mathscr{G}_-)) \oplus C^*(L; Hom(F, G)) \rightarrow C^*(\Lambda_-; Hom(F, G)) \xrightarrow{+1}.$$
\end{corollary}

    A flexible Weinstein manifold \cite[Chapter 11]{CE} is a Weinstein manifold whose attaching spheres of index-$n$ critical points are all loose Legendrian submanifolds \cite{loose}. Similar to the result in \cite{Cthulhu}, we are able to prove a stronger result that any Legendrian submanifold in the boundary of a flexible Weinstein manifold whose microlocal sheaf category of proper objects over $\Bbbk = \mathbb{Z}/2\mathbb{Z}$ is nontrivial does not admit a Lagrangian cap. Assuming the equivalence between partially wrapped Fukaya categories and Legendrian contact homologies, this means that any Legendrian submanifold whose contact homology over $\Bbbk = \mathbb{Z}/2\mathbb{Z}$ has a proper module does not admit a Lagrangian cap.

\begin{corollary}\label{cap}
    Let $X$ be a flexible Weinstein manifold with subanalytic skeleton $\mathfrak{c}_X$, and $\Lambda_- \subset \partial_\infty X$ be a connected Legendrian submanifold. Suppose $\mu Sh^b_{\mathfrak{c}_X \cup \Lambda_- \times \mathbb{R}}(\mathfrak{c}_X \cup \Lambda_- \times \mathbb{R}) $ contains a nontrivial object which restricts to a constant local system along $\Lambda_-$. Then there is no Lagrangian cobordism from $\Lambda_-$ to $\varnothing$ with vanishing Maslov class.
\end{corollary}
\begin{remark}
    Since there are examples whose partially wrapped Fukaya category only has higher dimensional representations \cite{Lazexample,Lazexample2}, by the equivalence between Fukaya categories and sheaf categories \cite{GPS3} and the fact that \cite[Theorem 3.21]{NadWrapped} (see Section \ref{variouscat})
    $$\mu Sh_{\mathfrak{c}_X}^b(\mathfrak{c}_X) \simeq \mathrm{Fun}^\text{ex}(\mu Sh_{\mathfrak{c}_X}^c(\mathfrak{c}_X)^\text{op}, \mathrm{Perf}(\Bbbk)),$$
    this corollary is expected to be stronger than the result in \cite{Cthulhu}. Note that there are also examples whose Legendrian contact homology is nontrivial but has only higher dimensional representations \cite{Sivek}.
\end{remark}

\begin{remark}
    The assumption that the sheaf which restrict to a constant local system along $\Lambda_-$ is necessary. For example, the Clifford Legendrian torus $\Lambda_\text{Cliff}$ discussed in Theorem \ref{2graphs} does admit a microlocal rank 1 sheaf. However, there is a Lagrangian cobordism from $\Lambda_\text{Cliff}$ to a loose Legendrian sphere $\Lambda_{S^2,\text{loose}}$ \cite[Example 4.26]{CasalsZas} (see Section \ref{application}), and hence there is a Lagrangian cap by \cite{LagCap}.
\end{remark}

\begin{proof}[Proof of Corollary \ref{cap}]
    Let $\mathscr{F}_- \in \mu Sh^b_{\mathfrak{c}_X \cup \Lambda_- \times \mathbb{R}}(\mathfrak{c}_X \cup \Lambda_- \times \mathbb{R})$ be a nonzero object with stalk at $\Lambda_-$ being $F$. Suppose there is an exact Lagrangian cobordism from $\Lambda_-$ to $\varnothing$. Then since $\mathscr{F}_-$ restricts to a constant local system and the stalk $F$ at $\Lambda_-$ is nonzero, it can be extended to a constant local system on $L$ with stalk $F$. Glue $\mathscr{F}_-$ with the local system $F_L$ and write $\mathscr{F}_+ = \Phi_L(\mathscr{F}_-, F_L)$. Since $X$ is flexible, $\Gamma(\mu hom(\mathscr{F}_+, \mathscr{F}_+)) \simeq 0$. From the Mayer-Vietoris exact triangle we know that (by setting $\mathscr{G}_- = \mathscr{F}_-$ and $\mathscr{G}_+ = \mathscr{F}_+$)
    $$\Gamma(\mu hom(\mathscr{F}_-, \mathscr{F}_-)) \oplus C^*(L; Hom(F, F)) \simeq C^*(\Lambda_-; Hom(F, F)).$$
    However, the fact that $H^0(L; Hom(F, F)) \simeq H^0(\Lambda_-; Hom(F, F))$ will force
    $$H^0(\mu hom(\mathscr{F}_-, \mathscr{F}_-)) = 0, $$
    i.e.~$\mathrm{id}_{\mathscr{F}_-} = 0,$ which gives a contradiction.
\end{proof}

\begin{remark}
    The fact that flexible Weinstein domains have trivial microlocal sheaf categories follows from \cite{GPS3}, the vanishing result for their symplectic cohomologies \cite[Theorem 3.2]{Subflexible} (using the embedding trick \cite[Corollary 6.3]{LagCap}) and Abouzaid's generation criterion \cite{AbGenerate}. In fact using the embedding functor \cite{NadShen} (see Section \ref{sheafquan}) we can also get a sheaf theoretic proof of this fact.
\end{remark}

    The next exact sequence is the following, analogous to results in \cite{Cthulhu}*{Theorem 1.1} and Pan \cite{PanAug}*{Theorem 1.2}.

\begin{corollary}\label{Exactsequence}
    Let $X$ be a Weinstein manifold with subanalytic skeleton $\mathfrak{c}_X$, and $\Lambda_-, \Lambda_+ \subset \partial_\infty X$ be Legendrian submanifolds. Let $L \subset \partial_\infty X \times \mathbb{R}$ be an exact Lagrangian cobordism from $\Lambda_-$ to $\Lambda_+$. Suppose there are sheaves $\mathscr{F}_-, \mathscr{G}_- \in \mu Sh^b_{\mathfrak{c}_X \cup \Lambda_- \times \mathbb{R}}(\mathfrak{c}_X \cup \Lambda_- \times \mathbb{R})$ which restrict to constant local systems along $\Lambda_-$, and their stalks at $\Lambda_-$ are $F, G$. Denoting by
    $$\mathscr{F}_+ = \Phi_L(\mathscr{F}_-, F_L), \,\,\, \mathscr{G}_+ = \Phi_L(\mathscr{G}_-, G_L),$$
    the images of $\mathscr{F}_-, \mathscr{G}_-$ glued with constant local systems on $L$ with stalks $F$ and $G$, then there is an exact triangle
    $$\Gamma(\mu hom(\mathscr{F}_+, \mathscr{G}_+)) \rightarrow \Gamma(\mu hom(\mathscr{F}_-, \mathscr{G}_-)) \rightarrow C^*(L, \Lambda_-; Hom(F, G))[1] \xrightarrow{+1}.$$
\end{corollary}
\begin{remark}
    Following \cite[Theorem 1.6]{PanAug}, restricting to the subcategory $\mu Sh^b_{\mathfrak{c}_X \cup \Lambda_- \times \mathbb{R}}(\mathfrak{c}_X \cup \Lambda_- \times \mathbb{R})_\text{tri} \subset \mu Sh^b_{\mathfrak{c}_X \cup \Lambda_- \times \mathbb{R}}(\mathfrak{c}_X \cup \Lambda_- \times \mathbb{R})$ of microlocal sheaves which restrict to constant local systems along $\Lambda_-$, the functor defined by gluing with the constant local system on $L$
    $$\mu Sh^b_{\mathfrak{c}_X \cup \Lambda_- \times \mathbb{R}}(\mathfrak{c}_X \cup \Lambda_- \times \mathbb{R})_\text{tri} \rightarrow \mu Sh^b_{\mathfrak{c}_X \cup \Lambda_+ \times \mathbb{R}}(\mathfrak{c}_X \cup \Lambda_+ \times \mathbb{R})_\text{tri}$$
    is injective on objects as long as $H^0(L, \Lambda_-) = 0$. The proof is the same as \cite{PanAug}, where one uses the fact that
    $$H^0(\mu hom(\mathscr{F}_+, \mathscr{G}_+)) \xrightarrow{\sim} H^0(\mu hom(\mathscr{F}_-, \mathscr{G}_-))$$
    preserves the identity.
\end{remark}

    In particular, when $\Lambda_- = \varnothing$, i.e.~when $L$ is an exact Lagrangian filling of $\Lambda_+$, by choosing the constant rank 1 local system on $L$, we are able to get a sheaf quantization $\mathscr{F}_+$ of $L$ and this recovers the Seidel isomorphism \cite{SeidelIso}. The first proof in sheaf theory when $X = T^*M$ is obtained by Jin-Treumann \cite{JinTreu}.

    Note that in contrary to \cite{SeidelIso}, the proof in sheaf theory does not require $\mathcal{W}(X)$ or $\mu Sh^c_{\mathfrak{c}_X}(\mathfrak{c}_X)$ to vanish (because the sheaf categories are always identified with Fukaya categories, but they are expected to be the Chekanov-Eliashberg dg algebra or its representations only when the ambient manifold is flexible).

\begin{corollary}[Nadler-Shende]
    Let $X$ be a Weinstein manifold with subanalytic skeleton $\mathfrak{c}_X$, and $\Lambda_+ \subset \partial_\infty X$ be a Legendrian submanifold. Let $\Bbbk$ be a ring. Let $L \subset X$ be an exact Lagrangian filling of $\Lambda_+$. Then there is $\mathscr{F}_+ \in \mu Sh^b_{\mathfrak{c}_X \cup \Lambda_+ \times \mathbb{R}}(\mathfrak{c}_X \cup \Lambda_+ \times \mathbb{R})$ such that
    $$\Gamma(\mu hom(\mathscr{F}_+, \mathscr{F}_+)) \simeq C^*(L; \Bbbk).$$
\end{corollary}
\begin{proof}
    Pick the rank 1 constant local system on $\mu Sh^b_L(L) \simeq Loc^b(L)$. Then by Corollary \ref{Exactsequence} we can get the result.
\end{proof}

\subsection{Relations with Other Works}
    There are at least two classes of special Lagrangian cobordisms that appear in literature and are well studied in microlocal sheaf theory.

\subsubsection{Relation with sheaf quantization of Legendrian isotopy}
    When there is a Legendrian isotopy $\varphi_H^s,\,s \in I$, from $\Lambda_0$ to $\Lambda_1$, it will define a Lagrangian cobordism $L$ from $\Lambda_0$ to $\Lambda_1$ \cite[Section 4.2.3]{GroEliashGF,Chantraine}. Hence we have a fully faithful Lagrangian cobordism functor
    $$\Phi_{L}: Sh^b_{\Lambda_0}(M) \hookrightarrow Sh^b_{\Lambda_1}(M).$$
    On the other hand, Guillermou-Kashiwara-Schapira \cite{GKS} constructed a sheaf quantization functor $\Psi_H$ from a Hamiltonian isotopy given by taking convolution with an integral kernel. We will prove the following comparison theorem in Section \ref{compare1}.

\begin{theorem}\label{comparegks}
    Let $\Lambda_s\subset T^{*,\infty}M,\,s \in I$, be a Legendrian isotopy induced by $\varphi_H^s,\,s\in I$, with vanishing Maslov class, and $L$ the Lagrangian cobordism from $\Lambda_0$ to $\Lambda_1$ coming from the isotopy. Then for $\Phi_{L}$ the Lagrangian cobordism functor and $\Psi_H$ the sheaf quantization functor,
    $$\Phi_{L} \simeq \Psi_H: \, Sh^b_{\Lambda_0}(M) \rightarrow Sh^b_{\Lambda_1}(M).$$
\end{theorem}

\subsubsection{Relation with sheaf quantization of Lagrangian fillings}
    When $\Lambda_- = \varnothing$, a Lagrangian cobordism from $\Lambda_-$ to $\Lambda_+$ is a Lagrangian filling. Jin-Treumann \cite{JinTreu} constructed a sheaf quantization functor $Loc^b(L) \rightarrow Sh_{\Lambda_+}^b(M)$ from any Lagrangian filling $L$ of $\Lambda_+ \subset T^{*,\infty}M$, that is, a fully faithful embedding
    $$\Psi_L^{JT}: \,Loc^b(L) \hookrightarrow Sh^b_{\Lambda_+}(M).$$
    We will show the following comparison result in Section \ref{compare2}.

\begin{proposition}\label{comparejt}
    Let $U \subset M$ be an open subset with subanalytic boundary, $\Lambda_+ = \nu^{*,\infty}_{U,-}M$ be the inward unit conormal and $L$ the standard Lagrangian brane associated to $U$ with Legendrian boundary $\Lambda_+$. Then for $\Phi_{L}$ the Lagrangian cobordism functor and $\Psi_L^{JT}$ the Jin-Treumann sheaf quantization functor,
    $$\Phi_L \simeq \Psi_L^{JT}: \,Loc^b(L) \hookrightarrow Sh^b_{\Lambda_+}(M).$$
\end{proposition}

    In fact, using Nadler-Zaslow correspondence \cite{NadZas,Nad} or Viterbo's sheaf quantization construction \cite{ViterboSheaf}, if one can prove additionally the functoriality of $\Phi_L$ and $\Psi_L^{JT}$ as functors from infinitesimal Fukaya categories, then $\Phi_L \simeq \Psi_L^{JT}$ for any Lagrangian filling of any Legendrians $\Lambda_+$.

\subsubsection{Other proposals of the cobordism functor}\label{proposal}
    An exact Lagrangian cobordism $L$ in $Y \times \mathbb{R}$ from $\Lambda_-$ to $\Lambda_+$ can be lifted to a Legendrian cobordism $\widetilde{L}$ in $Y \times \mathbb{R}^2$ between Legendrians $\Lambda_-$ and $\Lambda_+$. Pan-Rutherford \cite{PanRuther} considered for $\Bbbk$-coefficient dg algebras (instead of loop space coefficients) a diagram
    $$\mathcal{A}(\Lambda_+) \rightarrow \mathcal{A}(\widetilde{L}) \xleftarrow{\sim} \mathcal{A}(\Lambda_-)$$
    and showed that this coincides with the usual dg algebra map induced by Lagrangian cobordisms by symplectic field theory.

    For the dg algebra with loop space coefficients, we thus conjecture that there is a diagram
    $$\mathcal{A}_{C_{-*}(\Omega_*\Lambda_+)}(\Lambda_+) \twoheadrightarrow \mathcal{A}_{C_{-*}(\Omega_*\widetilde{L})}(\widetilde{L}) \xleftarrow{\sim} \mathcal{A}_{C_{-*}(\Omega_*\Lambda_-)}(\Lambda_-) \otimes_{C_{-*}(\Omega_*\Lambda_-)} C_{-*}(\Omega_*\widetilde{L})$$
    or in the language of sheaf theory
    $$Sh_{\Lambda_+}^c(M) \twoheadrightarrow Sh_{\widetilde{L}}^c(M \times \mathbb{R}) \xleftarrow{\sim} Sh_{\Lambda_-}^c(M) \otimes_{Loc^c(\Lambda_-)} Loc^c(L)$$
    that coincides with our construction here. The right adjoint functor will thus be
    $$Sh_{\Lambda_+}^b(M) \hookleftarrow Sh_{\widetilde{L}}^b(M \times \mathbb{R}) \xrightarrow{\sim} Sh_{\Lambda_-}^b(M) \times_{Loc^b(\Lambda_-)} Loc^b(L)$$
    Here $Sh_{\widetilde{L}}^b(M \times \mathbb{R}) \rightarrow \mu Sh^b_{\widetilde{L}}(\widetilde{L}) \xrightarrow{\sim} Loc^b(\widetilde{L})$ is the microlocalization functor (Section \ref{microsection}) while $Sh_{\widetilde{L}}^b(M \times \mathbb{R}) \rightarrow Sh^b_{\Lambda_\pm}(M)$ are the restriction functors\footnote{The author is grateful to Roger Casals and Eric Zaslow for explaining to us this alternative approach.}.

\subsection{Applications to Legendrian Surfaces}
    In the past few years, Treumann-Zaslow \cite{TZ} and Casals-Zaslow \cite{CasalsZas} have developed systematic approaches to compute the number of microlocal rank 1 sheaves over $\mathbb{F}_q$ for certain Legendrian surfaces using flag moduli. Combining with our fully faithful cobordism functor on proper sheaves, we will be able to get new obstructions to Lagrangian cobordisms for these Legendrian surfaces.

\begin{figure}
  \centering
  \includegraphics[width=0.7\textwidth]{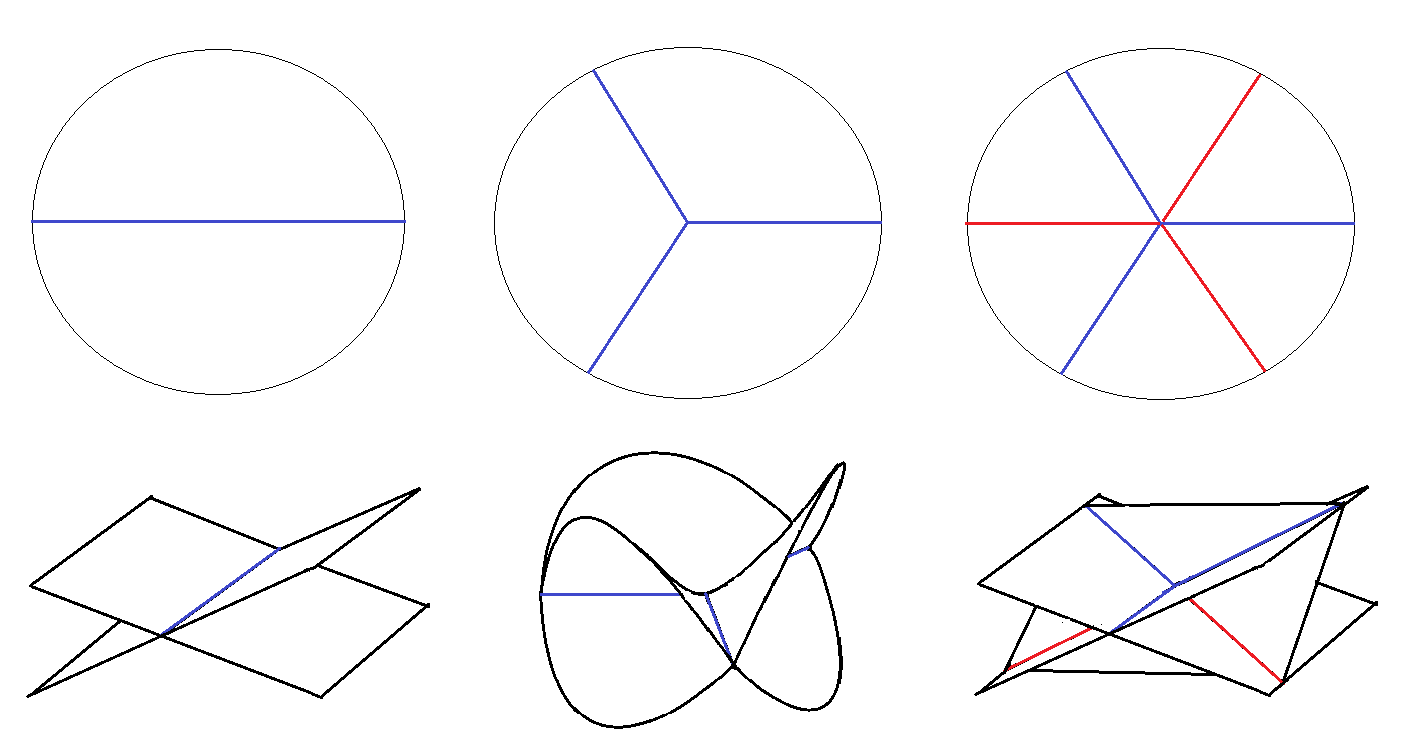}\\
  \caption{The front projection $\pi_\text{front}: J^1(\Sigma) \rightarrow \Sigma \times \mathbb{R}$ of the Legendrian surfaces corresponding to each planar $N$-graph on the top.}\label{weavedef}
\end{figure}

    First recall that Legendrian weaves \cite{CasalsZas} are Legendrian submanifolds in $J^1(\Sigma)$ that arise from planar $N$-graphs. Figure \ref{weavedef} roughly explains locally how an $N$-graph corresponds to the front projection of Legendrians.

\begin{figure}
  \centering
  \includegraphics[width=1.0\textwidth]{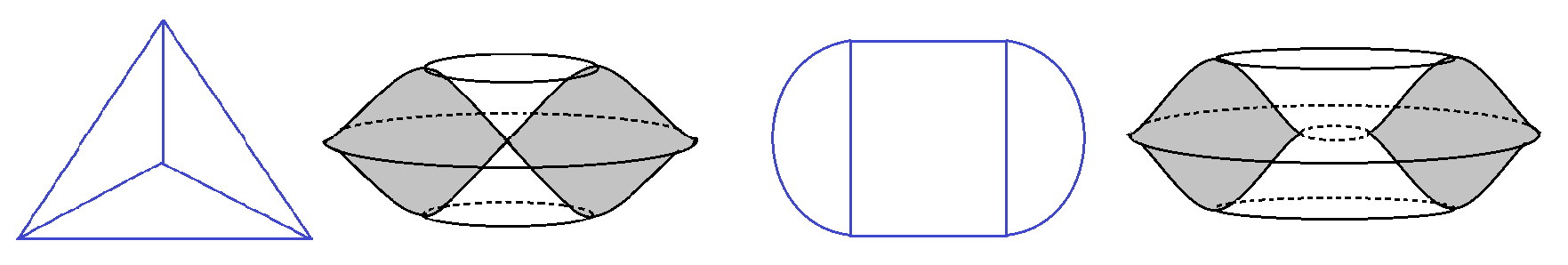}\\
  \caption{On the left is the Clifford Legendrian torus and its corresponding 2-graph, and on the right is the unknotted Legendrian torus and its corresponding 2-graph.}\label{clifunknot}
\end{figure}

%    The following examples are considered in \cite{CasalsZas}, where they showed that $\Lambda_n\,(n \geq 1)$ are smoothly isotopic but not Legendrian isotopic. Here we show the following result.
%
%\begin{theorem}\label{spherelink}
%    Let $\mathcal{L}_n$ be the 3-graph in $S^2$ coming from the $2n$-runged ladder 3-graph $L_n$ shown in Figure \ref{sphere}, and $\Lambda_n \subset J^1(S^2)$ be the Legendrian weave determined by $\mathcal{L}_n$. Then there is no Lagrangian cobordism with vanishing Maslov class from $\Lambda_n$ to $\Lambda_m$ as long as $n < m$.
%\end{theorem}
%
%\begin{figure}
%  \centering
%  \includegraphics[width=1.0\textwidth]{spherelink.png}\\
%  \caption{The ladder graph $L_n$ on the left and the corresponding 3-graph $\mathcal{L}_n$ on the right.}\label{sphere}
%\end{figure}
%
    The following examples of Legendrian surfaces $\Lambda_{g,k}$ are considered in \cite{Rizchordexample} and \cite{ShenZas} ($\Lambda_{g,0}$ are the unknotted Legendrian surfaces and $\Lambda_{g,g}$ are Clifford Legendrian surfaces). Dimitroglou Rizell showed that those $\Lambda_{g,k}$'s admit $\mathbb{Z}/2\mathbb{Z}$-coefficient augmentations and generating families only when $k = 0$, and hence it may not be easy to study Lagrangian cobordisms between them when $k \geq 1$. However, using the Legendrian weave description, we are able to show the following.

\begin{theorem}\label{2graphs}
    Let $\Gamma_\text{Unknot}, \Gamma_\text{Cliff}$ be the 2-graphs in $S^2$ shown in Figure \ref{clifunknot}, and $\Lambda_\text{Unknot}, \Lambda_\text{Cliff}$ the corresponding Legendrian weaves in $J^1(S^2) \subset T^{*,\infty}\mathbb{R}^3$. Let $\Lambda_{g,k}$ be the Legendrian surface with genus $g$ by taking $k$ copies of $\Lambda_\text{Cliff}$ and $g-k$ copies of $\Lambda_\text{Unknot}$. Then
\begin{enumerate}
  \item for any $g' \leq g$, there are Lagrangian cobordisms from $\Lambda_{g,k}$ to $\Lambda_{g',k}$ and also from $\Lambda_{g',k}$ to $\Lambda_{g,k}$;
  \item (Dimitroglou Rizell) for any $k \geq 1$, there are no Lagrangian cobordisms with vanishing Maslov class from $\Lambda_{g,0}$ to $\Lambda_{g',k}$;
  \item for any $k \geq 1, k' \geq 0$, there are Lagrangian cobordisms $L$ from $\Lambda_{g,k}$ to $\Lambda_{g,k'}$ such that $\dim\mathrm{coker}(H^1(L) \rightarrow H^1(\Lambda_{g,k})) \geq 1;$
  \item for any $k < k'$, there are no Lagrangian cobordisms $L$ with vanishing Maslov class from $\Lambda_{g,k}$ to $\Lambda_{g,k'}$ such that $H^1(L) \twoheadrightarrow H^1(\Lambda_{g,k});$ in particular there are no such Lagrangian concordances.
\end{enumerate}
\end{theorem}
\begin{remark}
    We will see that Part~(2) is a direct corollary of either \cite{Rizchordexample} or \cite{TZ}.
\end{remark}

    Roughly speaking, the Legendrian $\Lambda_{g,k}$ is closer to being Lagrangian fillable when $k$ is smaller (in particular $\Lambda_{g,0}$ are the only Lagrangian fillable ones). We would expect that it is difficult to have a Lagrangian cobordism from $\Lambda_{g,k}$ to $\Lambda_{g,k'}$ if $k > k'$. Our theorem shows that, for $k > k'$, there are indeed obstructions for Lagrangian cobordisms to exist from $\Lambda_{g,k}$ to $\Lambda_{g,k'}$ assuming either (2)~$k = 0$ or (4)~$H^1(L) \rightarrow H^1(\Lambda_{g,k})$ is surjective. On the contrary, as long as we assume (3)~$k \geq 1$ and $H^1(L) \rightarrow H^1(\Lambda_{g,k})$ is not surjective, then we enter the world of flexibility and there are no obstructions for Lagrangian cobordisms (and $\dim\mathrm{coker}(H^1(L) \rightarrow H^1(\Lambda_{g,k}))$ can even be very small).

    In earlier works, we knew that the Euler characteristic of the Lagrangian is determined by Bennequin-Thurston numbers of the Legendrians \cite{EESnoniso}. When the Chekanov-Eliashberg dg algebra has a $\Bbbk$-augmentation (without loop space coefficients), then there are obstructions on $H^*(L)$ coming from the Cthulhu complexes \cite{Cthulhu}. Our result gives new examples where we have more precise characterization of the smooth cobordism types, in particular the homotopy type of inclusion $\Lambda_- \hookrightarrow L$. In general, it will be an interesting problem whether certain smooth cobordism type can be realized by an exact Lagrangian cobordism.

\subsection{Organization of the Paper}
    Section \ref{prelim} will be the background of the microlocal sheaf theory that we will need in this paper, and in particular Section \ref{sheafquan} will explain Nadler-Shende's construction of sheaf categories of Weinstein manifolds and related results, which is the key technique in our main theorem. Section \ref{mainthm} and \ref{conca-inv} cover the proofs of the main theorems, and Section \ref{compare1} and \ref{compare2} cover the comparison results Theorem \ref{comparegks} and Proposition \ref{comparejt}. In Section \ref{elementary} we study elementary cobordisms, and finally in Section \ref{application} we prove the results for Legendrian surfaces Theorem %\ref{spherelink} and %
    \ref{2graphs}.

\subsection{Conventions}
    Geometric conventions: For a Weinstein domain $X$, $\partial_\infty X$ is its contact boundary. In particular, for $T^*M$, $T^{*,\infty}M$ is its contact boundary, and in the paper we will identify it with the unit cotangent bundle. $T^{*,\infty}_{\tau>0}(M \times \mathbb{R})$ is the subbundle of $T^{*,\infty}(M \times \mathbb{R})$ consisting of points so that the covector coordinate in $T^*\mathbb{R}$ is $\tau>0$. For a closed submanifold $N \subset M$, $\nu^{*,\infty}_NM$ is the unit conormal bundle. For an open subset $U \subset M$ with subanalytic boundary, $\nu^{*,\infty}_{U,+/-}M$ is the outward/inward unit conormal bundle.

    As is already mentioned at the beginning, all Lagrangians and Legendrians in this paper are equipped with Maslov data. We say that a Lagrangian cobordism $L$ is from $\Lambda_-$ to $\Lambda_+$ if $\Lambda_+$ is at the convex end and $\Lambda_-$ is at the concave end.

    Categorical conventions: All categories in this paper are dg categories, and all functors will be functors in dg categories. $Sh_{-}, \mu Sh_{-}$ are the dg categories consisting of all possibly unbounded complexes of sheaves with prescribed (isotropic) singular support, $Sh^c_{-}, \mu Sh^c_{-}$ are the dg subcategories of compact objects, and $Sh^b_{-}, \mu Sh^b_{-}$ are the dg subcategories of proper objects. They are all localized along acyclic objects.

\subsection*{Acknowledgements}
    I would like to thank my advisors Emmy Murphy and Eric Zaslow for plenty of helpful discussions and comments, in particular Emmy Murphy for explaining the general version of Lagrangian caps used in Theorem \ref{2graphs} Part~(3) and Eric Zaslow for helpful discussions on Section \ref{proposal}. I would like to thank Vivek Shende for explaining some details in his work and essentially explaining the strategy of the proof of Theorem \ref{comparegks}, and to thank anonymous referees for providing helpful comments and pointing out the mistakes in Lemma \ref{basechange}. I am also grateful to Roger Casals for helpful discussions and comments on Section \ref{proposal} and Section \ref{elementary}. Finally I am grateful to Honghao Gao and Yuichi Ike for helpful comments.

\section{Preliminaries in Sheaf Theory}\label{prelim}

\subsection{Singular Supports}\label{prelim-ss}
    We briefly review results in microlocal sheaf theory that we are going to use in this paper. For the theory of category of sheaves with unbounded cohomologies, one can refer to \cite{Unbound}.

\begin{definition}
    Let $\underline{Sh}(M)$ be the dg category of sheaves over $\Bbbk$, that consists of complexes of sheaves over $\Bbbk$. Then $Sh(M)$ is the dg localization of $\underline{Sh}(M)$ along all acyclic objects (with possibly unbounded cohomologies).
\end{definition}

\begin{example}
    We denote by $\Bbbk_M$ the constant sheaf on $M$. For a locally closed subset $i_V: V \hookrightarrow M$, abusing notations, we will write
    $$\Bbbk_V = i_{V !}\Bbbk_V \in Sh^b(M).$$
    In particular, $\Bbbk_V \in Sh^b(M)$ will have stalk $\Bbbk$ for $x \in V$ and stalk $0$ for $x \notin V$. Note that when $V \hookrightarrow M$ is a closed subset, we can also write $\Bbbk_V = i_{V *}\Bbbk_V$.
\end{example}

    We define the singular support of a sheaf, which is the starting point of microlocal sheaf theory. For the microlocal theory of sheaves with unbounded cohomologies, one may refer to \cite{MicrolocalInfty} or \cite[Section 2]{JinTreu}.

\begin{definition}
    Let $\mathscr{F} \in Sh(M)$. Then its singular support $SS(\mathscr{F})$ is the closure of the set of points $(x, \xi) \in T^*M$ such that there exists a smooth function $\varphi \in C^1(M)$, $\varphi(x) = 0, d\varphi(x) = \xi$ and
    $$\Gamma_{\varphi^{-1}([0,+\infty))}(\mathscr{F})_x \neq 0.$$
    The singular support at infinity is $SS^\infty(\mathscr{F}) = SS(\mathscr{F}) \cap T^{*,\infty}M$.

    For $\widehat \Lambda \subset T^*M$ any conical subset (resp.~$\Lambda \subset T^{*,\infty}M$ any subset), let $Sh_{\widehat \Lambda}(M) \subset Sh(M)$ (resp.~$Sh_\Lambda(M) \subset Sh(M)$) be the full subcategory of sheaves such that $SS(\mathscr{F}) \subset \widehat\Lambda$ (resp.~$SS^\infty(\mathscr{F}) \subset \Lambda$).
\end{definition}

\begin{example}
    Let $\mathscr{F} = \Bbbk_{\mathbb{R}^n \times [0,+\infty)}$. Then $SS(\mathscr{F}) = \mathbb{R}^n \times \{(x, \xi) | x \geq 0, \xi = 0 \text{ or } x = 0, \xi \geq 0\}$, $SS^\infty(\mathscr{F}) = \nu^{*,\infty}_{\mathbb{R}^n \times \mathbb{R}_{>0},-}\mathbb{R}^{n+1} = \{(x_1, ..., x_n, 0, 0,..., 0, 1)\}$, which is the inward conormal bundle of $\mathbb{R}^n \times \mathbb{R}_{>0}$.

    Let $\mathscr{F} = \Bbbk_{\mathbb{R}^n \times (0,+\infty)}$. Then $SS(\mathscr{F}) = \mathbb{R}^n \times \{(x, \xi) | x \geq 0, \xi = 0 \text{ or } x = 0, \xi \leq 0\}$, $SS^\infty(\mathscr{F}) = \nu^{*,\infty}_{\mathbb{R}^n \times \mathbb{R}_{>0},+}\mathbb{R}^{n+1} = \{(x_1, ..., x_n, 0, 0,..., 0, -1)\}$, which is the outward conormal bundle of $\mathbb{R}^n \times \mathbb{R}_{>0}$.
\end{example}

\begin{figure}
  \centering
  \includegraphics[width=0.4\textwidth]{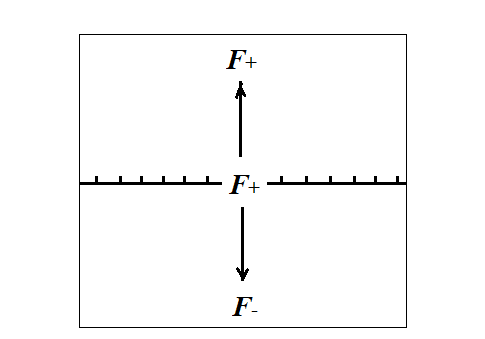}\\
  \caption{The singular support of a sheaf and the combinatoric description.}\label{singularsupport}
\end{figure}

    We have the following non-characteristic deformation lemma, that will allow us to write down explicit combinatoric models for a large class of sheaves, given the singular support condition.

\begin{proposition}[non-characteristic deformation lemma, {\cite[Proposition 2.7.2]{KS}}]\label{morselemma}
    Let $\mathscr{F} \in Sh(M)$ and $\{U_t\}_{t \in \mathbb{R}}$ be a family of open subsets and $Z_t = \bigcap_{t > s}\overline{U_t \backslash U_s}$. Suppose that
    \begin{enumerate}
      \item $U_t = \bigcup_{s < t} U_s$, for $-\infty < t < +\infty$;
      \item $\overline{U_t \backslash U_s} \cap \mathrm{supp}(\mathscr{F})$ is compact, for $-\infty < s < t < +\infty$;
      \item $\Gamma_{M \backslash U_t}(\mathscr{F})_x = 0$, for $x \in Z_s \backslash U_t, \, -\infty < s \leq t < +\infty$.
    \end{enumerate}
    Then for any $t \in \mathbb{R}$ we have
    $$\Gamma\bigg(\bigcup_{s \in \mathbb{R}} U_s, \mathscr{F}\bigg) \xrightarrow{\sim} \Gamma(U_t, \mathscr{F}).$$
\end{proposition}
%\begin{proposition}[microlocal Morse lemma, \cite{KS}*{Corollary 5.4.19}]\label{morselemma}
%    Let $\mathscr{F} \in Sh(M)$ and $f: M \rightarrow \mathbb{R}$ be a proper smooth function. Suppose for any $x\in f^{-1}([a, b))$, $df(x) \notin SS(\mathscr{F})$. Then
%    $$\Gamma(f^{-1}((-\infty, b)), \mathscr{F}) \xrightarrow{\sim} \Gamma(f^{-1}((-\infty, a)), \mathscr{F}).$$
%\end{proposition}
\begin{example}[\cite{STZ}*{Section 3.3}]\label{combin-model}
    Suppose $\Lambda = \nu^{*,\infty}_{\mathbb{R}^n \times \mathbb{R}_{>0},-}\mathbb{R}^{n+1} \subset T^{*,\infty}\mathbb{R}^{n+1}$ is the inward conormal bundle at infinity and $\mathscr{F} \in Sh_\Lambda(\mathbb{R}^{n+1})$. Then by non-characteristic deformation lemma, $\mathscr{F}|_{\mathbb{R}^n \times \{0\}}$, $\mathscr{F}|_{\mathbb{R}^n \times (0,+\infty)}$ and $\mathscr{F}|_{\mathbb{R}^n \times (-\infty,0)}$ are locally constant sheaves, and
    $$\Gamma(\mathbb{R}^n \times \{0\}, \mathscr{F}) \simeq \Gamma(\mathbb{R}^{n+1}, \mathscr{F}) \simeq \Gamma(\mathbb{R}^n \times [0, +\infty), \mathscr{F}).$$
    Suppose that
    $$\mathscr{F}|_{\mathbb{R}^n \times [0,+\infty)} = (F_+)_{\mathbb{R}^n \times [0,+\infty)}, \,\,\, \mathscr{F}|_{\mathbb{R}^n \times (-\infty,0)} = (F_-)_{\mathbb{R}^n \times (-\infty,0)}.$$
    Then $\mathscr{F}$ is determined by the diagram (Figure \ref{singularsupport})
    \[\xymatrix{
    F_- & F_+ \ar[l] \ar[r]^\sim & F_+
    }\]
\end{example}

    Here are some singular support estimates that we are going to use. Let $A, B \subset T^*M$. Then define $(x, \xi) \in A \,\widehat +\, B$ iff there exists $(a_n, \alpha_n) \in A, (b_n, \beta_n) \in B$ such that
    $$a_n, b_n \rightarrow x, \, \alpha_n + \beta_n \rightarrow \xi, \, |a_n - b_n||\alpha_n| \rightarrow 0.$$
    Let $i: M \rightarrow N$ be a closed embedding. Then for $A \subset T^*N$, define $(x, \xi) \in i^{\#}(A)$ iff there exists $(y_n, \eta_n, x_n, 0) \in A \times T^*M$ such that
    $$x_n, y_n \rightarrow x, \, i^*\eta_n \rightarrow \xi, \, |x_n - y_n||\eta_n| \rightarrow 0.$$

\begin{proposition}[\cite{KS}*{Theorem 6.3.1}]\label{sspushforward}
    Let $j: U \hookrightarrow N$ be an open embedding, $\mathscr{F} \in Sh(U)$. Then
    $$SS(j_*\mathscr{F}) \subset SS(\mathscr{F}) \,\widehat +\, \nu_{U,-}^*N,$$
    where $\nu_{U,-}^*N$ is the inward conormal bundle of $U \subset N$.
\end{proposition}
\begin{proposition}[\cite{KS}*{Corollary 6.4.4}]\label{sspullback}
    Let $i: M \rightarrow N$ be a closed embedding, $\mathscr{F} \in Sh(N)$. Then
    $$SS(i^{-1}\mathscr{F}) \subset i^{\#}SS(\mathscr{F}).$$
\end{proposition}

    Kashiwara-Schapira proved that the singular support of a sheaf is always a coisotropic conical subset in $T^*M$. When the singular support of a sheaf is a subanalytic Lagrangian subset, then it is called a weakly constructible sheaf \cite[Definition 8.4.3]{KS}.

    In particular, for a weakly constructible sheaf $\mathscr{F}$, when $\epsilon > 0$ is sufficiently small, the outward conormal bundle $\nu^{*,\infty}_{\partial B_\epsilon(x),+}M$ will be disjoint from the subanalytic Legendrian $SS^\infty(\mathscr{F})$, and thus by microlocal Morse lemma we have \cite[Lemma 8.4.7]{KS}
    $$\mathscr{F}_x \simeq
    %\varinjlim_{\epsilon > 0}\Gamma(\overline{B_\epsilon(x)}, \mathscr{F}) \simeq \varinjlim_{\epsilon > 0}\Gamma(B_\epsilon(x), \mathscr{F}) \simeq
    \Gamma(\overline{B_\epsilon(x)}, \mathscr{F}) \simeq \Gamma(B_\epsilon(x), \mathscr{F}).$$

\subsection{Microlocalization and $\mu Sh$}\label{microsection}
    We review the definition and properties of microlocalization and the sheaf of categories $\mu Sh$, which has been introduced and studied in \cite{KS}*{Section 6}, \cite{Gui}*{Section 6} or \cite{NadWrapped}*{Section 3.4}. This is a category that we will frequently use. Here we follow the definition in \cite{NadShen}*{Section 5}.

\begin{definition}\label{musheaf}
    Let $\widehat\Lambda \subset T^*M$ be a conical subset. Then define a presheaf of dg categories on $T^*M$ supported on $\widehat \Lambda$ to be
    $$\mu Sh^{\text{pre}}_{\widehat\Lambda}: \, \widehat \Omega \mapsto Sh_{\widehat \Lambda\, \cup \,T^*M \backslash \widehat \Omega}(M)/Sh_{T^*M \backslash \widehat \Omega}(M),$$
    The sheafification of $\mu Sh^{\text{pre}}_{\widehat\Lambda}$ is $\mu Sh_{\widehat\Lambda}$. In particular, we write $\mu Sh = \mu Sh_{T^*M}$ for the sheaf of categories on $T^*M$.

    Let $Sh_{(\widehat\Lambda)}(M)$ be the subcategory of sheaves $\mathscr{F}$ such that there exists some neighbourhood $\widehat \Omega$ of $\widehat\Lambda$ satisfying $SS(\mathscr{F}) \cap \widehat \Omega \subset \widehat\Lambda$. For $\mathscr{F, G} \in Sh_{(\widehat\Lambda)}(M)$, let the sheaf of homomorphisms in the sheaf of categories $\mu Sh_{\widehat\Lambda}$ be
    $$\mu hom(\mathscr{F}, \mathscr{G})|_{\widehat\Lambda}: \, \widehat\Omega \mapsto Hom_{\mu Sh_{\widehat\Lambda}(\widehat \Omega)}(\mathscr{F, G}).$$
    Write $\mu hom(\mathscr{F}, \mathscr{G})$ to be the sheaf of homomorphisms in $\mu Sh$.

    Let $\Lambda \subset T^{*,\infty}M$ be a subset where $T^{*,\infty}M$ is identified with the unit cotangent bundle. Then $\mu Sh_\Lambda$ is defined by $\mu Sh_\Lambda = \mu Sh_{\Lambda \times \mathbb{R}_{>0}}|_\Lambda$.
\end{definition}
\begin{remark}
    We define the sheafification in the (large) category of dg categories whose morphisms are exact functors. When $\widehat\Lambda$ is a conical subanalytic Lagrangian, the sheafification takes value in the (large) category of presentable dg categories whose morphisms are colimit preserving functors \cite{NadShen}*{Remark 6.1}.
\end{remark}

    Denote by $m_\Lambda$ the natural quotient functor on the sheaf of categories, which, on the level of global sections, induces
    $$m_\Lambda: \, Sh_\Lambda(M) \rightarrow \mu Sh_\Lambda(\Lambda).$$
    We call $m_\Lambda$ the microlocalization functor.

    The following lemma immediately follows from the identity $\Gamma(T^*M, \mu hom(\mathscr{F, G})) = Hom(\mathscr{F}, \mathscr{G})$ and the fact that $\mathrm{supp}(\mu hom(\mathscr{F, G})) \subset SS(\mathscr{F}) \cap SS(\mathscr{G})$ \cite{KS}*{Corollary 5.4.10}.

\begin{lemma}[\cite{NadWrapped}*{Remark 3.18}]\label{sheafmusheaf}
    Let $\widehat\Lambda \subset T^*M$ be a conical subanalytic Lagrangian. Then
    $$Sh_{M \cup \widehat\Lambda}(M) \xrightarrow{\sim} \mu Sh_{M \cup \widehat\Lambda}(M \cup \widehat\Lambda).$$
\end{lemma}
\begin{remark}
    Note that using the invariance of $\mu Sh$ under contact transformations \cite[Section 7.2]{KS} and \cite{NadShen}*{Lemma 6.3}, the right hand side only depends on the germ of $M \cup \widehat\Lambda$, and can be viewed as a sheaf of categories either in $M \cup \widehat\Lambda \subset T^*M$ or in some $T^{*,\infty}N$ through a Legendrian embedding $M \cup \widehat\Lambda \hookrightarrow T^{*,\infty}N$ (see also \cite{NadShen}*{Remark 8.25}).
\end{remark}

\begin{theorem}[\cite{Gui}*{Proposition 6.6 \& Lemma 6.7}, \cite{NadShen}*{Corollary 5.4}]\label{microstalk}
    Let $\widehat\Lambda \subset T^*M$ be a conical subanalytic Lagrangian. For a smooth point $p = (x, \xi) \in \Lambda \subset T^{*,\infty}M$, the stalk $\mu Sh_{\Lambda,p} \simeq \mathrm{Mod}(\Bbbk)$.
\end{theorem}

\begin{theorem}[Guillermou \cite{Gui}*{Theorem 11.5}]\label{Gui}
    Let $\Lambda \subset T^{*,\infty}M$ be a smooth Legendrian submanifold. Suppose the Maslov class $\mu(\Lambda) = 0$ and $\Lambda$ is relative spin, then as sheaves of categories
    $$\mu Sh_\Lambda \xrightarrow{\sim} Loc_\Lambda.$$
\end{theorem}

    We define the notion of microstalks, which defines the equivalence in Theorem \ref{microstalk}. Using that we are able to define simple sheaves and pure sheaves, or microlocal rank $r$ sheaves.

\begin{definition}
    Let $\Lambda \subset T^{*,\infty}M$ be a Legendrian submanifold. Suppose $\mu(\Lambda) = 0$ and $\Lambda$ is relative spin. For $p = (x, \xi) \in \Lambda$, the microstalk of $\mathscr{F} \in Sh(M)$ at $p$ is
    $$m_{\Lambda,p}(\mathscr{F}) = m_\Lambda(\mathscr{F})_p \in \mathrm{Mod}(\Bbbk).$$
    $\mathscr{F} \in Sh_\Lambda(M)$ is called microlocal rank $r$ if $m_{\Lambda,p}(\mathscr{F})$ is concentrated at a single degree with rank $r$. When $r = 1$ it is called simple, and in general it is called pure.
\end{definition}

    The microstalk of a sheaf can be computed explicitly as indicated by the following proposition.

\begin{proposition}[Guillermou \cite{Gui}*{Theorem 7.6~(iv), 7.9, 8.10 \& Lemma 11.4}]
    Let $\Lambda \subset T^{*,\infty}M$ be a smooth Legendrian. Suppose the Maslov class $\mu(\Lambda) = 0$ and $\Lambda$ is relative spin. When the front projection of $\Lambda$ onto $M$ is a smooth hypersurface near $p$ and $\varphi \in C^1(M)$ is a local defining function for $\Lambda$, then
    $$m_{\Lambda,p}(\mathscr{F}) = \Gamma_{\varphi \geq 0}(\mathscr{F})_x[-d(p)],$$
    where $d(p) \in \mathbb{Z}$ is called the Maslov potential.
\end{proposition}

\begin{example}\label{microstalk-cone}
    Suppose $\Lambda = \nu^{*,\infty}_{\mathbb{R}^n \times (0,+\infty),-}\mathbb{R}^{n+1} \subset T^{*,\infty}\mathbb{R}^{n+1}$ (which is the inward conormal bundle of $\mathbb{R}^n \times (0,+\infty)$) and $\mathscr{F} \in Sh_\Lambda(\mathbb{R}^{n+1})$. Then $\mathscr{F}$ is determined by
    \[\xymatrix{
    F_- & F_+ \ar[l] \ar[r]^\sim & F_+
    }\]
    Then for $p = (0,\dots,0, 0; 0, \dots, 0, 1) \in \Lambda$ we can pick $\varphi(x) = x_{n+1}$, and get
    $$\Gamma_{\varphi \geq 0}(\mathscr{F})_{(0,...,0)} = \mathrm{Cone}(F_+ \rightarrow F_-)[-1] \simeq \mathrm{Tot}(F_+ \rightarrow F_-).$$
    One can see that the definition of the microstalk coincides with the definition of the microlocal monodromy defined by Shende-Treumann-Zaslow \cite[Section 5.1]{STZ}, and indeed
    $$m_{\Lambda,p}(\mathscr{F}) \simeq \mu mon(\mathscr{F})_p[-1].$$
\end{example}

\begin{proposition}[\cite{Guisurvey}*{Equation (1.4.4)}]
    Let $\Lambda \subset T^{*,\infty}M$ be a Legendrian submanifold. $\mathscr{F} \in Sh_\Lambda(M)$ is microlocal rank $r$ at $p \in \Lambda$ iff
    $$\mu hom(\mathscr{F, F})_p \simeq \Bbbk^{r^2}.$$
\end{proposition}

%    Finally we recall the famous Sato's exact triangle.
%
%\begin{theorem}[Sato's exact triangle, \cite{Gui}*{Equation 2.17}, \cite{Guisurvey}*{Equation 1.3.5}]\label{sato}
%    Let $\mathscr{F} \in Sh(M)$ be a weakly constructible sheaf with perfect stalk. Then there is an exact triangle
%    $$\mathscr{H}om(\mathscr{F}, \Bbbk_M) \otimes \mathscr{G} \rightarrow \mathscr{H}om(\mathscr{F, G}) \rightarrow \pi_*(\mu hom(\mathscr{F, G})|_{T^{*,\infty}M}) \xrightarrow{+1}.$$
%\end{theorem}
%
\subsection{Functors for Hamiltonian Isotopies}
    We review the equivalence functors coming from a Hamiltonian isotopy, constructed for sheaves $Sh(M)$ by Guillermou-Kashiwara-Schapira \cite{GKS}, and for microlocal sheaves $\mu Sh_\Lambda(\Lambda)$ by Kashiwara-Schapira \cite[Section 7.2]{KS}.

\begin{definition}\label{lagmovie}
    Let $\widehat H_s: T^*M \times I \rightarrow \mathbb{R}$ be a homogeneous Hamiltonian on $T^*M$, and $H_s = \widehat H_s|_{T^{*,\infty}M}$ the corresponding contact Hamiltonian on $T^{*,\infty}M$.
%    Then the Lagrangian graph of the Hamiltonian isotopy $\varphi_{\widehat H}^t\,(t \in I)$ is
%    $$\mathrm{Graph}_{\widehat H} = \{(x, \xi, x', \xi', t, \tau) | (x', \xi') = \varphi_{\widehat H}^t(x, \xi), \tau = -\widehat H_t \circ \varphi_{\widehat H}^t(x, \xi)\} \subset T^*(M \times M \times I).$$
    For a conical Lagrangian $\widehat\Lambda$, the Lagrangian movie of $\widehat\Lambda$ under the Hamiltonian isotopy $\varphi_{\widehat H}^s\,(s \in I)$ is
    $$\widehat\Lambda_{\widehat H} = \{(x, \xi, s, \sigma) | (x, \xi) = \varphi_{\widehat H}^s(x_0, \xi_0), \sigma = -\widehat H_s \circ \varphi_{\widehat H}^s(x_0, \xi_0), (x_0, \xi_0) \in \widehat \Lambda\} \subset T^*(M \times I).$$
    For a Legendrian $\Lambda$, the Legendrian movie of $\Lambda$ under the corresponding contact Hamiltonian isotopy is $\Lambda_H = \widehat\Lambda_{\widehat H} \cap T^{*,\infty}M$.
\end{definition}

%\begin{theorem}[Guillermou-Kashiwara-Schapira, \cite{GKS}]\label{GKSkernel}
%    Let $\widehat H_t: T^*M \times I \rightarrow T^*M$ be a homogeneous Hamiltonian on $T^*M$. Then there is a unique sheaf $\mathscr{K} \in Sh^b(M \times M \times I)$ such that $\mathscr{K}|_{t=0} = \Bbbk_\Delta$, and
%    $$SS(\mathscr{K}) \backslash (M \times M \times I) = \mathrm{Graph}_{\widehat H}.$$
%\end{theorem}

\begin{theorem}[Guillermou-Kashiwara-Schapira \cite{GKS}*{Proposition 3.12}]\label{GKS}
    Let $H_s: T^{*,\infty}M \times I \rightarrow \mathbb{R}$ be a contact Hamiltonian on $T^{*,\infty}M$ and $\Lambda$ a Legendrian in $T^{*,\infty}M$. Then there are equivalences
    $$Sh_{\Lambda}(M) \xleftarrow{\sim} Sh_{\Lambda_{H}}(M \times I) \xrightarrow{\sim} Sh_{\varphi_{H}^1(\Lambda)}(M),$$
    given by restriction functors $i_0^{-1}$ and $i_1^{-1}$ where $i_s: M \times s \hookrightarrow M \times I$ is the inclusion. We denote their inverses by $\Psi_H^0$ and $\Psi_H^1$, and $\Psi_H = i_1^{-1} \circ \Psi_H^0$.
\end{theorem}
\begin{remark}[\cite{GKS}*{Remark 3.9}]\label{GKSpara}
    This theorem also works for a $U$-parametric family of Hamiltonian isotopies on $T^{*,\infty}M \times U \rightarrow T^{*,\infty}M$ for a contractible manifold $U$.
\end{remark}

    For the category of microlocal sheaves $\mu Sh_\Lambda(\Lambda)$, Kashiwara-Schapira \cite[Theorem 7.2.1]{KS} showed that it is invariant under contact transformations, which are just (local) contactomorphisms. Nadler-Shende explained how this will imply the invariance of $\mu Sh_\Lambda(\Lambda)$ under (global) Hamiltonian isotopies.

\begin{theorem}[Nadler-Shende \cite{NadShen}*{Lemma 6.6}]\label{cont-trans}
    Let $H_s: T^{*,\infty}M \times I \rightarrow \mathbb{R}$ be a contact Hamiltonian on $T^{*,\infty}M$ and $\Lambda$ a Legendrian in $T^{*,\infty}M$. Then there are equivalences
    $$\mu Sh_\Lambda(\Lambda) \xleftarrow{\sim} \mu Sh_{\Lambda_H}(\Lambda_H) \xrightarrow{\sim} \mu Sh_{\varphi_H^1(\Lambda)}(\varphi_H^1(\Lambda))$$
    given by restriction functors $i_0^{-1}$ and $i_1^{-1}$ where $i_s: M \times s \hookrightarrow M \times I$ is the inclusion. We denote their inverses by $\Psi_H^0$ and $\Psi_H^1$, and $\Psi_H = i_1^{-1} \circ \Psi_H^0$.
\end{theorem}
\begin{proof}
    For any open subset $\Omega \subset T^{*,\infty}M$, consider the contact movie $\Omega_{H,s,\epsilon} \subset T^{*,\infty}(M \times I)$ in the time interval $I_{s,\epsilon} = (s - \epsilon, s + \epsilon)$. Then $i_s^{-1}$ induces equivalences of categories
    $$Sh_{\Lambda_H \cup \,\Omega_{H,s,\epsilon}^c}(M \times I_{s,\epsilon}) \xrightarrow{\sim} Sh_{\varphi_H^s(\Lambda \cup \Omega^c)}(M), \,\,\,
    Sh_{\,\Omega_{H,s,\epsilon}^c}(M \times I_{s,\epsilon}) \xrightarrow{\sim} Sh_{\varphi_H^s(\Omega^c)}(M).$$
    Since $Sh(M \times I_{s,\epsilon}) = Sh(M \times I)/Sh_{T^*(M \times I \backslash I_{s,\epsilon})}(M \times I)$, we get an equivalence of presheaves
    $$i_s^{-1}: \varinjlim_{\epsilon \rightarrow 0}\,\mu Sh_{\Lambda_H}^\text{pre}(\Omega_{H,s,\epsilon}) \xrightarrow{\sim} \mu Sh^\text{pre}_{\varphi_H^s(\Lambda)}(\varphi_H^s(\Omega)),$$
    where the left hand side is the pull back of a presheaf, as $\Lambda_H \cap \Omega_{H,s,\epsilon}\,(\epsilon > 0)$ form a relative neighbourhood basis of $\varphi_H^s(\Lambda \cap \Omega)$. Therefore, after sheafification, we can get an equivalence given by the pull back
    $$i_s^{-1}: \mu Sh_{\Lambda_H}(\varphi_H^s(\Lambda)) \xrightarrow{\sim} \mu Sh_{\varphi_H^s(\Lambda)}(\varphi_H^s(\Lambda)).$$
    Then, since $\mu Sh_{\Lambda_H}^\text{pre}(\Omega_{H,s,\epsilon}) \simeq \mu Sh_{\Lambda_H}^\text{pre}(\Omega_{H,s',\epsilon})$, we also know that $\mu Sh_{\Lambda_H}^\text{pre}$ forms a presheaf that is locally constant in the $I$ direction (along contact movies of points). Since $I$ is contractible, we can conclude that there is an equivalence given by the restriction
    $$\mu Sh_{\Lambda_H}(\Lambda_H) \xrightarrow{\sim} \mu Sh_{\varphi_H^s(\Lambda)}(\varphi_H^s(\Lambda)).$$
    This completes the proof of the theorem.
\end{proof}
\begin{remark}\label{cont-trans-para}
    One can show that the theorem also works for a $U$-parametric family of Hamiltonian isotopies for a contractible manifold $U$, following Remark \ref{GKSpara}.
\end{remark}
\begin{remark}\label{GKSviaCont}
    From our proof, one may notice that there is a commutative diagram
    \[\xymatrix{
    Sh_{\Lambda_{H}}(M \times I) \ar[r]^{\hspace{7pt}i_s^{-1}} \ar[d] & Sh_{\varphi_{H}^s(\Lambda)}(M) \ar[d] \\
    \mu Sh_{\Lambda_H}(\Lambda_H) \ar[r]^{i_s^{-1}\hspace{12pt}} & \mu Sh_{\varphi_H^s(\Lambda)}(\varphi_H^s(\Lambda)).
    }\]
%    Note that we need to condition that $\Lambda_H \cap \pi_I^*T^{*,\infty}I = \emptyset$, (i.e.~$\Lambda_H$ is non-characterstic with respect to $i_t$) to lift the functors on the base to the cotangent bundle.
\end{remark}
%\begin{proof}
%    There is a natural transformation since we have a natural morphism
%    $$\Gamma(\Lambda, i_t^{-1}\mu hom(\mathscr{F, G})) \rightarrow \Gamma(\Lambda, \mu hom(i_t^{-1}\mathscr{F}, i_t^{-1}\mathscr{G})).$$
%    It suffices to check that the stalks are isomorphic.
%\end{proof}

\subsection{Sheaf Quantization of Weinstein Manifolds}\label{sheafquan}
    In this section we state the series of results by Nadler-Shende \cite{NadShen}, which will be the main ingredients of our constructions.

    Basically they are able to embed the Weinstein manifold $(X, d\lambda)$ into the contact boundary of some cotangent bundle and thus construct a microlocal sheaf category $\mu Sh_{\mathfrak{c}_X}(\mathfrak{c}_X)$ from the Lagrangian skeleton $\mathfrak{c}_X$ of $X$. Moreover, they are able to construct functors with respect to Liouville inclusions and homotopies that are fully faithful.

    First of all, let us recall their construction of the microlocal sheaf category $\mu Sh_{\mathfrak{c}_X}(\mathfrak{c}_X)$ for any Weinstein manifold $X$ with subanalytic skeleton $\mathfrak{c}_X$ (\cite[Section 8]{NadShen}).

\begin{remark}
    It is explained in \cite[Section 7.7]{GPS3} how to make the Lagrangian skeleton $\mathfrak{c}_X$ of a Weinstein manifold $X$ subanalytic. Namely any Weinstein manifold admits some Weinstein structure with a subanalytic skeleton.
\end{remark}

    Gromov's $h$-principle \cite{EliMisha}*{Theorem 12.3.1} enables us to embed the contactization of the Weinstein manifold $(X \times \mathbb{R}, \ker(dt - \lambda))$ into the contact boundary of a higher dimensional cotangent bundle $T^{*,\infty}N$, as long as (1)~$\dim T^*N \geq \dim X + 2$ and (2)~there is a bundle map $\Psi_s: TX \times T\mathbb{R} \rightarrow T(T^{*,\infty}N)$ covering a smooth embedding $f: X \times \mathbb{R} \hookrightarrow T^{*,\infty}N$ such that $\Psi_0 = df$ and $\Psi_1|_{TX \times \mathbb{R}}$ is a symplectic bundle map into the contact distribution $\xi_{T^{*,\infty}N}$. The second condition is purely algebraic topological. For example, $N = \mathbb{R}^m$ for sufficiently large $m$, this is satisfied as long as $X$ is stably polarizable \cite{Shendehprinciple}.

    %Once the embedding is chosen, the subanalytic Lagrangian skeleton $\mathfrak{c}_X$ is a then singular isotropic subset in $T^{*,\infty}N$.

    Consider the symplectic normal bundle $\nu_{X \times \mathbb{R}}(T^{*,\infty}N)$ of $X \times \mathbb{R} \hookrightarrow T^{*,\infty}N$. Assume that by choosing $\dim T^*N > 0$ to be sufficiently large, we can find a Lagrangian subbundle $(X \times \mathbb{R})_\sigma \subset \nu_{X \times \mathbb{R}}(T^{*,\infty}N)$ by choosing a section $\sigma$ of the Lagrangian Grassmannian of the normal bundle $\nu_{X \times \mathbb{R}}(T^{*,\infty}N)$, as in \cite[Lemma 9.1]{NadShen}. This is a null homotopy of
    $$X \times \mathbb{R} \rightarrow BU \rightarrow BLGr$$
    (where $BLGr$ is the classifying space of the stable Lagrangian Grassmannian). Let the Legendrian thickening of $\mathfrak{c}_X$ be
    $$\mathfrak{c}_{X,\sigma} = (X \times \mathbb{R})_\sigma|_{\mathfrak{c}_X \times \{0\}}.$$

\begin{definition}
    The microlocal sheaf category on a Weinstein manifold $X$, with a chosen section $\sigma$ in the stable Lagrangian Grassmannian, is defined by
    $$\mu Sh_{\mathfrak{c}_X} = \mu Sh_{\mathfrak{c}_{X,\sigma}}|_{\mathfrak{c}_X \times \{0\}}.$$
\end{definition}
\begin{remark}
    Nadler-Shende showed that this microlocal sheaf category is independent of the Lagrangian skeleton and the contact embedding we choose. It does depend on the thickening because that is determined by the section in Lagrangian Grassmannian.
\end{remark}
\begin{remark}\label{maslovdata}
    More generally, the existence of a section in the stable Lagrangian Grassmannian can be relaxed to simply the existence of a section $\sigma: X \times \mathbb{R} \rightarrow B\mathrm{Pic}(\Bbbk)|_{X \times \mathbb{R}}$, which is classified by Maslov data \cite[Definition 10.6]{NadShen}, i.e.~a null homotopy
    $$X \times \mathbb{R} \rightarrow B^2\mathrm{Pic}(\Bbbk),$$
    and the microlocal sheaf category can be defined by $\sigma^{-1}\mu Sh_{B\mathrm{Pic}(\Bbbk)|_{\mathfrak{c}_X}}$. Existence of Maslov data for ring spectrum coefficients are carefully studied by Jin \cite{JinJhomo} and \cite{NadShen}*{Section 11}. When $\Bbbk$ is a ring, this is ensured as long as $2c_1(X) = 0$.

    Therefore, from now on we will always assume the existence of a section in the Lagrangian Grassmannian of the stable normal bundle without loss of generality.
\end{remark}

    Given a Weinstein subdomain $U \subset X$ equipped with Maslov data, let $\lambda' = \lambda - df$ be the Liouville form on $X$ such that the Liouville flow $Z_{\lambda'}$ is transverse to $\partial_\infty U$, and $\mathfrak{c}_U$ the skeleton of $U$ under the Liouville flow $Z_{\lambda'}$. Then the primitive $f|_U: U \rightarrow \mathbb{R}$ determines the Legendrian lift of the skeleton $\mathfrak{c}_U$ in $X \times \mathbb{R}$ being $\widetilde{\mathfrak{c}}_U = \{(x, f(x)) | x \in \mathfrak{c}_U\}$. Define
    $$\mu Sh_{\widetilde{\mathfrak{c}}_U} = \mu Sh_{\widetilde{\mathfrak{c}}_{U,\sigma}}|_{\widetilde{\mathfrak{c}}_U}.$$
    In particular, when $U = T^*L$ is a Weinstein subdomain, we write $\widetilde{L}$ for the Legendrian lift of $L$ and consider $\mu Sh_{\widetilde{L}}$. It will be natural to construct an embedding functor
    $$\mu Sh_{\widetilde{\mathfrak{c}}_U}(\widetilde{\mathfrak{c}}_U) \longrightarrow \mu Sh_{\mathfrak{c}_X}(\mathfrak{c}_X).$$
    Nadler-Shende's main result is about constructing such an embedding functor and proving its full faithfulness. When $U = T^*L$, this realizes exact Lagrangian submanifolds $L \subset X$ as objects in the microlocal sheaf category.

\begin{definition}[Nadler-Shende \cite{NadShen}*{Definition 2.9}]
    Let $\Lambda_\zeta, \Lambda_\zeta' \subset Y\,(\zeta \in \mathbb{R})$ be two families of subsets in a contact manifold. $\Lambda_\zeta, \Lambda_\zeta'$ are gapped if there exists $\epsilon>0$, so that for any $\zeta \in \mathbb{R}$ there are no Reeb chords connecting $\Lambda_\zeta$ with $\Lambda'_\zeta$ with length shorter than $\epsilon$.
\end{definition}

\begin{theorem}[Nadler-Shende \cite{NadShen}*{Theorem 8.3 \& 9.7}]\label{NadShen}
    Consider a subanalytic Legendrian $\Lambda_1 \subset X \times \mathbb{R}$, which is either compact or locally closed, relatively compact with cylindrical ends. Let $\varphi_H^\zeta: X \times \mathbb{R} \rightarrow X \times \mathbb{R}$ be a contact isotopy for $\zeta \in (0, 1]$ conical near the cylindrical ends. Let ${\Lambda}_H \subset X \times \mathbb{R} \times (0, 1]$ be the Legendrian movie of $\varphi_H^\zeta$ and $\overline{\Lambda}_H$ be the closure of ${\Lambda}_H$ in $X \times \mathbb{R} \times [0, 1]$. Let $\Lambda_0 = \overline{\Lambda}_H \cap (X \times \mathbb{R} \times \{0\}) \subset X \times \mathbb{R}$ be the set of limit points of $\varphi_H^\zeta(\Lambda_1)$ as $\zeta \rightarrow 0$.

    Assume that for some contact form on $X \times \mathbb{R}$, the family $\varphi_H^\zeta(\Lambda_1)\,(\zeta \in (0, 1])$ is self gapped. Then there is a fully faithful functor
    $$\mu Sh_{\Lambda_1}(\Lambda_1) \hookrightarrow \mu Sh_{\Lambda_0}(\Lambda_0).$$
\end{theorem}

    In particular, when $U \subset X$ is a Weinstein subdomain (with Liouville complement), consider the Liouville vector field $Z_\lambda$ on $(X, d\lambda)$ defined by
    $$\iota(Z_\lambda)d\lambda = \lambda.$$
    The Liouville flow of $Z_\lambda$ for negative time will compress $\mathfrak{c}_U$ onto $\mathfrak{c}_X$ as $z \rightarrow -\infty$. The Liouville flow on $X$ extends to a contact flow $\varphi_Z^z$ in $X \times \mathbb{R}$ with
    $$d\varphi_Z^z/dz = t\partial/\partial t + Z_\lambda,$$
    and thus we can consider the Legendrian movie of $\widetilde{\mathfrak{c}}_U$ under the flow. The theorem then gives a fully faithful embedding of microlocal sheaves on $\widetilde{\mathfrak{c}}_U$ to sheaves on $\lim_{z \rightarrow -\infty}\varphi_Z^z(\widetilde{\mathfrak{c}}_U) \subset \mathfrak{c}_X \times \{0\}$. Write $\phi_Z^\zeta = \varphi_Z^{\ln z}$. Applying the flow $\varphi_Z^z\,(z\in (-\infty,0])$ or $\phi_Z^\zeta\,(\zeta \in (0,1])$, we have \cite{NadShen}*{Section 8.2}
    \[\begin{split}\mu Sh_{\mathfrak{c}_U}(\mathfrak{c}_U) & \hookrightarrow \mu Sh_{\lim_{z \rightarrow -\infty}\varphi_Z^z(\widetilde{\mathfrak{c}}_U)}(\lim{}_{z \rightarrow -\infty}\varphi_Z^z(\widetilde{\mathfrak{c}}_U)) \\
    &\xrightarrow{\sim} \mu Sh_{\lim_{\zeta \rightarrow 0}\phi_Z^\zeta(\widetilde{\mathfrak{c}}_U)}(\lim{}_{\zeta \rightarrow 0}\,\phi_Z^\zeta(\widetilde{\mathfrak{c}}_U)) \hookrightarrow \mu Sh_{\mathfrak{c}_X}(\mathfrak{c}_X)
    \end{split}\]

    For the proof of the theorem, consider a contact embedding $X \times \mathbb{R} \hookrightarrow T^{*,\infty}N$ and choose a Lagrangian section $(X \times \mathbb{R})_\sigma \subset \nu_{X \times \mathbb{R}}(T^{*,\infty}N)$. One can pull back the contact form and the contact isotopy via the projection $\nu_{X \times \mathbb{R}}(T^{*,\infty}N) \rightarrow X \times \mathbb{R}$. Then $\varphi_H^\zeta(\Lambda_{1,\sigma})\,(\zeta \in (0, 1])$ is self gapped iff $\varphi_H^\zeta(\Lambda_1)\,(\zeta \in (0, 1])$ is. Hence one can replace $X \times \mathbb{R}$ in the theorem by $T^{*,\infty}N$.

    The proof consists of two steps. First, we need to construct a fully faithful embedding from $\mu Sh_\Lambda(\Lambda)$ back to $Sh(N)$ where we have Grothendieck's six functors; second, we need to construct a fully faithful functor between subcategories of $Sh(N)$.

    Here is the first step, called antimicrolocalization. Similar constructions for $\Lambda \subset J^1(M)$ with the standard Reeb flow have been obtained by Guillermou \cite{Gui}*{Section 13-15}. In wrapped Fukaya categories, this is called the stop doubling construction \cite[Example 8.7]{GPS2}.

\begin{definition}
    Let $\Lambda \subset T^{*,\infty}N$ be a subanalytic Legendrian with cylindrical end $\partial\Lambda$, i.e.~a contact embedding
    $$(T^{*}(U \times (-1, 1)) \times \mathbb{R}, \partial\Lambda \times [0,1)) \hookrightarrow (T^{*,\infty}N, \Lambda).$$
    Let $\varphi_s\,(s \in \mathbb{R})$ be some Reeb flow on $T^{*,\infty}N$. For $\partial\Lambda_{\pm s} \times [0, 1) \subset T^{*}(U \times (-1, 1)) \times \mathbb{R}$, connect the ends $\partial \Lambda_{\pm s}$ by a family of standard cusps $\partial\Lambda\, \times \prec$. Then
    $$(\Lambda, \partial\Lambda)^\prec_s = \Lambda_{-s} \cup \Lambda_s \cup (\partial\Lambda \,\times \prec).$$
\end{definition}

\begin{theorem}[Nadler-Shende \cite{NadShen}*{Theorem 7.28}]\label{antimicro}
    Let $\Lambda \subset T^{*,\infty}N$ be a subanalytic Legendrian, which is either compact or locally closed, relatively compact with cylindrical ends. Let $c$ be the shortest length of Reeb orbits starting and ending on $\Lambda$. For $\epsilon < c/2$, the microlocalization functor
    $$Sh_{(\Lambda,\partial\Lambda)^\prec_\epsilon}(N)_0 \rightarrow \mu Sh_{\Lambda_{-\epsilon}}(\Lambda_{-\epsilon})$$
    admits a right inverse. Here the subscript $0$ means the subcategory of objects with $0$ stalk away from a compact set.
\end{theorem}

    By applying the antimicrolocalization functor, we now only need to construct a functor in $Sh(N)$. Namely we consider the nearby cycle functor and show that it is fully faithful in our setting. This full faithfulness criterion is proposed by Nadler \cite{NadNonchar} and proved for families of Legendrians by Zhou \cite{Zhou}*{Proposition 3.2}.

\begin{definition}\label{relativess}
    For a fibration $\pi: N \times B \rightarrow B$, let the projection of the cotangent bundle to the fiber be $\Pi: T^*(N \times B) \rightarrow T^*(N \times B) / \pi^*T^*B$. For $\mathscr{F} \in Sh(N \times B)$, the singular support relative to $\pi$ is
    $$SS_\pi(\mathscr{F}) = \overline{\Pi(SS(\mathscr{F}))}.$$
\end{definition}

\begin{theorem}[Nadler-Shende \cite{NadShen}*{Theorem 5.1}]\label{nearby}
    Let $\mathscr{F}, \mathscr{G}$ be weakly constructible sheaves on $N \times [-1, 0) \cup (0, 1]$. Write $j: N \times [-1, 0) \cup (0, 1] \rightarrow N \times [-1, 1]$ and $i: N \times \{0\} \rightarrow N \times [-1, 0) \cup (0, 1]$. Suppose
\begin{enumerate}
  \item $SS^\infty(\mathscr{F}), SS^\infty(\mathscr{G}) \cap \pi_{\mathbb{R}}^*T^*([-1, 0) \cup (0, 1]) = \varnothing$;
  \item The family of pairs $SS^\infty_\pi(\mathscr{F}), SS^\infty_\pi(\mathscr{G})$ are gapped for some contact form.
\end{enumerate}
    Then we have a natural isomorphism
    $$\Gamma(i^{-1}\mathscr{H}om(j_*\mathscr{F}, j_*\mathscr{G})) \xrightarrow{\sim} Hom(i^{-1}j_*\mathscr{F}, i^{-1}j_*\mathscr{G}).$$
\end{theorem}

    Finally, instead of considering the whole category $Sh(N)$, we need to restrict to the subcategories $Sh_{(\Lambda_1, \partial\Lambda_1)_\epsilon^\prec}(N)$ and $Sh_{(\Lambda_0, \partial\Lambda_0)_\epsilon^\prec}(N)$. Therefore we need the following estimate, which follows from Proposition \ref{sspushforward} and \ref{sspullback} \cite[Theorem 6.3.1 \& Corollary 6.4.4]{KS}.

\begin{lemma}[\cite{NadShen}*{Lemma 3.16}]\label{ssnearby}
    For $\mathscr{F} \in Sh(N \times (0, 1])$, denoting $j: N \times (0, 1] \rightarrow N \times [-1, 1]$ and $i: N \times \{0\} \rightarrow N \times [-1, 1]$,
    $$SS(i^{-1}j_*\mathscr{F}) \subset \overline{\Pi(SS(\mathscr{F}))} \cap T^*(N \times \{0\}).$$
\end{lemma}
%\begin{remark}\label{ssnearbyplus}
%    This follows from Proposition \ref{sspushforward} and \ref{sspullback} \cite[Theorem 6.3.1 \& Corollary 6.4.4]{KS}.
%\end{remark}

    Note that by Theorem \ref{cont-trans}, since $\Lambda_H$ is the Legendrian movie of $\Lambda_1$ under the flow $\varphi_H^\zeta\,(\zeta \in (0,1])$, we have a quasi-equivalence of categories
    $$\mu Sh_{\Lambda_1}(\Lambda_1) \simeq \mu Sh_{\Lambda_H}(\Lambda_H).$$
    Using Theorem \ref{antimicro}, \ref{nearby} together with Lemma \ref{ssnearby}, Theorem \ref{NadShen} now immediately follows.

\begin{remark}\label{rem:lim-colim}
    One may also notice that, when the contact isotopy is positive (resp.~negative), the nearby cycle functor in Theorems \ref{nearby} and \ref{NadShen} can be realized by a colimit (resp.~limit) \cite[Corollary 3.4 \& 3.8]{Kuo}. In the situation of Weinstein subdomain embeddings, by choosing a negative or positive primitive, one can ensure that the contact lift of the Liouville flow can be made both positive and negative. 
\end{remark}

\subsection{Various Microlocal Sheaf Categories}\label{variouscat}
    We have defined the sheaf of categories $\mu Sh_\Lambda$ consisting of microlocal sheaves with possibly unbounded and infinite rank cohomology. However, in general we are really dealing with either the sheaf category of compact objects or the one of proper objects. We explain how to restrict to these categories. Most of the discussions can be found in \cite{NadWrapped}*{Section 3.6 \& 3.8} and \cite{GPS3}*{Section 4.4 \& 4.5}.

    Throughout the discussion, we will be considering the microlocal sheaf category $\mu Sh_\Lambda$ on a subanalytic Legendrian (or conical subanalytic Lagrangian) subset.

\begin{definition}
    For $\mathscr{F} \in \mu Sh_\Lambda(U)$, we call it a compact object if the Yoneda module $\Gamma(U, \mu hom(\mathscr{F}, -))$ commutes with coproducts. $\mu Sh_\Lambda^c(U) \subset \mu Sh_\Lambda(U)$ is the full subcategory of compact objects.
\end{definition}

    We know that for a subanalytic Legendrian subset, $\mu Sh_\Lambda$ is a sheaf of compactly generated presentable categories (closed under limits and colimits), and in addition, for $V \subset U$, the restriction functor
    $$\rho_{UV}: \, \mu Sh_\Lambda(U) \rightarrow \mu Sh_\Lambda(V)$$
    preserves limits and colimits. Since it preserves limits, there is a left adjoint called the corestriction functor
    $$\rho^*_{UV}: \, \mu Sh_\Lambda(V) \rightarrow \mu Sh_\Lambda(U).$$
    Since $\rho_{UV}$ preserves colimits, $\rho^*_{UV}$ preserves compact objects. Hence the corestriction functor restricts to the subsheaf of category of compact objects
    $$\rho^*_{UV}: \, \mu Sh^c_\Lambda(V) \rightarrow \mu Sh^c_\Lambda(U).$$
    Note that $\mu Sh_{\Lambda \cap U}(U) = \mu Sh_\Lambda(U)$, so the corestriction functor is indeed a functor on global sections of categories $\mu Sh^c_{\Lambda \cap V}(V) \rightarrow \mu Sh^c_{\Lambda \cap U}(U).$

\begin{lemma}[Nadler \cite{NadWrapped}*{Theorem 3.16}]\label{cosheaf}
    $\mu Sh^c_\Lambda$ together with the corestriction functors form a cosheaf of dg categories.
\end{lemma}

    On the other hand, we can consider the subcategory of proper (pseudoperfect) objects.

\begin{definition}
    $\mu Sh^b_\Lambda(U)$ is the category of proper objects in $\mu Sh_\Lambda(U)$ defined by
    $$\mu Sh^b_\Lambda(U) = \mathrm{Fun}^\text{ex}(\mu Sh^c_\Lambda(U)^\text{op}, \mathrm{Perf}(\Bbbk)),$$
    where $\mathrm{Fun}^\text{ex}(-, -)$ is the dg category of exact functors.
\end{definition}

    Since $\mu Sh^c_\Lambda$ is a cosheaf of categories, we know that $\mu Sh^b_\Lambda$ is a sheaf of categories. The following theorem shows that $\mu Sh^b_\Lambda(U)$ is the equivalent to the subcategories of objects in $\mu Sh_\Lambda(U)$ with perfect stalks.

\begin{theorem}[Nadler \cite{NadWrapped}*{Theorem 3.21}]\label{perfcompact}
    The natural pairing $\Gamma(U, \mu hom(-, -))$ defines an equivalence between the category of proper objects $\mu Sh^b_\Lambda(U)$ and the full subcategory of $\mu Sh_\Lambda(U)$ of objects with perfect microstalks.
\end{theorem}

\section{Proof of the Main Results}\label{mainsec}

\subsection{Construction of cobordism functor}\label{mainthm}
    In this section we construct the Lagrangian cobordism and prove full faithfulness, which is the first part of Theorem \ref{main}. The proof here will be relatively concise, yet it still includes an outline of the constructions in Section \ref{sheafquan} and \ref{variouscat}. The reader may find more detailed explanation in those sections.

\begin{theorem}\label{exist}
    Let $X$ be a Weinstein manifold with subanalytic skeleton $\mathfrak{c}_X$, $\Lambda_-, \Lambda_+ \subset \partial_\infty X$ be Legendrian submanifolds, and $L \subset \partial_\infty X \times \mathbb{R}$ an exact Lagrangian cobordism from $\Lambda_-$ to $\Lambda_+$. There is a cobordism functor between the microlocal sheaf categories of compact objects
    $$\Phi_L^*: \, \mu Sh^c_{\mathfrak{c}_X \cup \Lambda_+ \times \mathbb{R}}(\mathfrak{c}_X \cup \Lambda_+ \times \mathbb{R}) \longrightarrow \mu Sh^c_{\mathfrak{c}_X \cup \Lambda_- \times \mathbb{R}}(\mathfrak{c}_X \cup \Lambda_- \times \mathbb{R}) \otimes_{Loc^c(\Lambda_-)} Loc^c(L),$$
    and a fully faithful adjoint functor between microlocal sheaf categories of proper objects
    $$\Phi_L: \, \mu Sh^b_{\mathfrak{c}_X \cup \Lambda_- \times \mathbb{R}}(\mathfrak{c}_X \cup \Lambda_- \times \mathbb{R}) \times_{Loc^b(\Lambda_-)} Loc^b(L) \hookrightarrow \mu Sh^b_{\mathfrak{c}_X \cup \Lambda_+ \times \mathbb{R}}(\mathfrak{c}_X \cup \Lambda_+ \times \mathbb{R}),$$
\end{theorem}
\begin{proof}
    Following Section \ref{sheafquan}, Gromov's $h$-principle \cite{EliMisha}*{Theorem 12.3.1} enables us to embed the contactization of the Weinstein manifold $X \times \mathbb{R}$ into the contact boundary of a higher dimensional cotangent bundle $T^{*,\infty}N$.

    Consider the symplectic normal bundle $\nu_{X \times \mathbb{R}}(T^{*,\infty}N)$ of $X \times \mathbb{R} \hookrightarrow T^{*,\infty}N$, and as in Remark \ref{maslovdata} we assume that there is a Lagrangian subbundle $(X \times \mathbb{R})_\sigma \subset \nu_{X \times \mathbb{R}}(T^{*,\infty}N)$ by choosing a section in the Lagrangian Grassmannian of the normal bundle $\nu_{X \times \mathbb{R}}(T^{*,\infty}N)$. Consider the subanalytic Lagrangian skeleta $\mathfrak{c}_X \cup \Lambda_\pm \times \mathbb{R}$ and the Legendrian lifts $(\mathfrak{c}_X \cup \Lambda_\pm \times \mathbb{R}) \times \{0\}$ in $X \times \mathbb{R}$. Let the microlocal sheaf category supported on $\mathfrak{c}_X \cup \Lambda_\pm \times \mathbb{R}$ be
    $$\mu Sh_{\mathfrak{c}_X \cup \Lambda_\pm \times \mathbb{R}} = \mu Sh_{(\mathfrak{c}_X \cup \Lambda_\pm \times \mathbb{R})_\sigma}|_{\mathfrak{c}_X \cup \Lambda_\pm \times \mathbb{R}}.$$
    It is determined by $X \times \mathbb{R}$ and the choice of a section in the stable Lagrangian Grassmannian and is independent of the contact embedding.

    Since $(X, d\lambda)$ is a Weinstein manifold, we may assume that $X \setminus \mathfrak{c}_X \cong \partial_\infty X \times (-\infty, +\infty)$ where the Liouville flow $Z_\lambda = e^r\partial/\partial r$. Suppose $L \cap \partial_\infty X \times (-\infty, -r_0] = \Lambda_- \times (-\infty, -r_0]$. Glue ${L} \cap \partial_\infty X \times [-r_0, +\infty)$ with $\Lambda_- \times (-\infty, -r_0]$ along $\Lambda_- \times \{-r_0\}$, and denote by $\Lambda_- \times \mathbb{R} \cup {L}$ their concatenation in $X$. Note that this is the same as $L$, but we use the notation to emphasize that we will view it as the union of two separate parts to apply the cosheaf property later.

    We can glue the Legendrian lift $\widetilde{L}$ of the Lagrangian $L$ to the skeleton $\mathfrak{c}_X \cup \Lambda_\pm \times \mathbb{R}$ in the contactization $X \times \mathbb{R}$. As the primitive of $L$ defined by $df_L = \lambda|_L$ is a constant when the $\mathbb{R}$ coordinate in $\partial_\infty X \times \mathbb{R}$ satisfies $r < -r_0$, we may assume that $f_L = 0$ when $r < -r_0$. The Legendrian lift of $L$ is defined by
    $$\widetilde{L} = \{(x, f_L(x)) | x \in L\} \subset X \times \mathbb{R}.$$
    Then we consider the sheaf of categories $\mu Sh_{\widetilde{L}}$. Since $\widetilde{L}$ coincides with $\Lambda_- \times \mathbb{R} \subset X \times \{0\}$ when $r < -r_0$, we can glue $\widetilde{L} \cap \partial_\infty X \times [-r_0, +\infty) \times \mathbb{R}$ with $\Lambda_- \times (-\infty, -r_0] \times \{0\}$, and get their concatenation in $X \times \mathbb{R}$. Denote it by $\Lambda_- \times \mathbb{R} \cup \widetilde{L}$. We can thus consider the sheaf and cosheaf of categories $\mu Sh_{\mathfrak{c}_X \cup \Lambda_- \times \mathbb{R} \cup \widetilde{L}}$.

    Since for any Lagrangian skeleton $\mathfrak{c}_X$, $\mu Sh_{\mathfrak{c}_X}$ is a sheaf and cosheaf of dg categories, we have a quasi-equivalence of sheaves
    $$\mu Sh_{\mathfrak{c}_X \cup \Lambda_- \times \mathbb{R} \cup \widetilde{L}}\big(\mathfrak{c}_X \cup \Lambda_- \times \mathbb{R} \cup \widetilde{L}\big) \xrightarrow{\sim} \mu Sh_{\mathfrak{c}_X \cup \Lambda_- \times \mathbb{R}}(\mathfrak{c}_X \cup \Lambda_- \times \mathbb{R}) \times_{\mu Sh_{\Lambda_-}(\Lambda_-)} \mu Sh_{\widetilde{L}}\big(\widetilde{L}\big).$$
    For a smooth Legendrian with Maslov data, we know by Theorem \ref{Gui} \cite[Theorem 11.5]{Gui} that $\mu Sh_{\Lambda_-}(\Lambda_-) \simeq Loc(\Lambda_-), \mu Sh_{\widetilde{L}}\big(\widetilde{L}\big) \simeq Loc(L)$. Hence we have a quasi-equivalence
    $$\mu Sh_{\mathfrak{c}_X \cup \Lambda_- \times \mathbb{R} \cup L}\big(\mathfrak{c}_X \cup \Lambda_- \times \mathbb{R} \cup \widetilde{L}\big) \xrightarrow{\sim} \mu Sh_{\mathfrak{c}_X \cup \Lambda_- \times \mathbb{R}}(\mathfrak{c}_X \cup \Lambda_- \times \mathbb{R}) \times_{Loc(\Lambda_-)} Loc(L).$$

    We construct an embedding functor (also explained in Section \ref{sheafquan} after Theorem \ref{NadShen})
    $$\mu Sh_{\mathfrak{c}_X \cup \Lambda_- \times \mathbb{R} \cup \widetilde{L}}\big(\mathfrak{c}_X \cup \Lambda_- \times \mathbb{R} \cup \widetilde{L}\big) \hookrightarrow \mu Sh_{\mathfrak{c}_X \cup \Lambda_+ \times \mathbb{R}}(\mathfrak{c}_X \cup \Lambda_+ \times \mathbb{R}).$$
    Consider the Liouville flow $\varphi_Z^z,\,z\in \mathbb{R}$, on $X$ for negative time, which will compress $\mathfrak{c}_X \cup \Lambda_- \times \mathbb{R} \cup \widetilde{L}$ onto $\mathfrak{c}_X \cup \Lambda_+ \times \mathbb{R}$ as $z \rightarrow -\infty$. The Liouville flow on $X$ extends to a contact Hamiltonian $\varphi_Z^z$ in $T^{*,\infty}N$ with
    $$d\varphi_Z^z/dz = t\partial/\partial t + Z_\lambda.$$
    Write $\phi_Z^\zeta = \varphi_Z^{\ln \zeta}$, and consider the Legendrian movie of $\mathfrak{c}_X \cup \Lambda_- \times \mathbb{R} \cup \widetilde{L}$ under the flow $\varphi_Z^z,\,z \in (-\infty,0]$, or $\phi_Z^\zeta,\,\zeta \in (0,1]$. Since $M \cup \Lambda_- \times \mathbb{R}$ is the Legendrian lift of a Lagrangian skeleton while $\widetilde{L}$ is the lift of an embedded Lagrangian, there are no self Reeb chords and the gapped condition automatically holds. By Theorem \ref{NadShen} \cite[Theorem 8.3]{NadShen}, the nearby cycle functor gives us a fully faithful embedding of microlocal sheaves on the Legendrian movie of $\mathfrak{c}_X \cup \Lambda_- \times \mathbb{R} \cup \widetilde{L}$ to sheaves on
    $$\lim_{z \rightarrow -\infty}\varphi_Z^z\big(\mathfrak{c}_X \cup \Lambda_- \times \mathbb{R} \cup \widetilde{L}\big) = \lim_{\zeta \rightarrow 0}\phi_Z^\zeta\big(\mathfrak{c}_X \cup \Lambda_- \times \mathbb{R} \cup \widetilde{L}\big) \subseteq \mathfrak{c}_X \cup \Lambda_+ \times \mathbb{R}.$$
    Thus we have a fully faithful embedding functor, and combining with the quasi-equivalence of microlocal sheaves this induces the functor
    $$\Phi_L: \, \mu Sh_{\mathfrak{c}_X \cup \Lambda_- \times \mathbb{R}}(\mathfrak{c}_X \cup \Lambda_- \times \mathbb{R}) \times_{Loc(\Lambda_-)} Loc(L) \hookrightarrow \mu Sh_{\mathfrak{c}_X \cup \Lambda_+ \times \mathbb{R}}(\mathfrak{c}_X \cup \Lambda_+ \times \mathbb{R}).$$

    When restricting to compact objects, for any Lagrangian skeleton $\mathfrak{c}$, $\mu Sh^c_{\mathfrak{c}}$ is a cosheaf of dg categories (Lemma \ref{cosheaf} \cite[Proposition 3.16]{NadWrapped}). Hence we get a quasi-equivalence
    $$\mu Sh_{\mathfrak{c}_X \cup \Lambda_- \times \mathbb{R} \cup L}^c\big(\mathfrak{c}_X \cup \Lambda_- \times \mathbb{R} \cup \widetilde{L}\big) \xrightarrow{\sim} \mu Sh_{\mathfrak{c}_X \cup \Lambda_- \times \mathbb{R}}^c(\mathfrak{c}_X \cup \Lambda_- \times \mathbb{R}) \otimes_{Loc^c(\Lambda_-)} Loc^c(L).$$
    Since $\Phi_L$ is realized by the nearby cycle functor of a positive and negative contact isotopy as in Remark \ref{rem:lim-colim}, it preserves limits and colimits. Hence there is a left adjoint functor that preserves compact objects
    $$\mu Sh_{\mathfrak{c}_X \cup \Lambda_+ \times \mathbb{R}}^c(\mathfrak{c}_X \cup \Lambda_+ \times \mathbb{R}) \longrightarrow \mu Sh_{\mathfrak{c}_X \cup \Lambda_- \times \mathbb{R} \cup \widetilde{L}}^c\big(\mathfrak{c}_X \cup \Lambda_- \times \mathbb{R} \cup \widetilde{L}\big).$$
    This proves the existence of the left adjoint $\Phi_L^*$ on the subcategories of compact objects
    $$\Phi_L^*: \, \mu Sh^c_{\mathfrak{c}_X \cup \Lambda_+ \times \mathbb{R}}(\mathfrak{c}_X \cup \Lambda_+ \times \mathbb{R}) \longrightarrow \mu Sh^c_{\mathfrak{c}_X \cup \Lambda_- \times \mathbb{R}}(\mathfrak{c}_X \cup \Lambda_- \times \mathbb{R}) \otimes_{Loc^c(\Lambda_-)} Loc^c(L).$$

    Finally, for Lagrangian sksleta $\mathfrak{c}$, by passing to (Theorem \ref{perfcompact} \cite{NadWrapped}*{Theorem 3.21})
    $$\mu Sh_\mathfrak{c}^b(\mathfrak{c}) = \mathrm{Fun}^\text{ex}(\mu Sh_\mathfrak{c}^c(\mathfrak{c})^\text{op}, \mathrm{Perf}(\Bbbk)),$$
    we get the second functor $\Phi_L$, which is just the restriction of the functor from (large) dg categories to the subcategories of proper objects. The full faithfulness of $\Phi_L$ follows from the full faithfulness of the embedding functor. This completes the proof. The special case when $X = T^*M$ follows from Lemma \ref{sheafmusheaf}.
\end{proof}
\begin{remark}
    The functor $\Phi_L$ can also be obtained in the setting of partially wrapped Fukaya categories. Indeed one can consider Weinstein manifolds with stops $(X, \Lambda_\pm)$ and view $T^*L$ as a Weinstein sector. First apply the cosheaf property of partially wrapped Fukaya categories \cite{GPS2}*{Theorem 1.27} to get
    $$\mathcal{W}(X, \Lambda_-) \otimes_{\mathcal{W}(T^*(\Lambda \times [-1, 1]))} \mathcal{W}(T^*L) \xrightarrow{\sim} \mathcal{W}(X \cup_{T^*(\Lambda \times [-1, 1])} T^*L)$$
    or in other words
    $$\mathcal{W}(X, \Lambda_-) \otimes_{Loc^c(\Lambda)} Loc^c(L) \xrightarrow{\sim} \mathcal{W}(X \cup_{T^*(\Lambda \times [-1, 1])} T^*L).$$
    Then one can view $X \cup_{T^*(\Lambda \times [-1, 1])} T^*L$ as a Liouville subsector of $(T^*X, \Lambda_+)$ (the compliment is a Liouville cobordism). Since $X \cup_{T^*(\Lambda \times [-1, 1])} T^*L$ is Weinstein, following \cite[Section 8.3]{GPS2} or \cite{SylvanOrlov} one can define a Viterbo restriction functor
    $$\mathcal{W}(X, \Lambda_+) \longrightarrow \mathcal{W}(X \cup_{T^*(\Lambda \times [-1, 1])} T^*L).$$
\end{remark}

\begin{remark}\label{sing-cob}
    In fact the main theorem works in more general settings, as long as the gapped condition in Definition \ref{NadShen} is satisfied. For example, when $i: L \hookrightarrow \partial_\infty X \times \mathbb{R}$ is an exact Lagrangian cobordism with vanishing action self intersection points, i.e.~for the primitive $i^*\lambda = df_L$, $f_L(x) = f_L(x')$ whenever $i(x) = i(x')$, then $L$ can be lifted to an immersed Legendrian with no Reeb chords and the theorem still holds. Similarly, when $\Lambda_\pm$ are subanalytic Legendrians and $L$ is the Lagrangian projection of a subanalytic Legendrian cobordism, the theorem still applies as long as the gapped condition holds.
\end{remark}

    Using the full faithfulness of $\Phi_L$ and the sheaf property, it is not hard to get all the exact triangles. The key tool is the following lemma.

\begin{lemma}
    Let $X$ be a Weinstein manifold with subanalytic skeleton $\mathfrak{c}_X$, and $\Lambda_-, \Lambda_+ \subset \partial_\infty X$ be Legendrian submanifolds. Let $L \subset \partial_\infty X \times \mathbb{R}$ be an exact Lagrangian cobordism from $\Lambda_-$ to $\Lambda_+$. Suppose there are sheaves $\mathscr{F}_-, \mathscr{G}_- \in \mu Sh^b_{\mathfrak{c}_X \cup \Lambda_- \times \mathbb{R}}(\mathfrak{c}_X \cup \Lambda_- \times \mathbb{R})$ which restrict to constant local systems along $\Lambda_-$, and their stalks at $\Lambda_-$ are $F, G$. Denoting by
    $$\mathscr{F}_+ = \Phi_L(\mathscr{F}_-), \,\,\, \mathscr{G}_+ = \Phi_L(\mathscr{G}_-),$$
    the images of $\mathscr{F}_-, \mathscr{G}_-$ glued with the constant local systems on $L$ with stalks $F$ and $G$, then there is a homotopy pullback diagram
    \[\xymatrix@R=8mm{
    \Gamma(\mu hom(\mathscr{F}_+, \mathscr{G}_+)) \ar[r] \ar[d] & \Gamma(\mu hom(\mathscr{F}_-, \mathscr{G}_-)) \ar[d] \\
    C^*(L; Hom(F, G)) \ar[r] & C^*(\Lambda_-; Hom(F, G)).
    }\]
\end{lemma}
\begin{proof}
    Denote by $\widetilde{\mathscr{F}}_+, \widetilde{\mathscr{G}}_+$ the sheaves in $\mu Sh^b_{\mathfrak{c}_X \cup \Lambda_- \times \mathbb{R} \cup \widetilde{L}}\big(\mathfrak{c}_X \cup \Lambda_- \times \mathbb{R} \cup \widetilde{L}\big)$ obtained by gluing $\mathscr{F}_-, \mathscr{G}_-$ by the constant sheaf on $L$ with stalk $F$ and $G$. Then by the sheaf property of $\mu Sh^b_{\mathfrak{c}}$ for a Lagrangian sksleton $\mathfrak{c}$, we have a pullback diagram
    \[\xymatrix@R=8mm{
    \Gamma\big(\mu hom(\widetilde{\mathscr{F}}_+, \widetilde{\mathscr{G}}_+)\big) \ar[r] \ar[d] & \Gamma(\mu hom(\mathscr{F}_-, \mathscr{G}_-)) \ar[d] \\
    C^*(L; Hom(F, G)) \ar[r] & C^*(\Lambda_-; Hom(F, G)).
    }\]
    By full faithfulness of $\Phi_L$, we know that
    $$\Gamma\big(\mu hom(\widetilde{\mathscr{F}}_+, \widetilde{\mathscr{G}}_+)\big) \xrightarrow{\sim} \Gamma(\mu hom(\mathscr{F}_+, \mathscr{G}_+)).$$
    This proves the assertion.
\end{proof}

\begin{proof}[Proof of Corollary \ref{MVsequence}]
    The result immediately follows from the lemma.
\end{proof}

\begin{proof}[Proof of Corollary \ref{Exactsequence}]
    Note that the map $C^*(L; Hom(F, G)) \rightarrow C^*(\Lambda_-; Hom(F, G))$ fits into an exact triangle
    $$C^*(L; Hom(F, G)) \rightarrow C^*(\Lambda_-; Hom(F, G)) \rightarrow C^*(L, \Lambda_-; Hom(F, G))[1] \xrightarrow{+1}.$$
    Since a pullback diagram preserves (co)fibers, this gives the exact sequence
    $$\Gamma(\mu hom(\widetilde{\mathscr{F}}_+, \widetilde{\mathscr{G}}_+)) \rightarrow \Gamma(\mu hom(\mathscr{F}_-, \mathscr{G}_-))\rightarrow C^*(L, \Lambda_-; Hom(F, G))[1] \xrightarrow{+1},$$
    and hence completes the proof.
\end{proof}

\subsection{Concatenation and Invariance}\label{conca-inv}
    In this section we show some fundamental properties of the Lagrangian cobordism functor. We prove the second part of Theorem \ref{main} that concatenations of cobordisms give rise to compositions of cobordism functors. We also prove the invariance under compactly supported Hamiltonian isotopies.

    This section is a little bit technical and is not related to the rest of the paper, so the reader may feel free to skip it.

\subsubsection{Base change formula for push forward}
    For the proof of compatibility the results on concatenations of Lagrangian cobordisms, we need the commutativity criterion of compositions of nearby cycle functors in for example \cite{NadNearby} or \cite{KochNearby,MaiNearby}, while for the proof of Hamiltonian invariance of Lagrangian cobordisms, we need the commutativity of nearby cycles functors and Hamiltonian isotopy functors. We extract the main technical lemma as follows, which is a base change formula that does not hold in general.

    Write the projection maps
    $$\pi_i: N \times [-1,1] \times [-1,1] \rightarrow [-1,1], \,\, (x, t_1, t_2) \mapsto t_i, \,\,(i = 1, 2)$$
    and let $\pi = \pi_1 \times \pi_2: N \times [-1, 1] \times [-1, 1] \rightarrow [-1, 1] \times [-1, 1]$. Write the inclusions
    \[\xymatrix{
    N \times \{0\} \times [-1, 0) \cup (0, 1] \ar[r]^{i} \ar[d]_{j} & N \times [-1, 1] \times [-1, 0) \cup (0, 1] \ar[d]^{\overline{j}} \\
    N \times \{0\} \times [-1, 1] \ar[r]^{\overline{j}} & N \times [-1, 1] \times [-1, 1].
    }\]
    In our applications, all the sheaves are supported in $N \times [0, 1] \times [0, 1]$, but considering $N \times [-1, 1] \times [-1, 1]$ makes the proof cleaner by avoiding singular support estimates on manifolds with boundary. For a subset $A \subseteq T^*(N \times [-1, 1] \times [-1, 1])$, recall the definition of $i^\#(A) \subseteq T^*(N \times [-1, 1])$ in Section \ref{prelim-ss}.

\begin{proposition}\label{basechange}
    Let $\mathscr{F} \in Sh(N \times [-1,1] \times [-1, 0) \cup (0, 1])$ be a sheaf such that
    \begin{enumerate}
      \item $i^\#SS^\infty(\mathscr{F}) \cap \pi_2^*T^{*,\infty}([-1, 0) \cup (0, 1]) = \varnothing$,
      \item $SS^\infty(\mathscr{F}) \cap \pi^*T^{*,\infty}([-1, 0) \cup (0, 1] \times [-1, 0) \cup (0, 1]) = \varnothing$,
      \item $\overline{SS_{\pi}^\infty(\mathscr{F})} \cap T^{*,\infty}N \times \{(0, 0)\}$ is a subanalytic Legendrian.
    \end{enumerate}
    Then there is a natural isomorphism of sheaves
    $$\overline{i}^{-1}\overline{j}_*\mathscr{F} \simeq j_*i^{-1}\mathscr{F}.$$
\end{proposition}
\begin{remark}
    For our applications, we always have the stronger condition $SS^\infty(\mathscr{F}) \cap \pi^*T^{*,\infty}([-1,1] \times [-1, 0) \cup (0, 1]) = \varnothing$, in which case $i^\#SS^\infty(\mathscr{F}) \cap \pi_2^*T^{*,\infty}([-1, 0) \cup (0, 1]) = \varnothing$ follows immediately. We choose to state a more general result without assuming that because in general for commutativity of nearby cycles, when $\mathscr{F}$ is the push forward of $\mathscr{F}_0 \in Sh(N \times [-1, 0) \cup (0, 1] \times [-1, 0) \cup (0, 1])$ it might be difficult to check the stronger condition.
\end{remark}
\begin{remark}
    We remark that Condition~(3) is essential (even for weakly constructible sheaves). The following example is due to an anonymous referee. Consider $N = \mathbb{R}$, $S = \{(x, t_1, t_2) | t_1 = xt_2\} \subset N \times [-1, 1] \times [-1, 0) \cup (0, 1]$ and $\mathscr{F} = \Bbbk_S$. Then Condition~(3) does not hold and one can check that the base change formula does not hold.
\end{remark}

    We have a natural morphism $\overline{i}^{-1}\overline{j}_*\mathscr{F} \rightarrow j_*i^{-1}\mathscr{F}$ by adjunction. Since the natural morphism induces quasi-isomorphisms on stalks on $N \times \{0\} \times [-1, 0) \cup (0, 1]$, it suffices to show that the it also induces quasi-isomorphisms on stalks on $N \times \{(0, 0)\}$.

    First we compute the stalks of $\overline{i}^{-1}\overline{j}_*\mathscr{F}$ at $(x, 0, 0)$. The following lemma is basically \cite[Corollary 4.4]{NadShen}. Let $U_x$ be an open ball around $x \in N$, $D_{(0,0)}(\epsilon) = (-\epsilon, \epsilon) \times (-\epsilon, \epsilon)$, $D_{(0,0)}^\circ(\epsilon) = (-\epsilon, \epsilon) \times (-\epsilon, -\delta) \cup (\delta, \epsilon)$, and respectively $\overline{U}_x$, $\overline{D}_{(0,0)}(\epsilon)$ and $\overline{D}_{(0,0)}^\circ(\epsilon)$ be their closures.

\begin{lemma}\label{stalk1}
    Let $\mathscr{F} \in Sh(N \times [-1,1] \times [-1, 0) \cup (0, 1])$ be a sheaf so that $i^\#SS^\infty(\mathscr{F}) \cap \pi_2^*T^{*,\infty}([-1, 0) \cup (0, 1]) = \varnothing$, $SS^\infty(\mathscr{F}) \cap \pi^*T^{*,\infty}([-1, 0) \cup (0, 1] \times [-1, 0) \cup (0, 1]) = \varnothing$, and $\overline{SS_{\pi}^\infty(\mathscr{F})} \cap T^{*,\infty}N \times \{(0, 0)\}$ is a subanalytic Legendrian. Then for any $x \in N$, $U_x \subset N$ a sufficiently small open neighbourhood and $\epsilon > 0$ sufficiently small,
    $$\overline{j}_*\mathscr{F}_{(x,0,0)} \simeq \Gamma\big(\overline{U}_{x} \times \overline{D}_{(0,0)}^\circ(\epsilon), \mathscr{F}\big).$$
\end{lemma}
\begin{proof}
    Since $\overline{SS_{\pi}^\infty(\mathscr{F})} \cap T^{*,\infty}N \times \{(0, 0)\}$ is a subanalytic Legendrian, for any sufficiently small neighbourhood $U_{x}$ of $x \in N$, we have
    $$\overline{SS_{\pi}^\infty(\mathscr{F})} \cap \nu^{*,\infty}_{U_{x},\pm}N \times \{(0, 0)\} = \varnothing$$
    by general position argument.

    First consider $N \times [-1, 0) \cup (0, 1] \times [-1, 0) \cup (0, 1]$. Since $SS^\infty(\mathscr{F}) \cap \pi^*T^{*,\infty}([-1, 0) \cup (0, 1] \times [-1, 0) \cup (0, 1]) = \varnothing$, we can get a projection to the relative singular support in the relative cotangent bundle $SS^\infty(\mathscr{F}) \rightarrow SS_{\pi}^\infty(\mathscr{F})$ on $N \times [-1, 0) \cup (0, 1] \times [-1, 0) \cup (0, 1]$. Hence nonzero covectors project to nonzero covectors, i.e.~we get a map
    $$SS^\infty(\mathscr{F}) \cap \nu^{*,\infty}_{U_{x} \times D_{(0,0)}(\epsilon),\pm}(N \times (0,1] \times (0,1]) \rightarrow SS_{\pi}^\infty(\mathscr{F}) \cap \nu^{*,\infty}_{U_{x},\pm}N \times D_{(0, 0)}(\epsilon).$$
    Then consider $N \times \{0\} \times [-1, 0) \cup (0, 1]$. Recall the definition of $i^\#SS^\infty(\mathscr{F})$ in Section \ref{prelim-ss}. Write $D_0(\epsilon) = (-\epsilon, \epsilon)$. Since $\nu^{*,\infty}_{U_{x} \times \{0\} \times D_{0}(\epsilon),\pm}(N \times \{0\} \times [-1, 0) \cup (0, 1])$ only consists of covectors tangent to $N \times \{0\} \times [-1, 0) \cup (0, 1]$, we have an inclusion
    \begin{align*}
    SS^\infty(\mathscr{F}) &\cap \nu^{*,\infty}_{U_{x} \times \{0\} \times D_{0}(\epsilon),\pm}(N \times \{0\} \times [-1, 0) \cup (0, 1]) \\
    &\hookrightarrow i^\#SS^\infty(\mathscr{F}) \cap \nu^{*,\infty}_{U_{x} \times D_{0}(\epsilon),\pm}(N \times [-1, 0) \cup (0, 1]),
    \end{align*}
    Then by the assumption $i^\#SS^\infty(\mathscr{F}) \cap \pi_2^*T^{*,\infty}(0,1] = \varnothing$, we similarly get a map
    \begin{align*}
    i^\#SS^\infty(\mathscr{F}) & \cap \nu^{*,\infty}_{U_{x} \times D_{0}(\epsilon),\pm}(N \times [-1,0) \cup (0,1]) \\
    & \rightarrow SS_{\pi_2}^\infty(\mathscr{F}) \cap \nu^{*,\infty}_{U_x \times D_0(\epsilon), \pm}(N \times [-1,0) \cup (0,1]) \\
    & \rightarrow SS_{\pi}^\infty(\mathscr{F}) \cap \nu^{*,\infty}_{U_{x},\pm}N \times \{0\} \times D_{0}(\epsilon).
    \end{align*}
    Combining the two cases of $N \times [-1, 0) \cup (0, 1] \times [-1, 0) \cup (0, 1]$ and $N \times \{0\} \times [-1, 0) \cup (0, 1]$, we obtain a projection map
    $$SS^\infty(\mathscr{F}) \cap \nu^{*,\infty}_{U_{x} \times D_{(0,0)}(\epsilon),\pm}(N \times [-1,1] \times [-1, 0) \cup (0, 1]) \rightarrow SS_{\pi}^\infty(\mathscr{F}) \cap \nu^{*,\infty}_{U_{x},\pm}N \times D_{(0, 0)}(\epsilon).$$

    We claim that the right hand side is empty when $\epsilon > 0$ is sufficiently small. Otherwise, we can define a sequence in $SS_{\pi}^\infty(\mathscr{F}) \cap \nu^{*,\infty}_{U_{x},\pm}N \times D_{(0, 0)}(\epsilon)$ that converges to $\overline{SS_{\pi}^\infty(\mathscr{F})} \cap \nu^{*,\infty}_{U_{x},\pm}N \times \{(0, 0)\}$ as $\epsilon \rightarrow 0$. However, since $\overline{SS_{\pi}^\infty(\mathscr{F})} \cap \nu^{*,\infty}_{U_{x},\pm}N \times \{(0, 0)\} = \varnothing$, there are not any such sequences in the intersection of the relative singular support and unit conormal bundles. Therefore, when $\epsilon > 0$ is sufficiently small,
    $$SS_{\pi}^\infty(\mathscr{F}) \cap \nu^{*,\infty}_{U_{x},\pm}N \times D_{(0, 0)}(\epsilon) = \varnothing.$$
    Hence by the projection map we can conclude that for sufficiently small $\epsilon > 0$,
    $$SS^\infty(\mathscr{F}) \cap \nu^{*,\infty}_{U_{x} \times D_{(0,0)}(\epsilon),\pm}(N \times [-1,1] \times [-1, 0) \cup (0, 1]) = \varnothing.$$

    Consequently, by non-characteristic deformation lemma Proposition \ref{morselemma} applied to the family $U_x \times D_{(0,0)}(\epsilon)$ and $U_x \times D_{(0,0)}^\circ(\epsilon)$ for sufficiently small $\epsilon > 0$ and $\delta \ll \epsilon$, we can conclude that
    \begin{align*}
    \overline{j}_*\mathscr{F}_{(x,0,0)} &\simeq \Gamma\big(U_{x} \times D_{(0,0)}(\epsilon), \overline{j}_*\mathscr{F}\big) \simeq \Gamma\big(\overline{U}_{x} \times \overline{D}_{(0,0)}(\epsilon), \overline{j}_*\mathscr{F}\big)\\
    &\simeq \Gamma\big(U_{x} \times D_{(0,0)}^\circ(\epsilon), \mathscr{F}\big) \simeq \Gamma\big(\overline{U}_{x} \times \overline{D}_{(0,0)}^\circ(\epsilon), \mathscr{F}\big). \qedhere
    \end{align*}
\end{proof}

    Then we compute the stalks of $j_*i^{-1}\mathscr{F}$ at $(x, 0)$. Let $U_x$ be an open ball around $x \in N$, $D_0 = (-\epsilon, \epsilon)$, $D_0^\circ = (-\epsilon, -\delta) \cup (\delta, \epsilon)$, and respectively $\overline{U}_x$, $\overline{D}_0(\epsilon)$, $\overline{D}_0^\circ(\epsilon)$ be their closures.

\begin{lemma}\label{stalk2}
    Let $\mathscr{G} \in Sh(N \times [-1, 0)\cup(0, 1])$ be a sheaf such that $SS^\infty(\mathscr{G}) \cap \pi^*T^{*,\infty}([-1, 0)\cup (0,1]) = \varnothing$, and $\overline{SS_{\pi}^\infty(\mathscr{G})} \cap T^{*,\infty}N \times \{0\}$ is a subanalytic Legendrian. Then for any $x \in N$, $U_x \subset N$ a sufficiently small open neighbourhood and $\epsilon > 0$ sufficiently small,
    $$j_*\mathscr{G}_{(x,0)} \simeq \Gamma\big(\overline{U}_{x} \times \overline{D}_{0}^\circ(\epsilon), \mathscr{G}\big).$$
\end{lemma}
\begin{proof}
    Since $\overline{SS_{\pi}^\infty(\mathscr{G})} \cap T^{*,\infty}N \times \{0\}$ is a subanalytic Legendrian, for any sufficiently small neighbourhood $U_{x}$ of $x \in N $, we have
    $$\overline{SS_{\pi}^\infty(\mathscr{G})} \cap \nu^{*,\infty}_{U_{x},\pm}N \times \{0\} = \varnothing$$
    by general position argument. Since $SS^\infty(\mathscr{G}) \cap \pi^*T^{*,\infty}([-1,0)\cup(0,1]) = \varnothing$, we have a projection map to the relative singular support
    $$SS^\infty(\mathscr{G}) \cap \nu^{*,\infty}_{U_{x} \times D_{0}(\epsilon),\pm}(N \times [-1,0)\cup (0,1]) \rightarrow SS_{\pi}^\infty(\mathscr{G}) \cap \nu^{*,\infty}_{U_{x},\pm}N \times D_{0}(\epsilon).$$
    We know that there exist no sequences in the intersection of the relative singular support and unit conormal bundles that converge to $\overline{SS_{\pi}^\infty(\mathscr{G})} \cap \nu^{*,\infty}_{U_{x},\pm}N \times \{0\} = \varnothing$. Hence when $\epsilon > 0$ is sufficiently small, the intersection between relative singular support and $\nu^{*,\infty}_{U_{x},\pm}N \times D_0(\epsilon)$ is empty.

    Therefore, by non-characteristic deformation lemma Proposition \ref{morselemma} applied to the family $U_x \times D_{0}(\epsilon)$ and $U_x \times D_0^\circ(\epsilon)$, we have
    \begin{align*}
    j_*\mathscr{G}_{(x,0)} & \simeq \Gamma\big({U}_{x} \times {D}_0(\epsilon), j_*\mathscr{G}\big) \simeq \Gamma\big(\overline{U}_{x} \times \overline{D}_0(\epsilon), j_*\mathscr{G}\big) \\
    &\simeq \Gamma\big({U}_{x} \times {D}^{\circ}_0(\epsilon), \mathscr{G}\big) \simeq \Gamma\big(\overline{U}_{x} \times \overline{D}^{\circ}_0(\epsilon), \mathscr{G}\big). \qedhere
    \end{align*}
\end{proof}
\begin{remark}
    The above lemmas will also follow from the weak constructibility of $\mathscr{F}$ \cite[Lemma 4.2.2]{NadNearby}. For the applications, we believe that in fact both conditions hold.
\end{remark}

\begin{proof}[Proof of Proposition \ref{basechange}]
    We apply Lemma \ref{stalk1} to $\mathscr{F}$ and apply Lemma \ref{stalk2} and Theorem \ref{sspullback} to $i^{-1}\mathscr{F}$. Then it suffices to show that
    $$\Gamma\big(\overline{U}_{x} \times \overline{D}^{\circ}_{(0,0)}(\epsilon), \mathscr{F}\big) \simeq \Gamma\big(\overline{U}_{x} \times \overline{D}^{\circ}_0(\epsilon), \mathscr{F}\big).$$
    Since $\overline{SS_{\pi}^\infty(\mathscr{F})} \cap T^{*,\infty}N \times \{(0, 0)\}$ is a subanalytic Legendrian, for any sufficiently small neighbourhood $U_{x}$ of $x \in N$, we have
    $$\overline{SS_{\pi}^\infty(\mathscr{F})} \cap \nu^{*,\infty}_{U_{x},\pm}N \times \{(0, 0)\} = \varnothing$$
    by the general position argument. Write $D_{(0,0)}^\circ(\epsilon, \epsilon') = (-\epsilon', \epsilon') \times (\delta, \epsilon)$ for $0 \leq \epsilon' \leq \epsilon$. Since $SS^\infty(\mathscr{F}) \cap \pi^*T^{*,\infty}([-1, 0) \cup (0,1] \times [-1,0) \cup (0,1]) = \varnothing$, there is a projection map
    \begin{align*}
    SS^\infty(\mathscr{F})\, \cap &\, \nu_{{U}_{x} \times {D}_{(0,0)}^\circ(\epsilon,\epsilon'), \pm}^{*,\infty} (N \times [-1, 0) \cup (0, 1] \times [-1, 0)\cup (0, 1]) \\
    &\rightarrow SS^\infty_\pi(\mathscr{F}) \cap \nu_{{U}_{x},\pm}^{*,\infty}N \times \overline{D}_{(0,0)}^\circ(\epsilon,\epsilon').
    \end{align*}

    We claim that the right hand side is empty when $\epsilon > 0$ is sufficiently small. Otherwise, we can define a sequence in $SS^\infty_\pi(\mathscr{F}) \cap \nu_{{U}_{x},\pm}^{*,\infty}N \times \overline{D}_{(0,0)}^\circ(\epsilon,\epsilon')$ that converges to $\overline{SS_\pi(\mathscr{F})} \cap \nu_{{U}_{x},\pm}^{*,\infty}N \times \{(0, 0)\}$ as $\epsilon', \epsilon \rightarrow 0$. However, since $\overline{SS_\pi(\mathscr{F})} \cap \nu_{{U}_{x},\pm}^{*,\infty}N \times \{(0, 0)\} = \varnothing$, there are not any such sequences in the intersection of the relative singular support and unit conormal bundles. Hence by the projection map we conclude that when $\epsilon, \epsilon' > 0$ are sufficiently small,
    $$SS^\infty(\mathscr{F}) \cap \nu_{{U}_{x} \times {D}_{(0,0)}^\circ(\epsilon,\epsilon'), \pm}^{*,\infty}(N \times [-1,0) \cup (0, 1] \times [-1,0)\cup (0, 1]) = \varnothing. $$

    Therefore, by non-characteristic deformation lemma Proposition \ref{morselemma} applied to the family $D_{(0,0)}^\circ(\epsilon, \epsilon')$ for $0 < \epsilon' \leq \epsilon$, we can conclude that
    \begin{equation*}
    \Gamma\big(\overline{U}_{x} \times \overline{D}_{(0,0)}^\circ(\epsilon), \mathscr{F}\big) \simeq \Gamma\big(\overline{U}_{x} \times \overline{D}_0^\circ(\epsilon, \epsilon'), \mathscr{F}\big) \simeq \Gamma\big(\overline{U}_{x} \times \overline{D}_0^\circ(\epsilon), \mathscr{F}\big). \qedhere
    \end{equation*}
\end{proof}
\begin{remark}
    When applying non-characteristic deformation lemma, one should notice that $\partial (\overline{U}_{x} \times \overline{D}_{(0,0)}^\circ(\epsilon))$ is piecewise smooth. Therefore, we need to use the condition that $SS^\infty(\mathscr{F}) \cap \pi^*T^{*,\infty}([-1,0)\cup (0,1] \times [-1,0)\cup (0,1]) = \varnothing$ rather than only considering the intersection with $\pi_i^*T^{*,\infty}([-1,0)\cup (0,1])$. For the same reason, we need the estimate on $\overline{SS^\infty_\pi(\mathscr{F})} \cap T^{*,\infty}N \times \{(0, 0)\}$ rather than estimates on $\overline{SS^\infty_{\pi_i}(\mathscr{F})} \cap T^{*,\infty}N \times \{(0, 0)\}$. The author is grateful to an anonymous referee for pointing out the mistake in the proposition.
\end{remark}

\subsubsection{Concatenation and composition}
    First, we show that concatenations of Lagrangian cobordisms give rise to compositions of our Lagrangian cobordism functors. Therefore our cobordism functor defines a functor from the category of Lagrangian cobordisms to the category of (small) dg categories.

    We recall how concatenations of Lagrangian cobordisms are defined. Let $L_0 \subset \partial_\infty X \times \mathbb{R}$ be a Lagrangian cobordism from $\Lambda_0$ to $\Lambda_1$, and $L_1$ be a Lagrangian cobordism from $\Lambda_1$ to $\Lambda_2$. Suppose $L_{0,1} \cap \partial_\infty X \times (-\infty,-r_0) \cup (r_0,+\infty)$ are standard cylinders. Then the concatenation $L_0 \cup L_1$ is an exact Lagrangian such that
\begin{enumerate}
  \item $(L_0 \cup L_1) \cap \partial_\infty X \times (-\infty,0) \cong \varphi_Z^{-r_0}(L_0) \cap \partial_\infty X \times (-\infty,0)$;
  \item $(L_0 \cup L_1) \cap \partial_\infty X \times (0,+\infty) \cong \varphi_Z^{r_0}(L_1) \cap \partial_\infty X \times (0,+\infty)$.
\end{enumerate}
    Here $\varphi_Z^z$ is the Liouville flow on $\partial_\infty X \times \mathbb{R} \subset X$.

\begin{theorem}[Concatenation]\label{concate}
    Let $X$ be a Weinstein manifold, $\Lambda_{0,1,2} \subset \partial_\infty X$ be Legendrian submanifolds, $L_0 \subset \partial_\infty X \times \mathbb{R}$ be a Lagrangian cobordism from $\Lambda_0$ to $\Lambda_1$, and $L_1$ be a Lagrangian cobordism from $\Lambda_1$ to $\Lambda_2$. Then
    $$\Phi_{L_0 \cup L_1}^* \simeq (\Phi_{L_0}^* \otimes \mathrm{id}_{Loc^c(L_1)}) \circ \Phi_{L_1}^*, \,\,\, \Phi_{L_0 \cup L_1} \simeq \Phi_{L_1} \circ (\Phi_{L_0} \times \mathrm{id}_{Loc^b(L_1)}).$$
\end{theorem}

    We will consider the (large) dg categories $\mu Sh_{\mathfrak{c}_X \cup \Lambda_{0,1,2} \times \mathbb{R}}(\mathfrak{c}_X \cup \Lambda_{0,1,2} \times \mathbb{R})$ and $Loc(L_{0,1})$, and show that
    \[\begin{split}
    \Phi_{L_0 \cup L_1} \simeq \,&\Phi_{L_1} \circ (\Phi_{L_0} \times \mathrm{id}_{Loc(L_1)}): \\
    & \, \mu Sh_{\mathfrak{c}_X \cup \Lambda_0 \times \mathbb{R}}(\mathfrak{c}_X \cup \Lambda_0 \times \mathbb{R}) \times_{Loc(\Lambda_0)} Loc(L_0) \times_{Loc(\Lambda_1)} Loc(L_1) \\
    & \rightarrow \mu Sh_{\mathfrak{c}_X \cup \Lambda_2 \times \mathbb{R}}(\mathfrak{c}_X \cup \Lambda_2 \times \mathbb{R}).
    \end{split}\]
    Then the results will immediately follow from the properties of adjoint functors.

    Our strategy is as follows. $\Phi_{L_0 \cup L_1}$ is defined by using the Liouville flow to compress $L_0 \cup L_1$ to the skeleton all at once, and $\Phi_{L_1} \circ (\Phi_{L_0} \times \mathrm{id}_{Loc(L_1)})$ is defined by first compressing $L_0$ to the skeleton while fixing $L_1$, and next compressing $L_1$ to the skeleton. We will try to define a 2-parametric family of contact flow that interpolates between them. Then following the construction, $\Phi_{L_0 \cup L_1}$ and $\Phi_{L_1} \circ (\Phi_{L_0} \times \mathrm{id}_{Loc(L_1)})$ are two different compositions of nearby cycles, and the theorem is reduced to commutativity of the nearby cycle functors.

    We now start the proof of the theorem. Consider the lifting of the Liouville flow $\varphi_Z^z$ in $T^{*,\infty}N$ that satisfies
    $$d\varphi_Z^z/dz = t\partial/\partial t + Z_\lambda$$
    on $X \times \mathbb{R}$. Suppose that the concatenation $(L_0 \cup L_1) \cap \partial_\infty X \times (-\epsilon,\epsilon) = \Lambda_1 \times (-\epsilon,\epsilon).$ Let $\eta: \mathbb{R} \rightarrow [0,1]$ be a cut-off function such that $\eta|_{(-\infty,-\epsilon]} \equiv 0$ and $\eta|_{[\epsilon,+\infty)} \equiv 0$. Then we consider a flow $\varphi_{Z'}^z$ on $\partial_\infty X \times \mathbb{R}$ defined by $Z' = \eta(r)Z_\lambda = \eta(r)e^r\partial/\partial r$, such that
    $$\varphi_{Z'}^z|_{\partial_\infty X \times (-\infty,-\epsilon)} = \varphi_Z^z, \,\,\, \varphi_{Z'}^z|_{\partial_\infty X \times (\epsilon,+\infty)} = \mathrm{id}.$$
    Note that $\varphi_{Z'}^z$ defines an exact Lagrangian isotopy of $L_0 \cup L_1$, which can be lifted to a Legendrian isotopy of $\widetilde{L}_0 \cup \widetilde{L}_1$. Therefore, lift $\varphi_{Z'}^z$ to a contact flow on $X \times \mathbb{R}$ and still denote it by $\varphi_{Z'}^z$. As a contact flow,
    $$d\varphi_{Z'}^z/dz|_{\partial_\infty X \times (-\infty,-\epsilon) \times \mathbb{R}} = t\partial/\partial t + Z_\lambda.$$
%    Suppose the primitive of $L_1$ is $f_{L_1}$ such that $f_{L_1} = 0$ near $\partial_\infty X \times (-\infty, -r_0)$. Then on $\varphi_Z^{r_0}(\widetilde{L}_1) \cap \partial_\infty X \times (\epsilon,+\infty)$ we have
%    $$d\varphi_Z^z(x, t)/dz\big|_{\varphi_Z^{r_0}(\widetilde{L}_1) \cap (\partial_\infty X \times (\epsilon,+\infty) \times \mathbb{R})} = \left(t - f_{L_1}(x)\right) {\partial}/{\partial t}.$$

    Write $\phi_Z^\zeta = \varphi_Z^{\ln\zeta}$ and $\phi_{Z'}^\zeta = \varphi_{Z'}^{\ln\zeta}$. Consider the 2-parameter family of contact Hamiltonian $\phi_{\overline{Z}'}^{\zeta,\eta} = \phi_Z^{\zeta} \circ \phi_{Z'}^{\eta - \zeta}$. Then $\phi_{\bar{Z}'}^{\zeta,\zeta} = \varphi_Z^{\zeta}, \,\,\, \phi_{\bar{Z}'}^{1,\eta} = \varphi_{Z'}^{\eta}.$ In particular, the limits satisfy
    \[\begin{split}
    \lim_{\zeta \rightarrow 0}\phi_{\bar{Z}'}^{\zeta,\zeta}(-) = \lim_{\zeta \rightarrow 0}\phi_Z^\zeta(-) &= \lim_{z \rightarrow -\infty}\varphi_Z^z(-), \\
    \lim_{\eta \rightarrow 0}\phi_{\bar{Z}'}^{\zeta,\eta}(-) = \phi_Z^\zeta \Big(\lim_{\eta \rightarrow 0}\phi_{Z'}^\eta(-)\Big) &= \phi_Z^{\zeta}\Big(\lim_{y \rightarrow -\infty}\varphi_{Z'}^y(-)\Big).
    \end{split}\]
    Write $\Delta = \{(\zeta, \eta) | 0 < \eta \leq \zeta \leq 1\}$, $\overline{\Delta} = \{(\zeta, \eta) | 0\leq \eta\leq \zeta\leq 1\}$ and $\overline{\Delta}_0 = \overline{\Delta} \backslash \{(0,0)\}$.

\begin{figure}
  \centering
  \includegraphics[width=0.7\textwidth]{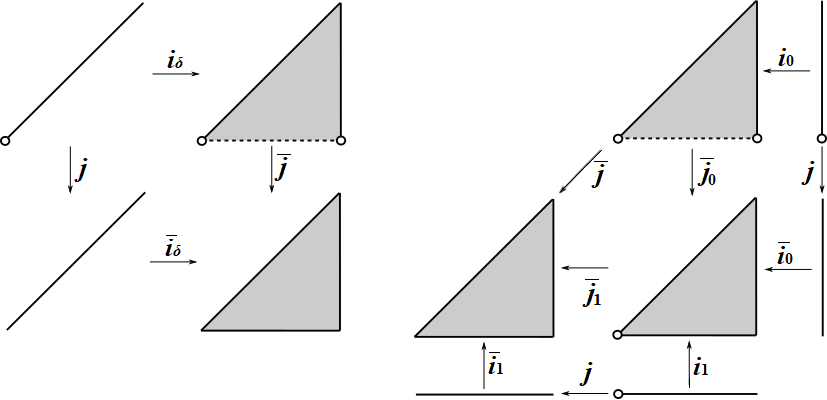}\\
  \caption{The diagram of maps in the proof of Theorem \ref{concate}.}\label{basechange-fig}
\end{figure}

\begin{proof}[Proof of Theorem \ref{concate}]
    Consider the 2-parameter family of contact flows $\phi_{\bar{Z}'}^{\zeta,\eta},\,(\zeta, \eta) \in \Delta$. By Theorem \ref{cont-trans} Remark \ref{cont-trans-para}, for $\mathscr{F} \in \mu Sh_{\mathfrak{c}_X \cup \Lambda_0 \times \mathbb{R} \cup \widetilde{L}_0 \cup \widetilde{L}_1}(\mathfrak{c}_X \cup \Lambda_0 \times \mathbb{R} \cup \widetilde{L}_0 \cup \widetilde{L}_1)$, we get
    $$\Psi_{\bar{Z}'}^{\zeta,\eta}(\mathscr{F}) \in \mu Sh_{(\mathfrak{c}_X \cup \Lambda_0 \times \mathbb{R} \cup \widetilde{L}_0 \cup \widetilde{L}_1)_{\bar{Z}'}}\big((\mathfrak{c}_X \cup \Lambda_0 \times \mathbb{R} \cup \widetilde{L}_0 \cup \widetilde{L}_1)_{\bar{Z}'}\big),$$
    where $(\mathfrak{c}_X \cup \Lambda_0 \times \mathbb{R} \cup \widetilde{L}_0 \cup \widetilde{L}_1)_{\bar{Z}'}$ is the Legendrian movie of $\mathfrak{c}_X \cup \Lambda_0 \times \mathbb{R} \cup \widetilde{L}_0 \cup \widetilde{L}_1$ under $\phi_{\bar{Z}'}^{\zeta,\eta}$ in Definition \ref{lagmovie}. Applying Theorem \ref{antimicro} \cite{NadShen}, we will write $\Psi_{\bar{Z}'}^{\zeta,\eta}(\mathscr{F})_\text{dbl} \in Sh(N \times \Delta)$
    for the image of $\Psi_{\bar{Z}'}^{\zeta,\eta}(\mathscr{F})$ under the antimicrolocalization functor.

    From Figure \ref{basechange-fig} one can notice that $\Phi_{L_0 \cup L_1}$ and $\Phi_{L_0} \circ (\Phi_{L_0} \times \mathrm{id}_{Loc(L_1)})$ are (compositions of) nearby cycles along different boundary edges of $\overline{\Delta}$. Therefore it suffices to show that the nearby cycle functors commute and they agree with the 2-parametric nearby cycle functor. In order to apply Proposition \ref{basechange} in our argument, note that firstly $SS^\infty(\Psi_{\bar{Z}'}^{\zeta,\eta}(\mathscr{F})) \cap \pi^*T^{*,\infty}\Delta = \varnothing$ because the singular support is the Legendrian movie under a contact Hamiltonian flow, and secondly the limit points of the relative singular support
    $$\overline{SS_{\pi \phantom{\bar{1}}}^{\infty}(\Psi_{\bar{Z}'}^{\zeta,\eta}(\mathscr{F})_\text{dbl})} \cap T^{*,\infty}N \times \{(0, 0)\} \subseteq \lim_{\eta, \zeta \rightarrow 0}\phi_{\bar{Z}'}^{\zeta,\eta}(\mathfrak{c}_X \cup \Lambda_0 \times \mathbb{R} \cup \widetilde{L}_0 \cup \widetilde{L}_1) \subseteq \mathfrak{c}_X \cup \Lambda_2 \times \mathbb{R}$$
%    \lim_{\eta \rightarrow 0}\phi_{\bar{Z}'}^{\zeta,\eta}(\mathfrak{c}_X \cup \Lambda_0 \times \mathbb{R} \cup \widetilde{L}_0 \cup \widetilde{L}_1) \subseteq \phi_{Z}^\zeta (\mathfrak{c}_X \cup \Lambda_1 \times \mathbb{R} \cup \widetilde{L}_1),
    form a subanalytic Legendrian. Therefore, in the following cases Proposition \ref{basechange} will apply.

    (1)~Firstly, we consider $\Phi_{L_0 \cup L_1}\left(\mathscr{F}\right)$ (Figure \ref{basechange-fig} left). Write $i_\delta: N \times (0,1] \hookrightarrow N \times \Delta, \, (x, \zeta) \mapsto (x, \zeta, \zeta)$, $j: N \times (0,1] \hookrightarrow N \times [0,1]$ and $i: N \times \{0\} \hookrightarrow N \times [0,1]$. Then since $\phi_{\overline{Z}'}^{\zeta,\zeta} = \phi_Z^{\zeta}$,
    \[\begin{split}
    \Phi_{L_0 \cup L_1}\left(\mathscr{F}\right)_\text{dbl} \xrightarrow{\sim} i^{-1}j_*\Psi_Z^\zeta(\mathscr{F})_\text{dbl} \xrightarrow{\sim} i^{-1}j_*\big(i_\delta^{-1}\Psi_{\overline{Z}'}^{\zeta,\eta}(\mathscr{F})\big)_\text{dbl}.
    \end{split}\]
    Write $\overline{i}_\delta: N \times [0,1] \hookrightarrow N \times \overline{\Delta}, \, (x, \zeta) \mapsto (x, \zeta, \zeta)$, $\overline{j}: N \times \Delta \rightarrow N \times \overline{\Delta}$ and $\overline{i}: N \times \{(0,0)\} \hookrightarrow N \times \overline{\Delta}$. By Proposition \ref{basechange} and Remark \ref{GKSviaCont}, we know that in fact
    $$\Phi_{L_0 \cup L_1}\left(\mathscr{F}\right)_\text{dbl} \xrightarrow{\sim} i^{-1}\overline{i}_\delta^{-1}\overline{j}_*\Psi_{\overline{Z}'}^{\zeta,\eta}(\mathscr{F})_\text{dbl} \xrightarrow{\sim} \overline{i}^{-1}\overline{j}_*\Psi_{\overline{Z}'}^{\zeta,\eta}(\mathscr{F})_\text{dbl}.$$

    (2)~Secondly, we consider $(\Phi_{L_0} \times \mathrm{id}_{Loc(L_1)})(\mathscr{F})$ (Figure \ref{basechange-fig} right). Note that $\varphi_{Z'}^y$ preserves $L_1$, and compresses $\Lambda_0 \times \mathbb{R} \cup L_0$ to $\mathfrak{c}_X \cup \Lambda_1 \times \mathbb{R}$ as the Liouville flow $\varphi_Z^y$. Thus the restriction of $i^{-1}j_*\Psi_{Z'}^\eta(\mathscr{F})_\text{dbl}$ to $\mathfrak{c}_X \cup \Lambda_1 \times \mathbb{R}$ is $\Phi_{L_0}(\mathscr{F})_\text{dbl} = i^{-1}j_*\Psi_{Z}^\eta(\mathscr{F})_\text{dbl}$. On the other hand, by Theorem \ref{cont-trans} the microlicalization of $i^{-1}j_*\Psi_{Z'}^\eta(\mathscr{F})_\text{dbl}$ to $\widetilde{L}_1$ gives an equivalence $\mu Sh_{\widetilde{L}_1}(\widetilde{L}_1) \simeq \mu Sh_{\widetilde{L}_1 \times [0,1]}(\widetilde{L}_1 \times [0,1]) \simeq \mu Sh_{\widetilde{L}_1}(\widetilde{L}_1)$ induced by the Reeb translation. By Theorem \ref{Gui} \cite{Gui} this is the identity functor on $Loc(L_1)$. Therefore,
    $$(\Phi_{L_0} \times \mathrm{id}_{Loc(L_1)})(\mathscr{F})_\text{dbl} \xrightarrow{\sim} i^{-1}j_*\Psi_{Z'}^\eta(\mathscr{F})_\text{dbl}.$$
    Write $i_0: N \times (0,1] \hookrightarrow N \times \Delta, \, (x, \eta) \mapsto (x, 1, \eta)$. Since $\phi_{\overline{Z}'}^{1,\eta} = \phi_{Z'}^\eta$, we know that
    $$(\Phi_{L_0} \times \mathrm{id}_{Loc(L_1)})(\mathscr{F})_\text{dbl} \xrightarrow{\sim} i^{-1}j_*\Psi_{Z'}^\eta(\mathscr{F})_\text{dbl} \xrightarrow{\sim} i^{-1}j_*\big(i_0^{-1}\Psi_{\bar{Z}'}^{\zeta,\eta}(\mathscr{F})\big)_\text{dbl}.$$
    Write $\overline{j}_0: N \times \Delta \hookrightarrow N \times \overline{\Delta}_0$ where $\overline{\Delta}_0 = \overline{\Delta} \backslash \{(0,0)\}$, and $\overline{i}_0: N \times [0,1] \hookrightarrow N \times \overline{\Delta}, \, (x, \eta) \mapsto (x, 1, \eta)$. By Proposition \ref{basechange} and Remark \ref{GKSviaCont}, we know that in fact
    $$(\Phi_{L_0} \times \mathrm{id}_{Loc(L_1)})(\mathscr{F})_\text{dbl} \xrightarrow{\sim} i^{-1}\overline{i}_0^{-1}\overline{j}_{0,*}\Psi_{\bar{Z}'}^{\zeta,\eta}(\mathscr{F})_\text{dbl}.$$

    Then we consider $\Phi_{L_1} \circ (\Phi_{L_0} \times \mathrm{id}_{Loc(L_1)})(\mathscr{F})$ (Figure \ref{basechange-fig} right). Write $i_1: N \times (0,1]\hookrightarrow N \times \overline{\Delta}_0, \, (x, \zeta) \mapsto (x, \zeta, 0)$ where $\overline{\Delta}_0 = \overline{\Delta} \backslash \{(0,0)\}$. Let $\varphi_{\overline{Z}}^z$ be the Liouville flow on $T^{*,\infty}(N \times [0,1])$ defined by the pullback (homogeneous) Hamiltonian $H_{\bar{Z}} = H_Z \circ \pi_{T^{*}N}$, and $\phi_{\bar{Z}}^\zeta = \varphi_{\bar{Z}}^{\ln\zeta}$. Let $\Psi_{\bar{Z}}^\zeta: \, Sh(N \times \{1\} \times [0,1]) \rightarrow Sh(N \times \overline{\Delta}_0)$ be the Hamiltonian isotopy functor as in Theorem \ref{GKS}. Thus by Proposition \ref{basechange}
    \[\begin{split}
    \big(\Psi_Z^\zeta \circ (\Phi_{L_0} \times \mathrm{id}_{Loc(L_1)})(\mathscr{F})\big)_\text{dbl} &\xrightarrow{\sim} \Psi_Z^\zeta\big(i^{-1}\overline{i}_0^{-1}\overline{j}_{0,*} \Psi_{\bar{Z}'}^{\zeta,\eta}(\mathscr{F})_\text{dbl}\big) \\
    &\xrightarrow{\sim} i_1^{-1}\Psi_{\bar{Z}}^\zeta\big(\overline{i}_0^{-1}\overline{j}_{0,*} \Psi_{\bar{Z}'}^{\zeta,\eta}(\mathscr{F})_\text{dbl}\big)
    \xrightarrow{\sim} i_1^{-1}\overline{j}_{0,*}\Psi_{\bar{Z}'}^{\zeta,\eta}(\mathscr{F})_\text{dbl}.
    \end{split}\]
    Then write $\overline{j}_1: N \times \overline{\Delta}_0 \hookrightarrow N \times \overline{\Delta}$ and $\overline{i}_1: N \times [0, 1] \hookrightarrow N \times \overline{\Delta}$. By Proposition \ref{basechange} again, we can show that
    \[\begin{split}\Phi_{L_1} \circ (\Phi_{L_0} \times \mathrm{id}_{Loc(L_1)})(\mathscr{F})_\text{dbl} & \xrightarrow{\sim} i^{-1}j_*\big(\Psi_Z^\zeta \circ (\Phi_{L_0}^* \times \mathrm{id}_{Loc(L_1)})(\mathscr{F})\big)_\text{dbl} \\
    & \xrightarrow{\sim} i^{-1}j_*i_1^{-1}\overline{j}_{0,*}\Psi_{\bar{Z}'}^{\zeta,\eta}(\mathscr{F})_\text{dbl}
    \xrightarrow{\sim} i^{-1}\overline{i}_1^{-1}\overline{j}_{1,*}\overline{j}_{0,*} \Psi_{\bar{Z}'}^{\zeta,\eta}(\mathscr{F})_\text{dbl} \\
    & \xrightarrow{\sim} i^{-1}\overline{i}_1^{-1}\overline{j}_{*}\Psi_{\bar{Z}'}^{\zeta,\eta}(\mathscr{F})_\text{dbl}
    \xrightarrow{\sim} \overline{i}^{-1}\overline{j}_*\Psi_{\bar{Z}'}^{\zeta,\eta}(\mathscr{F})_\text{dbl}.
    \end{split}\]
    Therefore, we can conclude that $\Phi_{L_0 \cup L_1}\left(\mathscr{F}\right) \simeq \Phi_{L_1} \circ (\Phi_{L_0} \times \mathrm{id}_{Loc(L_1)})(\mathscr{F})$.
%    (3).~On the level of morphisms, the base change formulas provide natural transformations between the morphism spaces, and Theorem \ref{nearby} \cite[Theorem 5.1]{NadShen} shows that the natural transformations are quasi-isomorphisms, and hence completes the proof.
\end{proof}

\subsubsection{Hamiltonian invariance}
    Next, we show that the Lagrangian cobordism functor is invariant under Hamiltonian isotopies in the symplectization that fix the positive (convex) and negative (concave) ends.

\begin{theorem}[Hamiltonian invariance]\label{invariance}
    Let $X$ be a Weinstein manifold, $\Lambda_\pm \subset \partial_\infty X$ be Legendrian submanifolds, and $L \subset \partial_\infty X \times \mathbb{R}$ be a Lagrangian cobordism from $\Lambda_-$ to $\Lambda_+$. Suppose there is a compactly supported Hamiltonian isotopy $\varphi_H^s,\,s\in I$, on $\partial_\infty X \times \mathbb{R}$. Then
    $$\Phi_L^* \simeq \Phi_{\varphi_H^1(L)}^*, \,\,\, \Phi_L \simeq \Phi_{\varphi_H^1(L)}.$$
\end{theorem}

    Again, we can only consider the dg categories $\mu Sh_{\mathfrak{c}_X \cup \Lambda_\pm \times \mathbb{R}}(\mathfrak{c}_X \cup \Lambda_\pm \times \mathbb{R})$ and $Loc(L)$, and show that
    \[\begin{split}
    \Phi_L \simeq \Phi_{\varphi_H^1(L)}: \, \mu Sh&\,{}_{\mathfrak{c}_X \cup \Lambda_- \times \mathbb{R}}(\mathfrak{c}_X \cup \Lambda_- \times \mathbb{R}) \times_{Loc(\Lambda_-)} Loc(L) \\
    &\rightarrow \mu Sh_{\mathfrak{c}_X \cup \Lambda_+ \times \mathbb{R}}(\mathfrak{c}_X \cup \Lambda_+ \times \mathbb{R}).
    \end{split}\]
    Then the results will immediately follow from the properties of adjoint functors.

    Our strategy is to compare $\Phi_L$ and $\Phi_{\varphi_H^s(L)}$ by constructing a 1-parametric family of Lagrangian cobordism functors, and then Theorem \ref{GKS} \cite{GKS} will allow us to show that $\Phi_L \simeq \Phi_{\varphi_H^1(L)}$ from the initial condition $\Phi_L \simeq \Phi_{\varphi_H^0(L)}$.

    Identify $\mathfrak{c}_X \cup \Lambda_\pm \times \mathbb{R}$ and $L$ with their Legendrian image in some higher dimensional contact manifold $T^{*,\infty}N$, and lift $\varphi_H^s$ to a contact Hamiltonian flow on $T^{*,\infty}N$. Consider $(\mathfrak{c}_X \cup \Lambda_- \times \mathbb{R}) \times I$. Then we have a Lagrangian cobordism functor
    \[\begin{split}
    \Phi_{L \times I}: \, \mu Sh&{}_{((\mathfrak{c}_X \cup \Lambda_- \times \mathbb{R}) \times I) \cup (L \times I)}\big(((\mathfrak{c}_X \cup \Lambda_- \times \mathbb{R}) \times I) \cup (L \times I)\big)\\
    &\rightarrow \mu Sh_{(\mathfrak{c}_X \cup \Lambda_+ \times \mathbb{R}) \times I}\big((\mathfrak{c}_X \cup \Lambda_+ \times \mathbb{R}) \times I\big).
    \end{split}\]
    On the other hand, let $\widetilde{L}_H$ be the Legendrian movie of $\widetilde{L}$ (in Definition \ref{lagmovie}) under the Hamiltonian flow $\varphi_H^s$. Then we have a Lagrangian cobordism functor
    \[\begin{split}
    \Phi_{L_H}: \, \mu Sh&{}_{((\mathfrak{c}_X \cup \Lambda_- \times \mathbb{R}) \times I) \cup \widetilde{L}_H}\big(((\mathfrak{c}_X \cup \Lambda_- \times \mathbb{R}) \times I) \cup \widetilde{L}_H\big)\\
    &\rightarrow \mu Sh_{(\mathfrak{c}_X \cup \Lambda_+ \times \mathbb{R}) \times I}\big((\mathfrak{c}_X \cup \Lambda_+ \times \mathbb{R}) \times I \big).
    \end{split}\]
    For $\mathscr{F} \in \mu Sh_{\mathfrak{c}_X \cup \Lambda_- \times \mathbb{R} \cup \widetilde{L}}\big(\mathfrak{c}_X \cup \Lambda_- \times \mathbb{R} \cup \widetilde{L}\big)$, write $\pi: T^{*,\infty}(N \times I) \rightarrow T^{*,\infty}N$. We consider
    $$\pi^{-1}(\mathscr{F}) \in \mu Sh_{((\mathfrak{c}_X \cup \Lambda_- \times \mathbb{R}) \times I) \cup ( \widetilde{L} \times I)}\big(((\mathfrak{c}_X \cup \Lambda_- \times \mathbb{R}) \times I) \cup (\widetilde{L} \times I)\big).$$
    On the other hand, by Theorem \ref{cont-trans} the Hamiltonian isotopy $\varphi_H^s$ defines a canonical sheaf
    $$\Psi_H^0(\mathscr{F}) \in \mu Sh_{((\mathfrak{c}_X \cup \Lambda_- \times \mathbb{R}) \times I) \cup \widetilde{L}_H}\big(((\mathfrak{c}_X \cup \Lambda_- \times \mathbb{R}) \times I) \cup \widetilde{L}_H\big).$$

\begin{lemma}\label{restrict-inv}
    Let $\pi: N \times I \rightarrow N$ be the projection, $i_s: N \times \{s\} \hookrightarrow N \times I$ be the inclusion, and $\widetilde{L}_H$ be the Legendrian movie of $\widetilde{L}$ under the Hamiltonian flow $\varphi_H^s,\,s\in I$. Then for any $\mathscr{F} \in \mu Sh_{\mathfrak{c}_X \cup \Lambda_- \times \mathbb{R} \cup \widetilde{L}}(\mathfrak{c}_X \cup \Lambda_- \times \mathbb{R} \cup \widetilde{L})$,
    $$i_s^{-1}\Phi_{L \times I}(\pi^{-1}(\mathscr{F})) = \Phi_L(\mathscr{F}), \,\,\, i_s^{-1}\Phi_{L_H}(\Psi_H^0(\mathscr{F})) = \Phi_{\varphi_H^s(L)}(\mathscr{F}).$$
\end{lemma}
\begin{proof}
    First of all, let $\varphi_{\overline{Z}}^z$ be the Liouville flow on $T^{*,\infty}(N \times I)$ defined by the pullback (homogeneous) Hamiltonian $H_{\overline{Z}} = H_Z \circ \pi_{T^*N}$. Let $\Psi_{\overline{Z}}^\zeta$ be the equivalence functor defined by the Liouville flow $\varphi_{\overline{Z}}^z,\,z \in (-\infty,0]$, or $\phi_{\overline{Z}}^\zeta,\,\zeta \in (0,1]$, on $T^{*,\infty}(N \times I)$, and
    \[\begin{split}
    i_Z: N \times \{0\} \hookrightarrow N \times [0,1],\,&\,\, j_Z: N \times (0,1] \hookrightarrow N \times [0,1], \\
    \overline{i}_{\overline{Z}}: N \times I \times \{0\} \hookrightarrow N \times I \times [0,1], \,&\,\, \overline{j}_{\overline{Z}}: N \times I \times (0,1] \hookrightarrow N \times I \times [0,1].
    \end{split}\]
    Write $(\Psi_{\overline{Z}}^\zeta(\Psi_H^0(\mathscr{F})))_\text{dbl} \in Sh(N \times I \times (0,1])$ for the image of $\Psi_{\overline{Z}}^\zeta(\Psi_H^0(\mathscr{F}))$ under the antimicrolocalization functor in Theorem \ref{antimicro} \cite{NadShen}. Then by Remark \ref{GKSviaCont},
    $$i_s^{-1}\big(\Psi_{\overline{Z}}^\zeta(\Psi_H^0(\mathscr{F}))\big)_\text{dbl} = \Psi_{Z}^\zeta(i_s^{-1}\Psi_H^0(\mathscr{F}))_\text{dbl}.$$
    Similarly, let $i_s^{-1}\Phi_{L \times I}(\pi^{-1}(\mathscr{F}))_\text{dbl} \in Sh(N \times \{s\})$ be the image of $i_s^{-1}\Phi_{L \times I}(\pi^{-1}(\mathscr{F}))$ under the antimicrolocalization functor. By Proposition \ref{basechange} we have
    \[\begin{split}
    i_s^{-1}\Phi_{L \times I}(\pi^{-1}(\mathscr{F}))_\text{dbl} &\xrightarrow{\sim} i_s^{-1}\overline{i}_{\overline{Z}}^{-1}\overline{j}_{\overline{Z},*}\big(\Psi_{\overline{Z}}^\zeta( \pi^{-1}(\mathscr{F}))\big)_\text{dbl} \\
    &\xrightarrow{\sim} i_Z^{-1}j_{Z,*}\big(\Psi_{Z}^\zeta(i_s^{-1}\pi^{-1}(\mathscr{F}))\big)_\text{dbl} \xrightarrow{\sim} \Phi_L(\mathscr{F})_\text{dbl}.
    \end{split}\]
    On the other hand, let $i_s^{-1}\Phi_{L_H}(\pi^{-1}(\mathscr{F}))_\text{dbl} \in Sh(N \times \{s\})$ be the image of $i_s^{-1}\Phi_{L_H}(\pi^{-1}(\mathscr{F}))$ under the antimicrolocalization functor. By Proposition \ref{basechange} again we also have
    \[\begin{split}
    i_s^{-1}\Phi_{L_H}(\pi^{-1}(\mathscr{F}))_\text{dbl} &\xrightarrow{\sim} i_s^{-1}\overline{i}_{\overline{Z}}^{-1}\overline{j}_{\overline{Z},*}\big(\Psi_{\overline{Z}}^\zeta(
    \Psi_H^0(\mathscr{F}))\big)_\text{dbl} \\
    &\xrightarrow{\sim} i_Z^{-1}j_{Z,*}\big(\Psi_{Z}^\zeta(i_s^{-1}\Psi_H^0(\mathscr{F}))\big)_\text{dbl} \xrightarrow{\sim} \Phi_{\varphi_H^s(L)}(\mathscr{F})_\text{dbl}.
    \end{split}\]
    Therefore the proof is completed.
\end{proof}

\begin{proof}[Proof of Theorem \ref{invariance}]
    For $\mathscr{F} \in \mu Sh_{\mathfrak{c}_X \cup \Lambda_- \times \mathbb{R} \cup \widetilde{L}}(\mathfrak{c}_X \cup \Lambda_- \times \mathbb{R} \cup \widetilde{L})$, we consider
    $$\pi^{-1}(\mathscr{F}) \in \mu Sh_{((\mathfrak{c}_X \cup \Lambda_- \times \mathbb{R}) \times I) \cup (\widetilde{L} \times I)}\big(((\mathfrak{c}_X \cup \Lambda_- \times \mathbb{R}) \times I) \cup (\widetilde{L} \times I)\big).$$
    On the other hand, for the Hamiltonian isotopy $\varphi_H^s$ we consider by Theorem \ref{cont-trans}
    $$\Psi_H^0(\mathscr{F}) \in \mu Sh_{((\mathfrak{c}_X \cup \Lambda_- \times \mathbb{R}) \times I) \cup \widetilde{L}_H}\big(((\mathfrak{c}_X \cup \Lambda_- \times \mathbb{R}) \times I) \cup \widetilde{L}_H\big).$$
    There is a natural morphism $\pi^{-1}(\mathscr{F}) \rightarrow \Psi_H^0(\mathscr{F})$, and thus a natural morphism
    $$\Phi_{L \times I}\big(\pi^{-1}(\mathscr{F})\big) \rightarrow \Phi_{L_H}\big(\Psi_H^0(\mathscr{F})\big).$$
    We will show that this is a natural quasi-isomorphism. In fact,
    $$\mathrm{Cone}\left(\Phi_{L \times I}\big(\pi^{-1}(\mathscr{F})\big) \rightarrow \Phi_{L_H}\big(\Psi_H^0(\mathscr{F})\big)\right) \in \mu Sh_{(\mathfrak{c}_X \cup \Lambda_+ \times \mathbb{R}) \times I}\big((\mathfrak{c}_X \cup \Lambda_+ \times \mathbb{R}) \times I\big).$$
    By Lemma \ref{restrict-inv}, we also know that when $s = 0$,
    \[\begin{split}
    i_0^{-1}\mathrm{Cone}&\left(\Phi_{L \times I}\big(\pi^{-1}(\mathscr{F})\big) \rightarrow \Phi_{L_H}\big(\Psi_H^0(\mathscr{F})\big)\right) \\
    &\simeq \mathrm{Cone}\left(i_0^{-1}\Phi_{L \times I}\big(\pi^{-1}(\mathscr{F})\big) \rightarrow i_0^{-1}\Phi_{L_H}\big(\Psi_H^0(\mathscr{F})\big)\right) \\
    &\simeq \mathrm{Cone}\left(\Phi_{L}\big(i_0^{-1}\pi^{-1}(\mathscr{F})\big) \rightarrow \Phi_{L}\big(i_0^{-1}\Psi_H^0(\mathscr{F})\big)\right) \simeq 0.
    \end{split}\]
    As by Theorem \ref{cont-trans}, $i_0^{-1}: \mu Sh_{(\mathfrak{c}_X \cup \Lambda_+ \times \mathbb{R}) \times I}\big((\mathfrak{c}_X \cup \Lambda_+ \times \mathbb{R}) \times I\big) \rightarrow \mu Sh_{\mathfrak{c}_X \cup \Lambda_+ \times \mathbb{R}}(\mathfrak{c}_X \cup \Lambda_+ \times \mathbb{R}\big)$ defines an equivalence, we can conclude that the mapping cone is identically zero, and thus
    $$\Phi_{L \times I}\big(\pi^{-1}(\mathscr{F})\big) \xrightarrow{\sim} \Phi_{L_H}\big(\Psi_H^0(\mathscr{F})\big).$$
    Therefore by restricting to $s = 1$ and applying Lemma \ref{restrict-inv} again the proof is completed.
\end{proof}

\subsection{Comparison with the Isotopy Functor}\label{compare1}
    In this section we show that when the Lagrangian cobordism $L$ from $\Lambda_-$ to $\Lambda_+$ is induced by a Hamiltonian isotopy in Theorem \ref{GKS} \cite{GKS}, i.e.~$\Lambda_- = \Lambda$ and $\Lambda_+ = \varphi_H^1(\Lambda)$, then our Lagrangian cobordism functor agrees with the sheaf quantization functor given by the Hamiltonian isotopy\footnote{The author is very grateful to Vivek Shende, who essentially explains to the author the strategy of the proof that appears here.}. This section is not related to the rest of the paper, so the readers may feel free to skip it.

    Our strategy is similar to the proof of Theorem \ref{invariance} (Hamiltonian invariance), that is, to realize the Lagrangian cobordism as a functor
    $$Sh^b_{\Lambda \times I}(M \times I) \rightarrow Sh^b_{\Lambda_H}(M \times I)$$
    where $\Lambda_H$ is the Legendrian movie of $\Lambda$ under the Hamiltonian flow $\varphi_H^s,\,s \in I$. Then Theorem \ref{GKS} \cite{GKS} will enable us to conclude that $\Phi_{L} \simeq \Psi_H$ at $M \times \{1\}$ from the initial condition at $M \times \{0\}$.

    First, recall the construction of Lagrangian cobordisms induced by a Hamiltonian isotopy \cite{Chantraine} (see \cite[Section 4.2.3]{GroEliashGF} for a different construction). Suppose the contact Hamiltonian is $H: T^{*,\infty}M \rightarrow \mathbb{R}$. Then consider the homogeneous symplectic Hamiltonian to be $\widehat H(x, \xi) = |\xi|H(x, \xi/|\xi|): T^*M \rightarrow \mathbb{R}$. Let $\eta: [0,+\infty) \rightarrow [0, 1]$ be a cut-off function such that $\eta(r) = 0$ when $r$ is small, and $\eta(r) = 1$ when $r$ is large. Then the Lagrangian cobordism induced by $H$ is
    $$L = \varphi_{\eta(|\xi|) \widehat H(x, \xi)}^1(\Lambda \times \mathbb{R}_{>0}).$$
    One can see that $L$ coincides with $\Lambda \times \mathbb{R}_{>0}$ when $|\xi|$ is small, and coincides with $\varphi_H^1(\Lambda) \times \mathbb{R}_{>0}$ when $|\xi|$ is large.

    Now we try to construct a Lagrangian cobordism $\overline{L}$ from $\Lambda \times I$ to $\Lambda_H$, so that the restriction to $T^*M \times \{0\}$ is just $\Lambda \times \mathbb{R}_{>0}$, and the restriction to $T^*M \times \{1\}$ is $L$. Let
    $$\overline{\varphi}_{\overline{H}}^t: \, T^{*,\infty}(M \times I) \rightarrow T^{*,\infty}(M \times I); \,\, (x, \xi, s, \sigma) \mapsto (\varphi_H^{st}(x, \xi), s, \sigma - sH \circ \varphi_H^{st}(x, \xi)).$$
    Then the Lagrangian cobordism $\overline{L}$ induced by $\overline{\varphi}_{\overline{H}}^t,\,t \in I$, will satisfy our conditions.

    Recall that to define the Lagrangian cobordism functor, we need to consider a proper embedding $e: T^*M \hookrightarrow T^{*,\infty}(M \times \mathbb{R})$. For example, consider a Riemannian metric $g$, let $\varphi_g^t$ be the geodesic flow, and define
    $$e(x, \xi) = (\varphi_g^{-1}(x, \xi), |\xi|_g^2/2, 1).$$
    Then we are going to work with the (singular) Legendrians $(M \cup \Lambda \times \mathbb{R}_{>0})_\epsilon^\prec$ and $(M \cup \varphi_H^1(\Lambda) \times \mathbb{R}_{>0})_\epsilon^\prec \subset T^{*,\infty}_{\tau > 0}(M \times \mathbb{R})$.

    Let $\mathscr{F} \in \mu Sh^b_{M \cup \Lambda \times \mathbb{R}_{>0}}(M \cup \Lambda \times \mathbb{R}_{>0})$. Let $\varphi_{\eta \widehat H}^s$ be the Hamiltonian flow on $T^*M$ that extends to $T^{*,\infty}_{\tau > 0}(M \times \mathbb{R})$. Then by Theorem \ref{cont-trans}, there is a canonical sheaf $\Psi_{\eta \widehat H}(\mathscr{F}) \in \mu Sh^b_{M \cup \varphi_{\eta \widehat H}^1(\Lambda \times \mathbb{R}_{>0})}(M \cup \varphi_{\eta \widehat H}^1(\Lambda \times \mathbb{R}_{>0}))$ whose restriction to $M \cup \Lambda \times \mathbb{R}_{>0}$ is $\mathscr{F}$, this means $\Psi_{\eta \widehat H}(\mathscr{F})$ is the unique lifting of $\mathscr{F}$ under the (restriction) functor
    \[\begin{split}
    \mu Sh^b_{M \cup \varphi_{\eta \widehat H}^1(\Lambda \times \mathbb{R}_{>0})} & \xrightarrow{\sim} \mu Sh^b_{M \cup \Lambda \times \mathbb{R}_{>0}}(M \cup \Lambda \times \mathbb{R}_{>0}) \times_{Loc^b(\Lambda)} Loc^b(L) \\
    & \xrightarrow{\sim} \mu Sh^b_{M \cup \Lambda \times \mathbb{R}_{>0}}(M \cup \Lambda \times \mathbb{R}_{>0}).
    \end{split}\]
    In other words, by abusing notations, we can write
    $$\Phi_{L}(\mathscr{F}) = \Phi_{L}(\Psi_{\eta \widehat H}(\mathscr{F})).$$

\begin{lemma}\label{restrict-compare}
    Let $\overline{L}$ be the Lagrangian cobordism from $\Lambda \times I$ to $\Lambda_H$ induced by $\overline{\varphi}_{\overline{H}}^s$, $i_s: T^{*,\infty}(M \times \mathbb{R}) \times \{s\} \hookrightarrow T^{*,\infty}(M \times \mathbb{R} \times I)$ be the inclusion and $\pi: T^{*,\infty}(M \times \mathbb{R} \times I) \rightarrow T^{*,\infty}(M \times \mathbb{R})$ be the projection. Then for any $\mathscr{F} \in \mu Sh^b_{M \cup \Lambda \times \mathbb{R}_{>0}}(M \cup \Lambda \times \mathbb{R}_{>0})$,
    $$i_s^{-1}\Phi_{\overline{L}}(\pi^{-1}(\mathscr{F})) = \Phi_{L_s}(\mathscr{F}),$$
    where $L_s = \varphi_{s\eta \widehat H}^1(\Lambda \times \mathbb{R}_{>0})$ is the Lagrangian cobordism induced by $\varphi_{sH}^t$.
\end{lemma}
\begin{proof}
    First of all, $\varphi_{\overline{Z}}^z$ be the Liouville flow on $T^{*,\infty}(M \times \mathbb{R} \times I)$ defined by the pullback (homogeneous) Hamiltonian $H_{\overline{Z}} = H_Z \circ \pi_{T^*(M \times \mathbb{R})}$. Let $\Psi_{\overline{Z}}^\zeta$ be the equivalence defined by the Liouville flow $\varphi_{\overline{Z}}^z,\,z\in (-\infty,0]$, or $\phi_{\overline{Z}}^\zeta,\,\zeta \in (0,1]$, on $T^{*,\infty}(M \times \mathbb{R} \times I)$, and
    \[\begin{array}{c}
    i_Z: M \times \mathbb{R} \times \{0\} \hookrightarrow  M \times \mathbb{R} \times [0,1],
    j_Z: M \times \mathbb{R} \times (0,1] \hookrightarrow M \times \mathbb{R} \times [0,1], \\
    \overline{i}_{\overline{Z}}: M \times \mathbb{R} \times I \times \{0\} \hookrightarrow M \times \mathbb{R} \times I \times [0,1],
    \overline{j}_{\overline{Z}}: M \times \mathbb{R} \times I \times (0,1] \hookrightarrow M \times \mathbb{R} \times I \times [0,1].
    \end{array}\]
    Write $(\Psi_{\overline{Z}}^\zeta(\Psi_{\eta\widehat H}^0(\mathscr{F})))_\text{dbl} \in Sh^b(M \times \mathbb{R} \times I)$ for the image of $\Psi_{\overline{Z}}^\zeta(\Psi_{\eta\widehat H}^0(\mathscr{F}))$ under the antimicrolocalization functor in Theorem \ref{antimicro} \cite{NadShen}. Then by Remark \ref{GKSviaCont},
    $$i_s^{-1}\big(\Psi_{\overline{Z}}^\zeta(\Psi_{\eta\widehat H}^0(\mathscr{F}))\big)_\text{dbl} = \Psi_{\overline{Z}}^\zeta(i_s^{-1}\Psi_{\eta\widehat H}^0(\mathscr{F}))_\text{dbl}.$$

    We can write down the Lagrangian cobordism functor as a series of compositions
    $$\Phi_{\overline{L}}(\pi^{-1}(\mathscr{F}))_\text{dbl} = \overline{i}_{\overline{Z}}^{-1}\overline{j}_{\overline{Z},*}\big({\Psi}_{\overline{Z}}^\zeta {\Psi}_{\eta H}^0(\pi^{-1}(\mathscr{F}))\big)_\text{dbl}.$$
    Note that ${\Psi}_{\overline{Z}}^\zeta$ is the equivalence functor defined by the Liouville flow on $T^{*,\infty}(M \times \mathbb{R} \times I)$. Then by Lemma \ref{basechange} there is a natural morphism
    \[\begin{split}
    i_s^{-1}\Phi_{\overline{L}}(\pi^{-1}(\mathscr{F}))_\text{dbl} &\xrightarrow{\sim} i_s^{-1}\overline{i}_{\overline{Z}}^{-1}\overline{j}_{\overline{Z},*}\big({\Psi}_{\overline{Z}}^\zeta {\Psi}_{\eta\widehat H}^0(\pi^{-1}(\mathscr{F}))\big)_\text{dbl} \\
    &\xrightarrow{\sim} i^{-1}_Zj_{Z,*}\Psi_{Z}^\zeta\big(i_s^{-1}{\Psi}_{\eta\widehat H}^0(\pi^{-1}(\mathscr{F}))\big)_\text{dbl} \\
    &\xrightarrow{\sim} i_Z^{-1}j_{Z,*}\big(\Psi_{Z}^\zeta\Psi_{s\eta\widehat H}^0(\mathscr{F})\big)_\text{dbl} \xrightarrow{\sim} \Phi_{L_s}(\mathscr{F})_\text{dbl}.
    \end{split}\]
    and thus we complete the proof.
\end{proof}

\begin{proof}[Proof of Theorem \ref{comparegks}]
    Consider the Lagrangian cobordism $\overline{L}$ induced by $\overline{\varphi}_{\overline{H}}^t$. By Lemma \ref{restrict-compare}, we know that for $i_0: T^{*,\infty}(M \times \mathbb{R}) \times \{0\} \hookrightarrow T^{*,\infty}(M \times \mathbb{R} \times I)$ and $\pi: T^{*,\infty}(M \times \mathbb{R} \times I) \rightarrow T^{*,\infty}(M \times \mathbb{R})$,
    $$i_0^{-1}\Phi_{\overline{L}}(\mathscr{F}) = \Phi_{\Lambda \times \mathbb{R}_{>0}}(\mathscr{F}) = \mathscr{F}.$$
    By Theorem \ref{cont-trans} and Remark \ref{GKSviaCont}, $i_0^{-1}: \mu Sh^b_{((M\cup \Lambda \times \mathbb{R}_{>0}) \times I) \cup \overline{L}}(((M\cup \Lambda \times \mathbb{R}_{>0}) \times I) \cup \overline{L}) \rightarrow \mu Sh^b_{M \cup \Lambda \times \mathbb{R}_{>0}}(M \cup \Lambda \times \mathbb{R}_{>0})$ is an equivalence and its inverse is the Hamiltonian isotopy functor $\Psi_H^0$ in Theorem \ref{GKS} \cite{GKS}. Therefore
    $$\Phi_{\overline{L}}(\pi^{-1}(\mathscr{F})) = {\Psi}_{H}^0(\mathscr{F}).$$
    Finally, by restricting to $M \times \{1\}$ and apply Lemma \ref{restrict-compare} again, we can conclude that $\Phi_{L}(\mathscr{F}) = \Psi_H(\mathscr{F}).$
\end{proof}

\subsection{Comparison with the Filling Functor}\label{compare2}
    When $\Lambda_- = \varnothing$, $L$ is a Lagrangian filling of $\Lambda_+$. In this section we basically show that for costandard Lagrangian branes, our fully faithful functor
    $$\Phi_L: \, Loc^b(L) \hookrightarrow Sh^b_{\Lambda_+}(M)$$
    coincides with the functor Jin-Treumann constructed \cite{JinTreu}. Again, the reader may skip this section.

    Fix an embedding $e: T^*M \hookrightarrow T^{*,\infty}(M \times \mathbb{R})$. For example, consider a Riemannian metric $g$, let $\varphi_g^t$ be the geodesic flow, and define
    $$e(x, \xi) = (\varphi_g^{-1}(x, \xi), |\xi|_g^2/2, 1).$$
    Then $M \cup \Lambda \times \mathbb{R}_{>0} \subset T^{*,\infty}(M \times \mathbb{R})$ is a (singular) Legendrian.

    Let $U \subset M$ be an open subset with subanalytic boundary $\partial U$. The outward conormal of $U$ is denoted by $\nu^*_{U,+}M$. Then the Lagrangian skeleton $M \cup \nu^*_{U,+}M$ is shown in Figure \ref{jintreumann}.

\begin{definition}
    Let $m_U: \overline{U} \rightarrow [0, +\infty)$ be the defining function of $\partial U$ such that $m_U^{-1}(0) = \partial U$. Let $f_U = -\ln(m_U)$. Then the graph of the exact 1-form $L = L_U = L_{df_U} \subset T^*M$ is the costandard Lagrangian brane associated to $U$.
\end{definition}

    When $L$ intersect the ideal contact boundary \cite{IdealBdy} of $T^*M$ at $\nu^{*,\infty}_{U,+}M$ such that it is tangent to $\nu^*_{U,+}M$ to infinite order (for example, when $0$ is a regular value of $m_U$), it can be viewed as a Lagrangian filling of $\nu^{*,\infty}_{U,+}M$, equipped with a different primitive $f_U'$ that is bounded on $L = L_U$. Via the embedding $e$, its image $\widetilde{L}$ will be a Legendrian in $T^{*,\infty}(M \times \mathbb{R})$ that coincides with $\nu^*_{U,+}M$ at infinity. %Indeed, one can consider
%    $$L_U' = \{(x, (1-e^{-|\xi|})\xi/|\xi|) | (x, \xi) \in L_U\} \cup \Lambda \times [1,+\infty).$$
%    Then $L_U'$ is Hamiltonian isotopic to $L_U$, and the primitive $f_U'(x, \xi) = $

\begin{figure}
  \centering
  \includegraphics[width=1.0\textwidth]{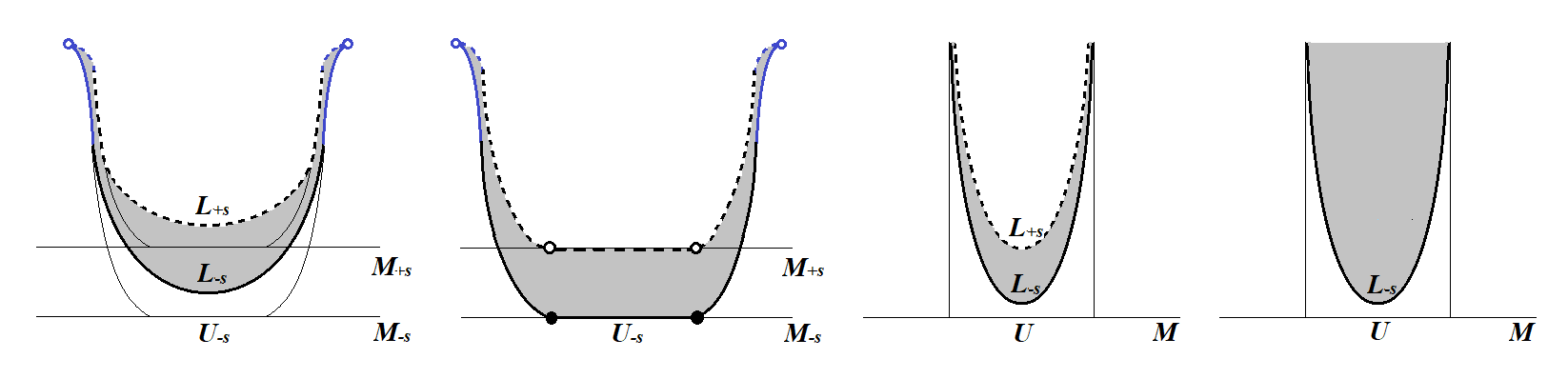}\\
  \caption{The Nadler-Shende construction (left) and the Jin-Treumann construction (right). The grey regions are the supports of the corresponding sheaves. The thin lines on the left are the skeleton $M \cup \nu^*_{U,+}M$ embedded in $J^1(M)$, and the thick lines there are the two copies of Lagrangian fillings. The blue lines are the family of cusps $\partial \Lambda \,\times \prec$.}\label{jintreumann}
\end{figure}

\begin{proof}[Proof of Proposition \ref{comparejt}]
    Consider a complex of local systems $\mathscr{F}_L$ on $L$ with stalk $F$. Note that the projection $\pi_L: L \hookrightarrow T^*M \rightarrow M$ defines a diffeomorphism $L \cong U$. Write $\mathscr{F}_U = \pi_{L,*}\mathscr{F}_L$. We will show that both functors send $\mathscr{F}_L$ to $\mathscr{F}_U$.

    (1)~We first compute $\Phi_L: Loc^b(L) \rightarrow Sh^b(M)$. Let the vertical vector field $\partial/\partial t$ be the Reeb vector field. Consider the skeleton $M \cup \nu^*_{U,-}M$ and its positive/negative Reeb pushoff $(M \cup \nu^*_{U,-}M)_{\pm \epsilon}$. Lift $L$ to a Legendrian $\widetilde{L}$ that coincides with $M \cup \nu^*_{U,-}M$ when the radius coordinate $r = \ln|\xi|$ in $T^*M$ is sufficiently large. When $r$ is large, we cut off the Legendrian $(\nu^{*,\infty}_{U,-}M)_{\pm \epsilon}$ and connect them by a family of cusps $\nu^{*,\infty}_{U,-}M \,\times \prec$, and also cut off $\tilde{L}_{\pm \epsilon}$ and connect them by a family of cusps $\nu^{*,\infty}_{U,-}M \,\times \prec$. We consider
    $$Loc^b(L) \xrightarrow{\sim} \mu Sh^b_{\widetilde{L}}(\widetilde{L}) \hookrightarrow Sh^b_{(\widetilde{L}, \partial \widetilde{L})^\prec_\epsilon}(M \times \mathbb{R})_0.$$
    Here the subscript $0$ means the subcategory of sheaves with $0$ stalk outside a compact set. Hence there is a sheaf $\mathscr{F}_\text{dbl}$ with singular support in $(\widetilde{L}, \partial \widetilde{L})^\prec_\epsilon$ whose microlocalization along $\widetilde{L}_{-\epsilon}$ is given by $\mathscr{F}_L$, given by the antimicrolocalization functor Theorem \ref{antimicro} \cite{NadShen}.

    Running the Liouville flow $\varphi_Z^z$ for $z\in(-\infty, 0]$ or $\phi_Z^\zeta$ for $\zeta \in (0,1]$, we can get a sheaf on $M \times \mathbb{R} \times (0, 1]$. Note that the end $(\nu^{*,\infty}_{U,-}M)_{\pm \epsilon}$ (which coincides with $\partial \widetilde{L}_{\pm \epsilon}$) is preserved by Liouville flow up to a Reeb translation (due to change of the primitive $f_U'$ of $L_U$), and the limit points
    $$\lim_{z\rightarrow -\infty}\varphi_Z^z(\widetilde{L}, \partial \widetilde{L})^\prec_\epsilon = \lim_{\zeta\rightarrow 0}\phi_Z^\zeta(\widetilde{L}, \partial \widetilde{L})^\prec_\epsilon \subset T^{*,\infty}_{\tau>0}(M \times \mathbb{R}) \times \{0\}$$
    are exactly $(U \cup \nu_{U,-}^*M)_{\epsilon}^\prec$. The resulting sheaf is therefore in $Sh^b_{(\overline{U} \cup \nu_{U,-}^*M)^\prec_\epsilon}(M \times \mathbb{R})$.

    Now we apply the microlocalization functor
    $$Sh^b_{(\overline{U} \cup \nu_{U,-}^*M)^\prec_\epsilon}(M \times \mathbb{R})_0 \rightarrow \mu Sh^b_{(\overline{U} \cup \nu_{U,-}^*M)_{-\epsilon}}((\overline{U} \cup \nu_{U,-}^*M)_{-\epsilon}) \xrightarrow{\sim} Sh^b_{\nu^*_{U,-}M}(M)_0.$$
    The microstalks for points in $\overline{U}_{-\epsilon}$ are $F$, and those for points in $M_{-\epsilon} \backslash \overline{U}_{-\epsilon}$ are $0$. The microlocal monodromy along $U$ is determined by $\mathscr{F}_U = \pi_{L,*}\mathscr{F}_L$ because topologically taking the limit $\lim_{z\rightarrow -\infty}\varphi_Z^z(L)$ under the Liouville flow gives a homotopy equivalence $L \simeq \lim_{z\rightarrow -\infty}\varphi_Z^z(L) \simeq U \cup \nu^*_{U,-}M \simeq U$ that is homotopic to the projection $\pi_L: L \xrightarrow{\sim} U$.

    (2)~Then we consider $\Psi_L^{JT}: Loc^b(L) \rightarrow Sh^b(M)$. In \cite{JinTreu} they considered the Legendrian lift $\widetilde{L}$ of $L$ whose front projection onto $M \times \mathbb{R}$ is the graph of the function $f_U$. Then consider the positive/negative Reeb pushoff $\widetilde{L}_{\pm \epsilon}$, which are the graphs of the functions $f_U \pm \epsilon$. We have \cite{JinTreu}
    $$Loc^b(L) \xrightarrow{\sim} \mu Sh^b_{\widetilde{L}}(\widetilde{L}) \hookrightarrow Sh^b_{\widetilde{L}_{\pm \epsilon}}(M \times \mathbb{R})_0.$$
    Then there is a sheaf $\mathscr{F}_\text{dbl}'$ with singular support in $\widetilde{L}_{-\epsilon} \cup \widetilde{L}_{\epsilon}$, given by the antimicrolocalization functor \cite[Section 3.15]{JinTreu}, whose microlocalization along $\widetilde{L}_{-\epsilon}$ gives the local system $\mathscr{F}_L$. Write $D_{\pm \epsilon} = \{(x, t) | t = f_U(x) \pm \epsilon\}$. Indeed the sheaf is supported in the region
    $$D_{[-\epsilon,\epsilon)} = \{(x, t) | f_U(x)-\epsilon \leq t < f_U(x)+\epsilon\}$$
    with stalk $F$. This is because the sheaf has zero stalk below $D_{-\epsilon} = \{(x, t) | t = f_U(x)-\epsilon\}$ and hence for sufficiently small $\epsilon' > \epsilon$ (as in Example \ref{microstalk-cone})
    $$\mathscr{F}_L = m_{\widetilde{L}_{-\epsilon}}(\mathscr{F}_\text{dbl}') \simeq \mathrm{Tot}(\mathscr{F}_\text{dbl}'|_{D_{-\epsilon}} \rightarrow \mathscr{F}_\text{dbl}'|_{D_{-\epsilon'}}) \simeq \mathscr{F}_\text{dbl}'|_{D_{-\epsilon}}.$$
    In addition, the monodromy of the local system in the region $D_{[-\epsilon,\epsilon)}$ bounded by $\pi(\widetilde{L}_{-\epsilon})$ and $\pi(\widetilde{L}_\epsilon)$ is also determined by $\mathscr{F}_L$, since for $\pi_M: M \times \mathbb{R} \rightarrow M$,
    $$\mathscr{F}_\text{dbl}'|_{D_{[-\epsilon,\epsilon)}} = \pi_M^{-1}(\mathscr{F}_\text{dbl}'|_{D_{-\epsilon}})|_{D_{[-\epsilon,\epsilon)}}.$$

    Now we consider a Legendrian isotopy to move $\widetilde{L}_{\epsilon}$ along the positive Reeb direction. Jin-Treumann showed that \cite[Section 3.18]{JinTreu} for $S > T > 0$ sufficiently large we have
    \[\xymatrix@C=4em{
    Sh_{\widetilde{L}_{-\epsilon} \cup \widetilde{L}_{\epsilon+S}}(M \times \mathbb{R}) \ar[r]^{j_{M \times (-\infty,T)}^{-1}} & Sh_{\widetilde{L}_{-\epsilon}}(M \times (-\infty, T)) & Sh_{\widetilde{L}_{-\epsilon}}(M \times \mathbb{R}) \ar[l]_{\hspace{20pt}\sim},
    }\]
    and hence one can get a sheaf $\mathscr{F}'$ in $Sh_{\widetilde{L}_{-\epsilon}}(M \times \mathbb{R})$ with stalk $F$ supported in the region $D_{[-\epsilon,+\infty)} = \{(x, t) | t \geq f_U(x)-\epsilon\}$ above $D_{-\epsilon}$, and the monodromy in this region determined by $\mathscr{F}_L$ since
    $$\mathscr{F}'|_{D_{[-\epsilon,+\infty)}} = \pi_M^{-1}(\mathscr{F}'|_{D_{-\epsilon}})|_{D_{[-\epsilon,+\infty)}}.$$

    Finally we push forward the sheaf to $Sh^b_{\nu^*_{U,-}M}(M)_0$ via the projection $\pi_M: M \times \mathbb{R} \rightarrow M$. Note that in the fiber of the projection $\{x\} \times \mathbb{R}$ ($x\in U$), the sheaf is $F_{r \geq f_U(x)}$, and $\pi_{x,*}(F_{r \geq f_U(x)}) = F$. Hence the projection will give a sheaf supported on $U$ with stalk $F$. In addition we claim that the monodromy defines the local system $\mathscr{F}_U = \pi_{L,*}\mathscr{F}_L$ on $\overline{U}$ because the projection of $L$ onto $M$ via $\widetilde{L} \hookrightarrow T^{*,\infty}_{\tau>0}(M \times \mathbb{R}) \rightarrow M \times \mathbb{R} \rightarrow M$ coincides with the projection $\pi_L: L \hookrightarrow T^*M \rightarrow M$ which gives the diffeomorphism $L \cong U$.

    Hence $\Phi_{L} \simeq \Psi_{L}^{JT}: Loc^b(L) \rightarrow Sh^b(M)$ when $L = L_U$ is a standard Lagrangian brane corresponding to the open subset $U \subset M$.
\end{proof}

%\begin{remark}
%    The computation also works for the standard Lagrangian fillings of $\nu_{U,-}^{*,\infty}M$ (the inward conormal bundle), though that of Jin-Treumann's functor is a little bit more difficult \cite[Example 3.4]{JinTreu}.
%\end{remark}

\section{Examples and Applications}

    We now focus on some concrete examples of Legendrian submanifolds and Lagrangian cobordisms and explain what the Lagrangian cobordism functor on sheaves will tell us.

\subsection{Examples of cobordism functors}\label{elementary}
    We consider the elementary Lagrangian cobordism given by attaching a Lagrangian $k$-handle ($0 \leq k \leq n$). The local model of the front projection in $\mathbb{R}^{n+1}$ is shown in Figure \ref{1handleattach} and \ref{2handleattach}.

    The front projection of $\Lambda_\pm$ gives a stratification on $\mathbb{R}^{n+1}$, such that on each stratum the sheaf is locally constant. Hence we are able to get a combinatoric model given by stalks on each stratum and the transition maps given by the microlocal Morse lemma as in Example \ref{combin-model} and \ref{microstalk-cone}. We will explain how objects behave under the cobordism functor.

    For the stratification given by $\Lambda_\pm$, denote by $V_\pm \subset \mathbb{R}^{n+1}$ the domain whose $x_{n+1}$-coordinate is bounded by the front projection of $\Lambda_\pm$ and $U_\pm \subset \mathbb{R}^{n+1}$ the domain whose $x_{n+1}$-coordinate is unbounded on each vertical slice $\{(x_1, \dots, x_n)\} \times \mathbb{R}$ (see Figure \ref{1handleattach} and \ref{2handleattach}). For a sheaf in $Sh^b_{\Lambda_-}(\mathbb{R}^{n+1})$, suppose the stalk in the region $V_-$ is $B$ and the stalk in $U_-$ is $A$ (Figure \ref{1handlesheaf}). Suppose the microstalk of $\mathscr{F}$ is
    $$F \simeq \mathrm{Tot}(A \rightarrow B).$$

\begin{figure}
  \centering
  \includegraphics[width=0.75\textwidth]{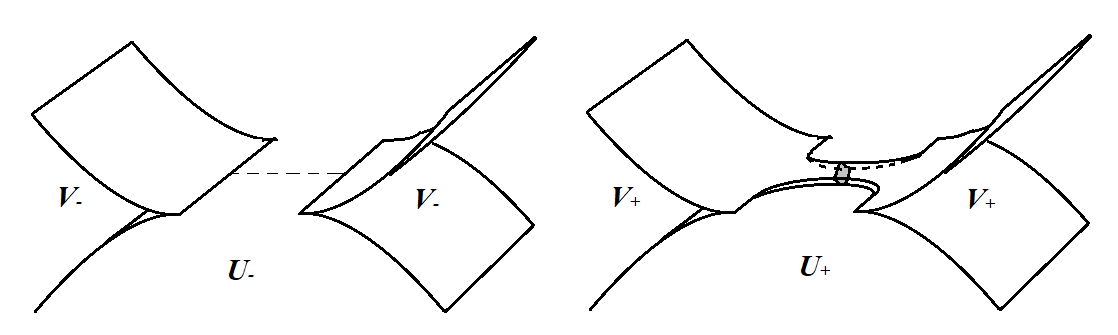}\\
  \caption{On the left is the front projection of $\Lambda_-$, and on the right is the front projection of $\Lambda_+$ after attaching a Lagrangian 1-handle connecting the two cusps along the dashed line, where in the middle of the tube (the grey slice) there is a unique Reeb chord.}\label{1handleattach}
\end{figure}
\begin{figure}
  \centering
  \includegraphics[width=0.7\textwidth]{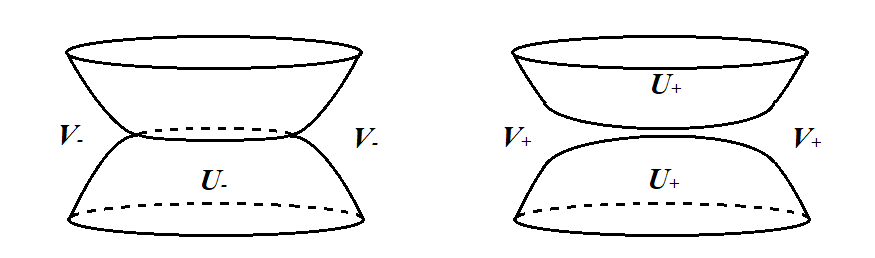}\\
  \caption{On the left is the front projection of $\Lambda_-$, and on the right is the front projection of $\Lambda_+$ after attaching a Lagrangian 2-handle connecting the $S^1$-family of cusps along the disk.}\label{2handleattach}
\end{figure}

    Instead of doing concrete computations, we will describe objects under the Lagrangian cobordism functor in three steps by only using a few properties of our functor:
\begin{enumerate}
  \item determine the stalks in $U_+$ and $V_+$ using the fact that the cobordism functor is identity outside a compact set in $\mathbb{R}^{n+1}$ and hence the stalks are preserved;
  \item determine the microlocalization along $\Lambda_+$ (relative to boundary), which is a local system with stalk $F$, using the fact that the Liouville flow fixes the end $\Lambda_+$ and hence the cobordism functor preserves the microlocalization;
  \item determine the local system in $V_+$ using the fact that $B \simeq A \oplus F$, and hence the local system with stalk $F$ on $\Lambda_+$ determines a local system with stalk $F$ on $V_+$ and a local system with stalk $B$ on $V_+$ (relative to boundary at infinity in $\mathbb{R}^{n+1}$).
\end{enumerate}
    The information above will uniquely determine the sheaf.

    Before starting to determine the sheaf $\mathscr{F}^+ \in Sh^b_{\Lambda_+}(\mathbb{R}^{n+1})$, we first note that $\mathscr{F}^- \in Sh^b_{\Lambda_-}(\mathbb{R}^{n+1})$ has an image in $Sh^b_{\Lambda_+}(\mathbb{R}^{n+1})$ via the cobordism functor iff it can be lifted into
    $$Sh^b_{\Lambda_-}(\mathbb{R}^{n+1}) \times_{Loc^b(\Lambda_-)} Loc^b(L).$$
    Since $\Lambda_- \cong S^{k-1} \times D^{n-k+1}$ while $L \cong D^k \times D^{n-k+1}$, this is the same as saying that the microlocalization $m_{\Lambda_-}(\mathscr{F}^-)$ can be trivialized over $S^{k-1}$.

\begin{remark}
    Note that not all complexes of local systems in $Loc^b(S^k)\,(k \geq 2)$ are trivial. For example for the Hopf fibration $\pi: S^3 \rightarrow S^2$, $R\pi_*\Bbbk_{S^3}$ is a nontrivial complex of local system on $S^2$. The reason is that although $H^1(S^k) = 0$, $H^k(S^k) \neq 0$ and that will give extension classes in $Ext^1(\Bbbk_{S^k}[1-k], \Bbbk_{S^k})$.
\end{remark}

    Here is how the sheaf $\mathscr{F}^+$ is determined. (1)~Firstly, note that away from the cusps, the Lagrangian cobordism is just a standard embedded cylinder $\Lambda_0 \times \mathbb{R}$, and hence is fixed by the Liouville flow. The functor
    $$\mu Sh^b_{\mathbb{R}^{n+1} \cup \Lambda_0 \times \mathbb{R} \cup L}(\mathbb{R}^{n+1} \cup \Lambda_0 \times \mathbb{R} \cup L) \rightarrow \mu Sh^b_{\mathbb{R}^{n+1} \cup \Lambda_0 \times \mathbb{R}}(\mathbb{R}^{n+1} \cup \Lambda_0 \times \mathbb{R})$$
    is the identity. This shows that the sheaf should remain the same away from compact subsets in $\mathbb{R}^{n+1}$. Then one can see explicitly that the stalks of $\mathscr{F}_+$ are determined by $\mathscr{F}_-$, where the stalk in the region $V_+$ must be $B$ and the stalk in $U_+$ must be $A$.

    (2)~Secondly, note that the complex of local systems $m_{\Lambda_-}(\mathscr{F}_-)$ on $\Lambda_-$ has stalk $F$. After gluing with a local system $\mathscr{L}_L$ on $L$, by restriction we can determine a complex of local systems on $\Lambda_+ \times \{+\infty\}$. Note that the restriction of the local system along $\partial L = \partial \Lambda_\pm \times \mathbb{R}$ is determined by the microlocalization on $\partial \Lambda_-$.

    Since $\Lambda_+ \times \{+\infty\}$ is preserved by the negative time Liouville flow $\varphi_Z^z$ up to a Reeb translation (due to the change of the primitive $f_L$ of $L$), the functor
    \[\begin{split}
    \mu Sh&{}^b_{M \cup \Lambda_- \times \mathbb{R}_{>0} \cup L}(M \cup \Lambda_- \times \mathbb{R}_{>0} \cup L) \\
    &\, \rightarrow \mu Sh^b_{\lim_{z \rightarrow -\infty}\varphi_Z^z(M \cup \Lambda_- \times \mathbb{R}_{>0} \cup L)}(\lim{}_{z \rightarrow -\infty}\varphi_Z^z(M \cup \Lambda_- \times \mathbb{R}_{>0} \cup L))\\
    &\, \rightarrow \mu Sh^b_{M \cup \Lambda_+ \times \mathbb{R}_{>0}}(M \cup \Lambda_+ \times \mathbb{R}_{>0})
    \end{split}\]
    is an equivalence on $\Lambda_+ \times \{+\infty\}$ induced by the Reeb translation (Theorem \ref{cont-trans}). Hence the complex of local systems on $\Lambda_+ \times \{+\infty\}$ is preserved by the nearby cycle functor. Therefore, after applying $\Phi_L$, the microstalk on $\Lambda_+$ is still $F$, where the microlocal monodromy is still the same as the restriction of the local system $\mathscr{L}_L$ onto $\Lambda_+$.

    Note that the restriction of the local system to boundary $\mathscr{L}_L|_{\partial \Lambda_\pm \times \mathbb{R}}$ is the pull back of the given local system $m_{\partial\Lambda_-}(\mathscr{F}_-)$. Therefore, after applying the cobordism functor we get the microlocalization in the fiber of $Loc^b(\Lambda_+) \rightarrow Loc^b(\partial \Lambda_+)$ at the point  $m_{\partial\Lambda_-}(\mathscr{F}_-)$.

\begin{figure}
  \centering
  \includegraphics[width=0.75\textwidth]{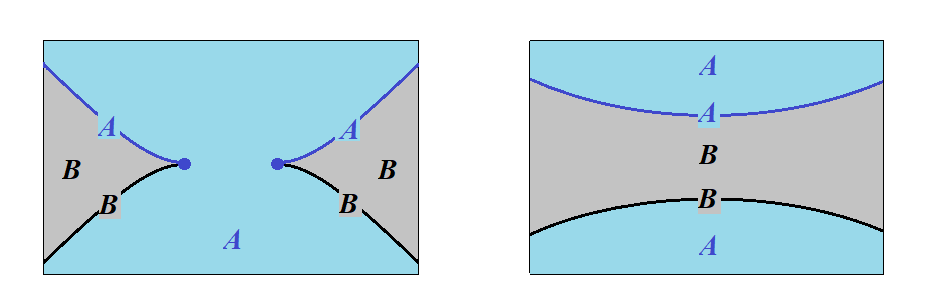}\\
  \caption{The microlocal sheaf on $Sh^b_{\Lambda_-}(\mathbb{R}^{n+1})$ (left) and $Sh^b_{\Lambda_+}(\mathbb{R}^{n+1})$ (right) before and after the Lagrangian 1-handle attachment. Here we assume $\Lambda_\pm \subset T^{*,\infty}_{\tau>0}(\mathbb{R}^n \times \mathbb{R}_\tau)$.}\label{1handlesheaf}
\end{figure}

    (3)~Finally, we determine the local system in the region $V_+$. Note that $V_+$ is not contractible relative to boundary at infinity $\partial V_+ = S^{k-1} \times D^{n-k+1}$. In particular, globally there could be nontrivial monodromy coming from our choice of the local monodromy relative to boundary, parametrized by the fiber of $Loc^b(V_+) \rightarrow Loc^b(\partial V_+)$. Because there are transition maps
    $$A \rightarrow B \rightarrow A$$
%    \[\xymatrix{
%    A \ar[dr] \ar[rr]^{\sim} & & A \\
%    & B \ar[ur] &
%    }\]
    whose composition is a quasi-isomorphism. Without loss of generality, we assume that it is the identity \cite[Corollary 3.18]{STZ}. Then there is a splitting of chain complexes
    $$B \simeq A \oplus \mathrm{Tot}(A \rightarrow B) \simeq A \oplus F.$$
    Therefore since the microlocal monodromy along $\Lambda_+$ has been determined by the local system on $L$ we chose, so is the monodromy of the sheaf in $V_+$ if we identify $\mathscr{F}_+|_{V_+}$ with $A_{V_+} \oplus \mathscr{L}_{V_+}$, where $A_{V_+}$ is just the constant local system and $\mathscr{L}_{V_+}$ is a local system on $V_+$ with stalk $F$ that extends $\mathscr{L}_L|_{\Lambda_+}$.

    In fact topologically $(V_+, \partial V_+) \simeq (L, \Lambda_-) \simeq (D^k \times D^{n-k+1}, S^{k-1} \times D^{n-k+1})$ by considering the projection map $L \hookrightarrow \mathbb{R}^{2n+1} \times \mathbb{R} \rightarrow \mathbb{R}^{2n+1} \rightarrow \mathbb{R}^{n+1}$. We claim that $\mathscr{L}_{V_+} \cong \mathscr{L}_L$ relative to the boundary $\partial V_+ \cong S^{k-1} \times D^{n-k+1} \cong \Lambda_-$, meaning that they live in the same fiber of the restriction functor. Indeed, the restriction of $\mathscr{L}_{V_+}$ and $\mathscr{L}_L$ to $\Lambda_+$ should both be $\mathscr{L}_L|_{\Lambda_+}$, but $\mathscr{L}_L|_{\Lambda_+}$ extends uniquely to $L$ since the inclusion $\Lambda_+ \hookrightarrow L$ is just $D^k \times S^{n-k} \hookrightarrow D^k \times D^{n-k+1}$. Therefore $\mathscr{L}_{V_+} \simeq \mathscr{L}_L$ (respectively, the restriction of $\mathscr{L}_{\partial V_+}$ and $\mathscr{L}_{\Lambda_-}$ to $\partial \Lambda_+ \cong \partial \Lambda_-$ agree, but $\mathscr{L}_L|_{\partial \Lambda_+}$ extends uniquely to $\partial V_+$, so the local systems live in the same fiber).

    Now we look at several different $k$-handle attachments to see what these data are in specific cases when $0 \leq k \leq 2$.

\subsubsection{Lagrangian $1$-handle attachment}
    When $k = 1$ there are 2 disconnected strata inside the cusps of $\Lambda_-$ (Figure \ref{1handleattach} and \ref{1handlesheaf}). The sheaf $\mathscr{F}_- \in Sh^b_{\Lambda_-}(\mathbb{R}^{n+1})$ can be extended only when the microlocal monodromy along $S^0 \times \mathbb{R}^{n} \subset \Lambda_-$ can be extended to a local system along $D^{1} \times \mathbb{R}^{n} \subset L$. This is equivalent to saying that the microstalks on two components $F \simeq F'$.

    Let the stalk in the region $V_-$ bounded by the 2 cusps be $B, B'$ and let the stalk outside be $A$. Then using the splitting of chain complexes
    $$B \simeq A \oplus \mathrm{Tot}(A \rightarrow B) \simeq A \oplus F \simeq A \oplus \mathrm{Tot}(A \rightarrow B') \simeq B',$$
    where $F = \mathrm{Tot}(A \rightarrow B) \simeq \mathrm{Tot}(A \rightarrow B')$ is the microstalk, we know that the condition implies that $B \simeq B'$. After applying the cobordism functor, the stalk in $V_+$ bounded by the front of $\Lambda_+$ is $B$ and the stalk outside is $A$.

    There is a choice we need to make for the quasi-isomorphism between all the $B$'s, and that is coming from our choice for the local system on $L$. Different identifications may give different monodromies along $\Lambda_+$ relative to the boundary at infinity $\partial L = S^0 \times D^{n}$.

    Namely, when gluing with a local system on $L$, we assign an extra quasi-isomorphism $f_F$ bewteen the stalks $F$ on both components of $\Lambda_-$. After applying $\Phi_L$, the microstalk on $\Lambda_+$ is still $F$, where the quasi-isomorphism from $F$ on the left to $F$ on the right is given by $f_F$. Then by the quasi-isomorphism
    $$B \simeq A \oplus F,$$
    the transition map of $B$ from left to right will be given by $f_B = (\mathrm{id}_A, f_F)$.

    In particular, if the microstalk $F \simeq \Bbbk^r$ ($\mathscr{F}_-$ is pure), then the choices are classified by $GL_r(\Bbbk)$. When $F \simeq \Bbbk$ ($\mathscr{F}_-$ is simple), then the choices are classified by $\Bbbk^\times$.

\begin{remark}
    One can compare our computation with the computation in \cite[Section 5]{AlgWeave} for Legendrian links and \cite[Section 5.5]{CasalsZas} for Legendrian surfaces, by decomposing those cobordisms into a composition of Reidemeister moves and a handle attachment.
\end{remark}

    What we described is only the local picture, globally there are different possibilities. Let's fix $F \simeq \Bbbk$ (this means $\mathscr{F}_-$ is simple). (1)~When the 1-handle $L$ connects two different components of $\Lambda_-$, then
    $$H^1(\Lambda_-; \Bbbk^\times) \cong H^1(L; \Bbbk^\times).$$
    Consider the moduli space of rank 1 local systems on $\Lambda$ (coming from the truncated derived moduli stacks of local systems) given by $Loc_1(\Lambda) = [H^1(\Lambda; \Bbbk^\times) / H^0(\Lambda; \Bbbk^\times)]$, and consider the framed moduli space of rank 1 local systems on a manifold $\Lambda$ given by $Loc_1^\textit{fr}(\Lambda) = H^1(\Lambda; \Bbbk^\times)$ with framing data, i.e.~fixed trivializations of stalks, at each component. Then
    $$Loc_1(\Lambda_-) \times [\Bbbk^\times/\Bbbk^\times] \cong Loc_1(L), \;\, Loc_1^\textit{fr}(\Lambda_-) \cong Loc_1^\textit{fr}(L).$$
    Consider the truncated derived moduli stack of microlocal rank 1 sheaves $t_0\mathbb{R}\mathcal{M}_1(\Lambda_\pm)$. Denote by $\mathcal{M}_1(\Lambda_\pm)$ the classical moduli stacks defined by the 1-rigid locus (the 1-rigid locus of $t_0\mathbb{R}\mathcal{M}_1(\Lambda_\pm)$ consisting objects with no negative self-extensions is always an Artin stack, but they may not coincide with the derived stack)\footnote{The flag moduli space considered in \cite{TZ,CasalsZas} is, strictly speaking, slightly different as they do not remember the trivial $\Bbbk^\times$-action by only taking quotients of flags by $PGL_n(\Bbbk)$ instead of $GL_n(\Bbbk)$. The moduli spaces they consider are equal to $\mathcal{M}_1(\Lambda)$ considered here after further taking quotients by the trivial $\Bbbk^\times$-action.} \cite{STWZ}*{Section 2.4}. Assuming that these classical moduli stacks coincide with the derived stacks, we have an embedding
    $$\mathcal{M}_1(\Lambda_-) \times [\Bbbk^\times/\Bbbk^\times] = \mathcal{M}_1(\Lambda_-) \times_{Loc_1(\Lambda_-)} Loc_1(L) \hookrightarrow \mathcal{M}_1(\Lambda_+).$$
    Consider $\mathcal{M}_1^\textit{fr}(\Lambda_\pm)$ the classical moduli stacks defined by the 1-rigid locus with framing data at each component of $\Lambda_\pm$. Then we have an embedding
    $$\mathcal{M}_1^\textit{fr}(\Lambda_-) = \mathcal{M}_1^\textit{fr}(\Lambda_-) \times_{Loc_1^{fr}(\Lambda_-; \Bbbk^\times)} Loc_1^\textit{fr}(L; \Bbbk^\times) \hookrightarrow \mathcal{M}_1^\textit{fr}(\Lambda_+).$$

    (2)~When the 1-handle $L$ is attached on one component of $\Lambda_-$, then the moduli spaces of rank 1 local systems satisfy
    $$H^1(\Lambda_-; \Bbbk^\times) \times \Bbbk^\times \cong H^1(L; \Bbbk^\times).$$
    Therefore, for the moduli spaces of rank 1 local systems we know that
    $$Loc_1(\Lambda_-) \times [\Bbbk^\times/\Bbbk^\times] \cong Loc_1(L), \;\, Loc_1^\textit{fr}(\Lambda_-) \times \Bbbk^\times \cong Loc_1^\textit{fr}(L).$$
    Hence assuming that the classical moduli stacks of microlocal rank 1 sheaves $\mathcal{M}_1(\Lambda_\pm)$ coincide with the derived stacks, we have an embedding
    $$\mathcal{M}_1(\Lambda_-) \times [\Bbbk^\times/\Bbbk^\times] = \mathcal{M}_1(\Lambda_-) \times_{Loc_1(\Lambda_-)} Loc_1(L) \hookrightarrow \mathcal{M}_1(\Lambda_+).$$
    For the moduli stacks of microlocal rank 1 sheaves with framing data at each component $\mathcal{M}_1^{fr}(\Lambda_\pm)$, we have an embedding
    $$\mathcal{M}_1^{fr}(\Lambda_-) \times \Bbbk^\times = \mathcal{M}_1^\textit{fr}(\Lambda_-) \times_{Loc_1^\textit{fr}(\Lambda_-; \Bbbk^\times)} Loc_1^\textit{fr}(L; \Bbbk^\times) \hookrightarrow \mathcal{M}_1^\textit{fr}(\Lambda_+).$$

\begin{remark}
    In \cite{GaoShenWeng} the authors considered augmentation varieties for Legendrian links of positive $n$-braid closures, and for any such 2 Legendrian links connected by a 1-handle cobordism they showed that
    $$Aug(\Lambda_-) \times \Bbbk^\times \hookrightarrow Aug(\Lambda_+).$$
    That is because when considering $Aug(\Lambda)$ they always fix $n$ marked points and do not change the number of marked points when the number of components increases/decreases. This should be thought of as equivalent to the moduli space of microlocal rank 1 sheaves together with framing data at $n$ base points \cite{STWZ}*{Section 2.4} or equivalently trivialization data of microstalks at $n$ base points.
\end{remark}

\subsubsection{Lagrangian $2$-handle attachment}
    When $k = 2$, the sheaf $\mathscr{F}_- \in Sh^b_{\Lambda_-}(\mathbb{R}^{n+1})$ can be extended only when the microlocal monodromy along $S^1 \times \mathbb{R}^{n-1} \subset \Lambda_-$ can be extended to a local system along $D^{2} \times \mathbb{R}^{n-1} \subset L$. As $C^*(D^2; \Bbbk) \cong \Bbbk$, this is equivalent to saying that the microlocal monodromy is trivial along $S^1 \times \mathbb{R}^{n-1} \subset \Lambda_-$.

    As in the case $k = 1$, there is a choice we need to take into consideration which is the contracting homotopy from the local system on $S^1$ to the trivial one, and the choice of the contracting homotopy will give different (higher) monodromies relative to the boundary at infinity $\partial L = S^1 \times D^{n-1}$. Consider a triangulation of $D^2 = \Delta^2$. Then this gives a stratification $D^2$. The 1-dimensional strata gives us quasi-isomorphisms
    $$f_{01}: F \rightarrow F, \,\, f_{12}: F \rightarrow F, \,\, f_{02}: F \rightarrow F.$$
    For the 2-dimensional stratum, we need to assign an extra chain homotopy $H_{012}$ from $f_{02}$ to $f_{12} \circ f_{01}$, i.e.~$H_{012}: F \rightarrow F[-1]$ such that
    $$H_{012}\delta_F - \delta_FH_{012} = f_{02} - f_{12} \circ f_{01}.$$

    The (higher) monodromy along $\Lambda_+$ is preserved by the functor $\Phi_L$ and hence determines the microlocal monodromy of $\mathscr{F}_+$ along $\Lambda_+$. Using the quasi-isomorphism
    $$B \simeq A \oplus F,$$
    the monodromy data of $F$ determines the monodromy data of the stalk $B$ in $\mathscr{F}_+$.

    When $F \simeq \Bbbk^r$ (the sheaf is pure), then the contracting homotopy data is trivial, and hence any such sheaf with trivial monodromy in $Sh^p_{\Lambda_-}(M)$ extends uniquely to a sheaf in $Sh^p_{\Lambda_+}(M)$.

    Suppose the classical moduli stacks of microlocal rank $r$ sheaves $\mathcal{M}_r(\Lambda_\pm)$ coincide with the derived stacks (with fixed framing data at a point). For $[\beta] \in \pi_1(\Lambda_-)$, let $\mathcal{M}_r^{[\beta]}(\Lambda_-)$ be the substack consisting of sheaves with trivial microlocal monodromy along $\beta$. Then for $L$ a Lagrangian $2$-handle cobordism attached along $\beta$, we have an embedding of algebraic stacks
    $$\mathcal{M}_r^{[\beta]}(\Lambda_-) \hookrightarrow \mathcal{M}_r(\Lambda_+).$$
    For the moduli stacks of microlocal rank $r$ sheaves with framing data at each component, we get a similar embedding.

\subsubsection{Lagrangian $k$-handle attachment ($k\geq 3$)}
    When $k \geq 3$, we need to choose higher homotopy data to ensure that the monodromy of the complex of  local systems along the attaching sphere $S^{k-1} \times D^{n-k+1} \subset \Lambda_-$ can be extended to $D^k \times D^{n-k+1} \subset L$. The monodromy along $\Lambda_+$ is preserved by the functor $\Phi_L$ and hence determines the monodromy of $\mathscr{F}_+$ along $\Lambda_+$. Using the quasi-isomorphism
    $$B \simeq A \oplus F,$$
    the (higher) monodromy data of $F$ determines the (higher) monodromy data of $B$ in $\mathscr{F}_+$.

    However, in the special case when $F \simeq \Bbbk^r$, there will be no nontrivial higher homotopy data, and since the attaching sphere is changed from $S^1$ to $S^{k-1}\,(k \geq 3)$, we know that any local system is trivial, so any such pure sheaf in $Sh^p_{\Lambda_-}(M)$ extends uniquely to a pure sheaf in $Sh^p_{\Lambda_+}(M)$.

    Suppose the classical moduli stacks of microlocal rank $r$ sheaves $\mathcal{M}_r(\Lambda_\pm)$ coincide with the derived stacks (with fixed framing data at a point). Then for $L$ a Lagrangian $k$-handle cobordism ($k \geq 3$), we have an embedding of algebraic stacks
    $$\mathcal{M}_r(\Lambda_-) \hookrightarrow \mathcal{M}_r(\Lambda_+).$$
    For the moduli stacks of microlocal rank $r$ sheaves with framing data at each component, we get a similar embedding.

\subsection{Applications to Legendrian surfaces}\label{application}
    In this section we use the computation of the number of microlocal rank 1 sheaves to prove the results Theorem \ref{2graphs}. We will frequently refer to \cite{TZ} and \cite{CasalsZas} for the theory of Legendrian weaves (which are certain type of Legendrian surfaces) and their moduli space of microlocal rank 1 sheaves.

%\begin{proof}[Proof of Theorem \ref{spherelink}]
%    Casals-Zaslow studied the moduli spaces of microlocal rank 1 sheaves on $\Lambda_n$, which, in this case, are flag moduli spaces. The moduli space is an orbifold and the point count over $\mathbb{F}_q$ is (\cite[Theorem 1.5]{CasalsZas})
%    $$|\mathcal{M}_1(\Lambda_n)(\mathbb{F}_q)| = \frac{q^{2n-3} - q^{n-2} + q^{n-1} + q - 1}{(q - 1)^2}.$$
%    Let $q = 2$. Then $|\mathcal{M}_1(\Lambda_n)(\mathbb{F}_2)| = 2^{2n-3} - 2^{n-2} + 2^{n-1} + 1$. Suppose there is a Lagrangian cobordism from $\Lambda_n$ to $\Lambda_m$. Then since any rank 1 local system over $\mathbb{F}_2$ is trivial, we have a fully faithful functor between microlocal rank 1 objects
%    $$Sh^\text{s}_{\Lambda_m}(S^2 \times \mathbb{R})_0 \hookrightarrow Sh^\text{s}_{\Lambda_n}(S^2 \times \mathbb{R})_0.$$
%    Here the subscript $0$ means the subcategory of objects with $0$ stalk near $-\infty$. Hence $|\mathcal{M}_1(\Lambda_m)(\mathbb{F}_2)| \leq |\mathcal{M}_1(\Lambda_n)(\mathbb{F}_2)|$. This means that the Lagrangian cobordism cannot exist for $n < m$.
%\end{proof}
%
    First, we recall that the correspondence between the front projection of Legendrian weaves and their planar graphs are illustrated in Figure \ref{weavedef}. The combinatoric constructions of Lagrangian handle attachments for Legendrian weaves are illustrated in Figure \ref{cubichandle}, and proved in \cite[Theorem 4.10]{CasalsZas}.

\begin{figure}
  \centering
  \includegraphics[width=0.8\textwidth]{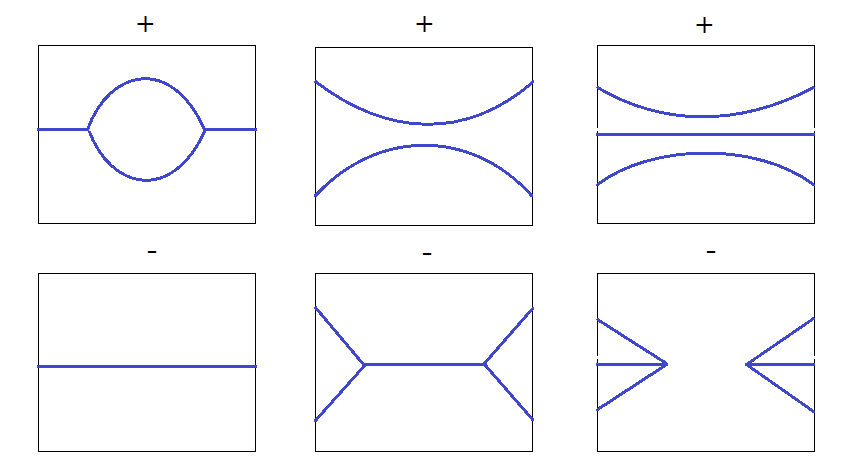}\\
  \caption{The graph on the left is a Lagrangian 1-handle attachment in Legendrian weaves; in the middle is a Lagrangian 2-handle attachment in Legendrian weaves; on the right is a Legendrian connected sum cobordism. $\Lambda_+$ are on the top while $\Lambda_-$ are on the bottom.}\label{cubichandle}
\end{figure}

\begin{figure}
  \centering
  \includegraphics[width=1.0\textwidth]{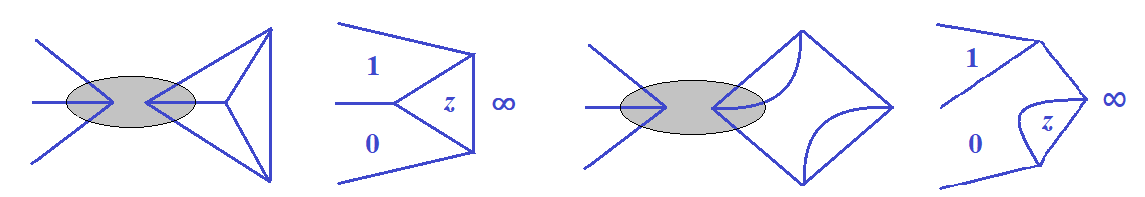}\\
  \caption{Taking connected sum with $\Lambda_\text{Cliff}$ (left) and with $\Lambda_\text{Unknot}$ (right). The cobordisms are from left to right in each picture. The labelling $0, 1, \infty, z$ is a $\Bbbk P^1$ coloring (so that regions sharing a common edge have different colors), which determines a microlocal rank 1 sheaf.}\label{connectsum}
\end{figure}

\begin{proof}[Proof of Theorem \ref{2graphs}]
    (1)~We start from $\Lambda_{g,k}$. Consider the local picture near a degree 3 vertex of the graph. Consider a Lagrangian 1-handle cobordism in the shadowed region (Figure \ref{2graphcobordism} left). This will give a cobordism from $\Lambda_{g,k}$ to $\Lambda_{g+1,k}$. Then consider a Lagrangian 2-handle cobordism in the shadowed region (Figure \ref{2graphcobordism} middle). This gives a cobordism from $\Lambda_{g+1,k}$ to $\Lambda_{g,k}$. For general $\Lambda_{g,k}, \Lambda_{g',k'}$, the cobordism can be constructed by concatenation.

\begin{figure}
  \centering
  \includegraphics[width=0.8\textwidth]{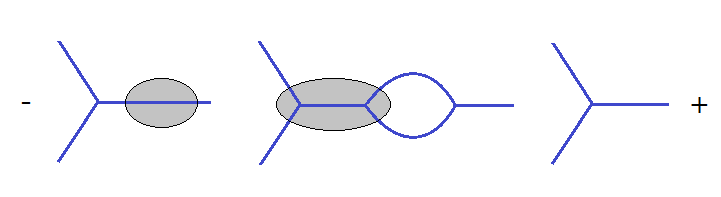}\\
  \caption{The cobordism from $\Lambda_{g,k}$ to $\Lambda_{g+1,k}$ to $\Lambda_{g,k}$ (from left to right). The grey regions are the places where we attach Lagrangian handles.}\label{2graphcobordism}
\end{figure}

    (2)~This is essentially proved by Dimitroglou Rizell \cite{Rizchordexample}. First of all, notice that $\Lambda_{g,0}$ admits an exact Lagrangian filling by taking a sequence of Lagrangian 1-handle cobordisms (Figure \ref{unknotfilling}). Next, we claim that for any $k \geq 1$, $\Lambda_{g',k}$ does not admit exact Lagrangian fillings. Assuming the claim, then clearly there cannot be exact Lagrangian cobordisms from $\Lambda_{g,0}$ to $\Lambda_{g',k}$.

\begin{figure}
  \centering
  \includegraphics[width=0.8\textwidth]{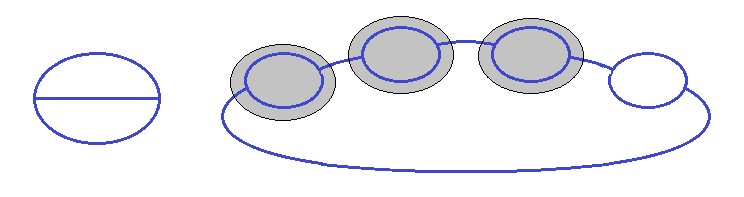}\\
  \caption{The Lagrangian filling of the Legendrian surface $\Lambda_{g,0}$ by Lagrangian 1-handle cobordisms in all the shadowed regions and finally fill the unknot on the left by a Lagrangian disk.}\label{unknotfilling}
\end{figure}

    We now prove the claim using sheaves. One way is to apply \cite[Theorem 1.3]{TZ}. An alternative approach is the following \cite[Theorem 5.12]{CasalsZas}. First, we know that the flag moduli spaces in \cite{TZ,CasalsZas} as spaces of flags modulo $PGL_2(\Bbbk)$-actions are identified with the framed moduli space of sheaves $\mathcal{M}_1^\textit{fr}(\Lambda)$ with framing data at a single point since
    $$\mathcal{M}_1(\Lambda) = [\mathcal{M}_1^\textit{fr}(\Lambda) / \Bbbk^\times]$$
    where the moduli stack $\mathcal{M}_1(\Lambda)$ is equivalent to the space of flags modulo $GL_2(\Bbbk)$-action. When $k \geq 1$, one can consider locally a triangle in the graph. A microlocal rank 1 sheaf is characterized by a $\Bbbk P^1$ colorings of regions (such that any regions sharing a common edge have different colors). Without loss of generality, one can assume that outside the triangle, the 3 regions are colored by $0, 1$ and $\infty$ (Figure \ref{connectsum}). Then
    the possible choices for colors in the triangle region are $\Bbbk^\times \backslash \{1\}$, i.e.
    $$\mathcal{M}_1^\textit{fr}(\Lambda_{g',k}) = \mathcal{M}_1^\textit{fr}(\Lambda_{g'-1,k-1}) \times (\Bbbk^\times \backslash \{1\}).$$
    When $\Bbbk = \mathbb{Z}/2\mathbb{Z}$, then there are no available choices and hence there are no microlocal rank 1 sheaves with $\mathbb{Z}/2\mathbb{Z}$-coefficients on $\Lambda_{g',k}$. Hence there cannot be any Lagrangian fillings. The claim is proved.

    (3)~First we should note that as explained in \cite[Example 4.26]{CasalsZas} there is a Lagrangian cobordism $L_0$ from $\Lambda_\text{Cliff}$ to a loose Legendrian 2-sphere $\Lambda_{S^2,\text{loose}}$ by a Lagrangian 2-handle attachment (Figure \ref{cliffloose}), where the fact that $\Lambda_{S^2,\text{loose}}$ is loose follows from \cite[Proposition 4.24]{CasalsZas}. Hence there is a Lagrangian cobordism from $\Lambda_{g,k} = \Lambda_{g-1,k-1} \# \Lambda_\text{Cliff}$ to a genus $g-1$ surface $\Lambda_{g-1,k-1} \# \Lambda_{S^2,\text{loose}}$, and $\Lambda_{g-1,k-1} \# \Lambda_{S^2,\text{loose}}$ is also loose.

\begin{figure}
  \centering
  \includegraphics[width=0.6\textwidth]{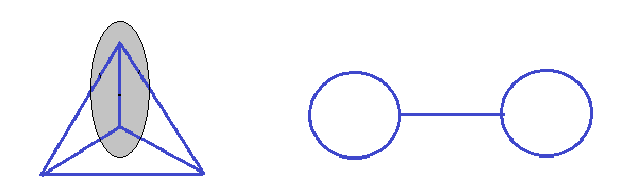}\\
  \caption{A Lagrangian 2-handle cobordism from $\Lambda_{S^2,\text{loose}}$ (right) to $\Lambda_\text{Cliff}$ (left).}\label{cliffloose}
\end{figure}

    We now apply \cite[Theorem 2.2]{LagCap}\footnote{The author thanks Emmy Murphy for pointing out that the Lagrangian cap construction helps build cobordisms in this setting.}. First we need to construct a formal Lagrangian cobordism, that is, a smooth cobordism $\psi: L_1 \hookrightarrow \mathbb{R}^6$ from $\Lambda_{g-1,k-1} \# \Lambda_{S^2,\text{loose}}$ to $\Lambda_{g,k'}$, and a family of bundle maps $\Psi_s: TL_1 \rightarrow T\mathbb{R}^6|_{L_1}$ such that $\Psi_0 = d\psi$, $\Psi_s \equiv d\psi$ near positive and negative ends, and $\Psi_1$ is a Lagrangian bundle map.

    Note that $\Lambda_{g,k'}$ and $\Lambda_{g,k-1}$ are formally Legendrian isotopic for any $k \geq 1, k' \geq 0$. This means that there is a smooth isotopy $\psi'_t: \Lambda_t \hookrightarrow \mathbb{R}^5,\,t \in I$, together with a family of bundle maps $\Psi'_{s,t},\,s, t\in I$, such that $\Psi'_{s,0} = d\psi'_0$, $\Psi'_{s,1} = d\psi'_1$, $\Psi'_{0,t} = d\psi'_t$, and $\Psi'_{1,t}$ are Lagrangian bundle maps into the contact distribution. Given a formal Legendrian isotopy, we can consider a smooth cobordism $L = \Lambda \times I$ from $\Lambda_0$ to $\Lambda_1$ being
    $$\psi: L \hookrightarrow \mathbb{R}^5 \times I,\,\, \psi(x, t) = (\psi'_t(x), t),$$
    and a family of bundle maps $\Psi_s: TL \rightarrow T(\mathbb{R}^5 \times I)|_{L}$ such that $\Psi_0 = d\psi$ and $\Psi_1$ is a Lagrangian bundle map by considering the homotopy such that
    $$\Psi_s|_{T\Lambda_t} = \Psi_{s,t}, \,\, \Psi_s|_{\left<\partial/\partial t\right>} = (1-s)d\psi(\partial/\partial t) + s\,\partial/\partial r.$$
    Therefore, we can get a formal Lagrangian concordance from $\Lambda_{g,k-1}$ to $\Lambda_{g,k'}$. By part~(1) we know that there is a genuine Lagrangian cobordism from $\Lambda_{g-1,0}$ to $\Lambda_{g,k-1}$. Thus by concatenation, we will get a formal Lagrangian cobordism $(L_1, \psi, \Psi_s)$ from $\Lambda_{g-1,k-1} \# \Lambda_{S^2,\text{loose}}$ to $\Lambda_{g,k'}$, and in fact
    $$H^1(L_1) \xrightarrow{\sim} H^1(\Lambda_{g-1,k-1} \# \Lambda_{S^2,\text{loose}}).$$

    Then by \cite[Theorem 2.2]{LagCap} there is a Lagrangian cobordism $L_1$ from $\Lambda_{g-1,k-1} \# \Lambda_{S^2,\text{loose}}$ to $\Lambda_{g,k'}$ such that
    $$H^1(\Lambda_{g-1,k-1} \# \Lambda_{S^2,\text{loose}}) = \Bbbk^{2g-2}, \quad H^1(L_1) = \Bbbk^{2g-1}.$$
    Taking the concatenation of $L_0$ and $L_1$, we will get a Lagrangian cobordism such that
    $$\dim \mathrm{coker}(H^1(L_0 \cup L_1) \rightarrow H^1(\Lambda_{g,k})) = 1.$$

    (4)~We show that there cannot be Lagrangian cobordisms $L$ with vanishing Maslov class from $\Lambda_{g,k}$ to $\Lambda_{g,k'}$ for $k < k'$ such that $H^1(L) \twoheadrightarrow H^1(\Lambda_{g,k})$. Indeed consider a degree 3 vertex in the graph of $\Lambda_{g-1,k-1}$. Taking connected sum with $\Lambda_\text{Cliff}$ and $\Lambda_\text{Unknot}$ will give $\Lambda_{g,k}$ and $\Lambda_{g,k-1}$. As explained in Part~(2), a microlocal rank 1 sheaf is characterized by the number $\Bbbk P^1$ colorings of the graph (Figure \ref{connectsum}). The possible choices for colors in the triangle region are $\Bbbk^\times \backslash \{1\}$,
    $$\mathcal{M}_1^\textit{fr}(\Lambda_{g,k}) = \mathcal{M}_1^\textit{fr}(\Lambda_{g-1,k-1}) \times (\Bbbk^\times \backslash \{1\}).$$
    On the other hand, for $\Lambda_{g,k-1}$, assume the upper half region and lower half region are colored $0, \infty$ (Figure \ref{connectsum}). Then the possible choices for colors in the bi-gon region are $\Bbbk^\times$, i.e.
    $$\mathcal{M}_1^\textit{fr}(\Lambda_{g,k-1}) = \mathcal{M}_1^\textit{fr}(\Lambda_{g-1,k-1}) \times \Bbbk^\times.$$
    In particular, $|\mathcal{M}_1^\textit{fr}(\Lambda_{g,k})(\mathbb{F}_q)| < |\mathcal{M}_1^\textit{fr}(\Lambda_{g,k-1})(\mathbb{F}_q)|$, and by induction $|\mathcal{M}_1^\textit{fr}(\Lambda_{g,k'})(\mathbb{F}_q)| < |\mathcal{M}_1^\textit{fr}(\Lambda_{g,k})(\mathbb{F}_q)|$.

    When $H^1(L) \twoheadrightarrow H^1(\Lambda_{g,k})$ is surjective, in the fiber product
    \[\xymatrix{
    \mathcal{M}_1^\textit{fr}(\Lambda_{g,k-1}) \times_{H^1(\Lambda_{g,k}; \Bbbk^\times)} H^1(L; \Bbbk^\times) \ar[r]^{\hspace{60pt}\widehat r} \ar[d] & \mathcal{M}_1^\textit{fr}(\Lambda_{g,k}) \ar[d]^P \\
    H^1(L; \Bbbk^\times) \ar[r]^r & H^1(\Lambda_{g,k}; \Bbbk^\times),
    }\]
    the horizontal map $r$ at the bottom is a projection map, and hence so is $\widehat r$ (in fact the vertical map on the right $P$ is called the period map in \cite[Section 4.7]{TZ}). Therefore
    $$|(\mathcal{M}_1^\textit{fr}(\Lambda_{g,k}) \times_{H^1(\Lambda_{g,k}; \Bbbk^\times)} H^1(L; \Bbbk^\times))(\mathbb{F}_q)| \geq |\mathcal{M}_1^\textit{fr}(\Lambda_{g,k})(\mathbb{F}_q)| > |\mathcal{M}_1^\textit{fr}(\Lambda_{g,k'})(\mathbb{F}_q)|.$$
    However, a fully faithful Lagrangian cobordism functor $\Phi_L$ should induce an embedding
    $$\mathcal{M}_1^\textit{fr}(\Lambda_{g,k}) \times_{H^1(\Lambda_{g,k}; \Bbbk^\times)} H^1(L; \Bbbk^\times) \hookrightarrow \mathcal{M}_1^\textit{fr}(\Lambda_{g,k'}).$$
    That is a contradiction.
\end{proof}

\begin{remark}
    For the Lagrangian cobordism from $\Lambda_{g,k}$ to $\Lambda_{g+1,k}$ and to $\Lambda_{g,k}$, one can see (in Figure \ref{2graphcobordism}) that the ascending manifold of Lagrangian 1-handle and the descending manifold of the Lagrangian 2-handle have geometric intersection number 1. Since these Lagrangians are regular \cite{FlexLag}, one can cancel the pair of critical points so that the Lagrangian is Hamiltonian isotopic to a cylinder.
\end{remark}

\bibliography{legendriansheaf}
\bibliographystyle{amsplain}

\end{document}